\documentclass[10pt]{amsart}

\usepackage{graphics}
\usepackage{amsfonts}
\usepackage{amsmath}
\usepackage{amssymb}

\usepackage[colorlinks=true, linkcolor=blue, citecolor=magenta]{hyperref}

\numberwithin{equation}{section}
\usepackage{pst-all}
\usepackage{float}
\usepackage{subfigure}

\def\RR{{\mathbb R}}
\def\MM{{\mathbb M}}

\def\NN{{\mathbb N}}
\def\OO{{\mathbb O}}

\def\SO{{\mathbb S\mathbb O}}
\def\SS{{\mathbb S}}
\def\Z{\mathcal{Z}}
\def\R{\mathcal{R}}
\def\X{Y}
\def\M{X}

\def\cal{\mathcal}
\def\det{{\rm det}}
\def\Det{{\rm Det}}

\def\Aff{{\rm Aff}}
\def\domain{D(f)}
\def\func{\hat f}
\def\exp{{\rm e}}

\def\wto{\rightharpoonup}
\def\dto{{\scriptstyle\buildrel{\scriptstyle\longrightarrow}
\over{{}_{\scriptstyle\longrightarrow}}}}
\def\eps{\varepsilon}
\def\then{\Rightarrow}
\def\inff{\mathop{\inf\limits_{a\in\RR^N}}}
\def\infff{\mathop{\inff\limits_{b\in\RR^m}}}
\def\cupp{\mathop{\cup}}
\def\capp{\mathop{\cap}}
\def\rank{{\rm rang}}
\def\ell{l}
\def\diag{{\rm diag}} 
\def\AffET{\Aff^{\rm reg}}
\def\AffETli{\AffET_{\rm li}}

\newtheorem{theorem}{Th\'eor\`eme}[section]
\newtheorem{lemma}[theorem]{Lemme}
\newtheorem{proposition}[theorem]{Proposition}
\newtheorem{corollary}[theorem]{Corollaire}
\newtheorem{question}[theorem]{Question}

\theoremstyle{definition}
\newtheorem{definition}[theorem]{D\'efinition}

\theoremstyle{remark}
\newtheorem{remark}[theorem]{Remarque}

\newcommand\trans{\pitchfork}

\begin{document}

\title{Relaxation et passage 3D-2D avec contraintes de type d\'eterminant}

\author[Omar Anza Hafsa \& Jean-Philippe Mandallena]{Omar Anza Hafsa \& Jean-Philippe Mandallena}

\address{UNIVERSITE MONTPELLIER II, UMR-CNRS 5508, Place Eug\`ene Bataillon, 34095 Montpellier, France.}

\email{Omar.Anza-Hafsa@univ-montp2.fr}

\address{UNIVERSITE DE NIMES, Site des Carmes, Place Gabriel P\'eri, 30021 N\^\i mes, France.}

\email{jean-philippe.mandallena@unimes.fr}

\keywords{Calcul des variations, hyper\'elasticit\'e, relaxation, passage 3D-2D, $\Gamma$-convergence, \'energie de membrane non lin\'eaire, contrainte d\'eterminant strictement positif, contraintes de type d\'eterminant.\\
\indent{\em Key words.} Calculus of variations, hyperelasticity, relaxation, 3D-2D passage, $\Gamma$ convergence, nonlinear membrane energy, constraint determinant positive, determinant type constraints.}

\date{Janvier 2009 - January 2009}

\begin{abstract}
L'objectif de cet article est, d'une part, de rendre accessible un ensemble d'outils et m\'ethodes pour traiter  la relaxation et le passage 3D-2D avec contraintes de type d\'eterminant, et, d'autre part, de donner une d\'emonstration compl\`ete du passage 3D-2D par $\Gamma$-convergence sous la contrainte d\'eterminant strictement positif.

\smallskip

\begin{center}
\large \sc Relaxation and 3D-2D passage with determinant type constraints
\end{center}

\noindent{\sc Abstract.} The goal of this paper is, on the one hand, to make available a set of tools and methods to deal with relaxation and 3D-2D passage with determinant type constraints, and, on the other hand, to give a complete proof of the 3D-2D passage by $\Gamma$-convergence under the constraint determinant positive.
\end{abstract}

\maketitle

\tableofcontents

\section{Introduction et pr\'eliminaires}

Une mod\'elisation math\'ematique r\'ealiste de l'hyper\'elasticit\'e doit prendre en compte {\em la non interp\'en\'etrabilit\'e de la mati\`ere} et {\em la n\'ecessit\'e d'une \'energie infinie pour compresser un volume de mati\`ere en un point}. Ainsi la densit\'e d'\'energie $W:\MM^{m\times N}\to[0,+\infty]$ (avec $m=N=3$ dans le cadre de l'hyper\'elasticit\'e) associ\'ee \`a un mat\'eriau hyper\'elastique  doit  satisfaire les deux contraintes suivantes :
\begin{eqnarray}
&& W(F)=+\infty \hbox{ si et seulement si }\det F\leq 0\ ;\label{Cst1}\\
&& W(F)\to+\infty\hbox{ lorsque }\det F\to 0,\label{Cst2}
\end{eqnarray}
o\`u $\MM^{m\times N}$ est l'espace des matrices \`a $m$ lignes et $N$ colonnes et $\det F$ est le d\'eterminant de $F$ (lorsque $m=N$). Cet article se concentre sur le d\'eveloppement d'outils et m\'ethodes pour traiter la relaxation et le passage 3D-2D en pr\'esence des contraintes (\ref{Cst1}) et (\ref{Cst2}) dans le cadre du calcul des variations dit ``moderne", i.e., l'\'etude de la minimisation des fonctionnelles int\'egrales du point de vue de la m\'ethode directe du calcul des variations (voir \S 1.1). Bien que ces contraintes apparaissent en hyper\'elasticit\'e peu de choses sont connues sur le calcul des variations avec contraintes de type d\'eterminant. En effet, la plupart des r\'esultats concernent le cas o\`u $W$ satisfait la condition de croissance d'ordre $p>1$ suivante 
\[
W(F)\leq c(1+|F|^p)\hbox{ pour tout }F\in\MM^{m\times N}\hbox{ avec }c>0,
\]
ce qui est incompatible avec (\ref{Cst1}) et (\ref{Cst2}) (voir \S 1.2-1.4). En fait, il est possible de d\'epasser le cas \`a croissance d'ordre $p$ en d\'egageant le concept plus ``souple" d'int\'egrande $p$-ample, i.e., $W$ est $p$-ample si 
\[
\Z W(F)\leq c(1+|F|^p)\hbox{ pour tout }F\in\MM^{m\times N}\hbox{ avec }c>0
\]
(avec $\Z W$ d\'efinie en \eqref{DefOfZW}). Celui-ci permet de traiter la relaxation et le passage 3D-2D sous la contrainte d\'eterminant non nul (voir \S 2.1-2.4 et \S 3.1-3.2), i.e., $W$ satisfait \eqref{Cst1prime} et \eqref{Cst2}, aussi bien que le passage 3D-2D sous la contrainte d\'eterminant strictement positif (voir \S 3.1 et \S 3.3), i.e., $W$ satisfait \eqref{Cst1} et \eqref{Cst2}, qui, \'etant plus difficile \`a traiter, n\'ecessite l'utilisation de deux th\'eor\`emes d'approximation ainsi qu'un th\'eor\`eme de permutation de l'infimum et de l'int\'egrale (voir \S 4-5).

\medskip

L'objectif principal de cet article est de fournir une d\'emonstration compl\`ete du th\'eor\`eme de passage 3D-2D par $\Gamma$-convergence sous la contrainte d\'eterminant strictement positif que nous avons mis au point dans \cite{oah-jpm08b} (voir le th\'eor\`eme \ref{ThMemb2}). Dans un souci de synth\`ese, il nous a paru n\'ecessaire de donner une vue d\'etaill\'ee et unifi\'ee, avec des am\'eliorations, de nos travaux \cite{oah-jpm06,oah-jpm07,oah-jpm08b,oah-jpm08a}.  On retrouvera donc  des r\'esultats d\'ej\`a d\'emontr\'es, mais aussi d'autres qui \'etaient seulement \'enonc\'es et difficiles d'acc\`es : un de nos objectifs \'etant de rendre accessible un corpus d'outils et m\'ethodes pour traiter la relaxation et le passage 3D-2D en pr\'esence de  contraintes de type d\'eterminant.  Enfin, nous avons essay\'e autant que possible, tout au long de cet article, d'\^etre pr\'ecis sur les multiples contributions (Percivale 1991, Le Dret-Raoult 1993, Ben Belgacem 1996, Ben Belgacem-Bennequin 1996) qui ont amen\'e \`a la r\'esolution compl\`ete du  probl\`eme du passage 3D-2D par $\Gamma$-convergence sous les contraintes \eqref{Cst1} et \eqref{Cst2}.

\subsection{G\'en\'eralit\'es sur le calcul des variations}
 
De fa\c con g\'en\'erale le calcul des variations vise \`a minimiser des fonctionnelles int\'egrales du type 
$$
E(\phi)=\int_{\Omega}L(x,\phi(x),\nabla \phi(x))dx,
$$
o\`u $\Omega\subset\RR^N$ est un ouvert born\'e et $L:\Omega\times\RR^N\times\MM^{m\times N}\to\RR\cup\{+\infty\}$ est l'int\'egrande, lorsque $\phi:\Omega\to\RR^m$ satisfait une condition au bord du type $\phi=\phi_0$ sur $\partial\Omega$ et d\'ecrit un espace de Banach $X$ (r\'eflexif) sur lequel $E$ et la condition au bord ont un sens. Pour traiter ce type de probl\`eme de minimisation on utilise {\em la m\'ethode directe du calcul des variations} qui consiste \`a v\'erifier les trois points suivants :
\begin{enumerate}
\item[$\diamond$] $\mathcal{A}:=\{\phi\in X:\phi=\phi_0\hbox{ sur }\partial\Omega\}$ est non vide et $E$ est minor\'ee sur $\mathcal{A}$, assurant ainsi que $-\infty<\alpha:=\inf\{E(\phi):\phi\in\mathcal{A}\}<+\infty$\ ;
\item[$\diamond$] il existe une suite $\{\bar\phi_n\}_{n\geq 1}$ minimisante pour $E$ dans $\mathcal{A}$, i.e., $\{\bar\phi_n\}_{n\geq 1}\subset\mathcal{A}$ et $\lim_{n\to+\infty}E(\bar\phi_n)=\alpha$, telle que, \`a une sous-suite pr\`es, $\bar \phi_n\wto \bar\phi$ avec $\bar\phi\in\mathcal{A}$ et $\bar\phi=\phi_0$ sur $\partial\Omega$ (o\`u ``$\wto$" d\'esigne  la convergence faible dans $X$)\ ;
\item[$\diamond$] $E$ est s\'equentiellement faiblement semi-continue inf\'erieurement (sfsci) dans $X$, i.e., $E(\phi)\leq\liminf_{n\to+\infty}E(\phi_n)$
pour tout $\phi_n\wto \phi$ dans $X$. 
\end{enumerate}
 On a alors $\bar \phi\in\mathcal{A}$ et $E(\bar\phi)\leq \lim_{n\to+\infty}E(\bar \phi_n)=\alpha$, i.e., $\bar \phi$ est un minimiseur de $E$ sur $\mathcal{A}$. 

Ici,  $X=W^{1,p}(\Omega;\RR^m)$ avec $p>1$ et $L(x,s,F)=W(F)-\langle f(x),s\rangle$ o\`u $\langle\cdot,\cdot\rangle$ est le produit scalaire dans $\RR^N$, $f\in L^q(\Omega;\RR^m)$ avec ${1/p}+{1/q}=1$ et $W:\MM^{m\times N}\to[0,+\infty]$ est Borel mesurable et coercive, i.e., il existe $C>0$ tel que $W(F)\geq C|F|^p$ pour chaque $F\in\MM^{m\times N}$. Avec ce cadre de travail, les deux premiers points se v\'erifient assez facilement. Le point difficile \`a v\'erifier est  le troisi\`eme, i.e., montrer que $E:W^{1,p}(\Omega;\RR^m)\to\RR\cup\{+\infty\}$ d\'efinie par $E:=I-J$ avec $I:W^{1,p}(\Omega;\RR^m)\to[0,+\infty]$ et $J:W^{1,p}(\Omega;\RR^m)\to\RR$ donn\'ees respectivement par 
\begin{equation}\label{DefsOfFuncts}
 I(\phi):=\int_{\Omega}W(\nabla \phi(x))dx\hbox{ et }J(\phi):=\int_{\Omega}\langle f(x),\phi(x)\rangle dx
\end{equation}
est sfsci dans $W^{1,p}(\Omega;\RR^m)$. Comme $J$ est (une forme) lin\'eaire et fortement continue dans $W^{1,p}(\Omega;\RR^m)$ tout revient \`a \'etudier la semi-continuit\'e inf\'erieure faible de $I$.  En pr\'esence des conditions (\ref{Cst1}) et (\ref{Cst2}) (et sans hypoth\`ese de polyconvexit\'e) ce probl\`eme de semi-continuit\'e inf\'erieure n'est pas encore r\'esolu (voir \cite{ball98,ball02}). Dans \cite{oah-jpm07,oah-jpm08a} cette question a \'et\'e abord\'e sous l'angle de la relaxation par la mise au point d'un th\'eor\`eme g\'en\'eral utilisable lorsque $W$ satisfait (\ref{Cst2}) et 
\begin{eqnarray}\label{Cst1prime}
 W(F)=+\infty\hbox{ si et seulement si }\det F=0.
\end{eqnarray}
(On peut noter que la condition (\ref{Cst1prime}) est ``moins m\'ecanique" que (\ref{Cst1}) car bien qu'elle interdise ``la compression d'un volume de mati\`ere en un point", elle n'emp\^eche pas ``l'interp\'en\'etration de la mati\`ere".)  Le passage 3D-2D en pr\'esence des conditions (\ref{Cst1prime}) et (\ref{Cst2}) a \'et\'e trait\'e dans \cite{oah-jpm06}. Utilisant la m\'ethode d\'evelopp\'ee dans \cite{oah-jpm06} et les travaux de Ben Belgacem (voir \cite{benbelgacem96,benbelgacem97,benbelgacem00}), dans \cite{oah-jpm08b} nous avons \'etudi\'e le  passage 3D-2D sous les contraintes plus r\'ealistes (\ref{Cst1}) et (\ref{Cst2}).

Avant de d\'ecrire en d\'etail ces r\'esultats (voir \S 2 pour \cite{oah-jpm07,oah-jpm08a} et voir \S 3 pour \cite{oah-jpm06,oah-jpm08b}) nous donnons une vue succincte de quelques notions (quasiconvexit\'e et semi-continuit\'e inf\'erieure (voir \S 1.2), relaxation (voir \S 1.3), passage 3D-2D (voir \S 1.4)) du calcul des variations.

\subsection{Quasiconvexit\'e et semi-continuit\'e inf\'erieure}

Le concept de quasiconvexit\'e a \'et\'e introduit par Morrey en 1952 (voir \cite{morrey52}).  
\begin{definition}\label{Def_Quasi} 
On dit que $W$ est quasiconvexe si  
 $$
 W(F)\leq \int_Y W(F+\nabla\varphi(x))dx
 $$
 pour tout $F\in\MM^{m\times N}$ et tout $\varphi\in W^{1,\infty}_0(Y;\RR^m)$ avec $Y:=]0,1[^N$. 
\end{definition}
 Malgr\'e les apparences, la d\'efinition \ref{Def_Quasi} est ind\'ependante de $Y$ (voir \cite[Proposition 2.3]{ball-murat84}, voir aussi la proposition \ref{FonsecaProperties}(a)). La quasiconvexit\'e n'est pas identique \`a la convexit\'e. Plus pr\'ecis\'ement, si $W$ est convexe, i.e., 
\begin{equation}\label{CondConvexity}
W(tF+(1-t)F^\prime)\leq tW(F)+(1-t)W(F^\prime)
\end{equation}
pour tous $F,F^\prime\in\MM^{m\times N}$ et tout $t\in]0,1[$, alors $W$ est quasiconvexe. La r\'eciproque est vraie dans le cas scalaire, i.e., $\min\{m,N\}=1$, mais fausse dans le cas vectoriel, i.e., $\min\{m,N\}>1$, (par exemple, si $m=N>1$, l'application $F\mapsto|\det F|$ est quasiconvexe mais n'est pas convexe). La quasiconvexit\'e est aussi reli\'ee \`a d'autres notions de convexit\'e : la polyconvexit\'e et la rang-1 convexit\'e. On dit que $W$ est polyconvexe si $W(F)=g({\rm T}(F))$ o\`u $g$ est une certaine fonction convexe et ${\rm T}(F)$ est le vecteur form\'e par toutes les matrices des ``sous-d\'eterminants" de $F$. Par exemple, si $m=N=3$, $W$ est polyconvexe si et seulement si il existe une fonction convexe $g:\MM^{3\times 3}\times\MM^{3\times 3}\times\RR\to[0,+\infty]$ telle que 
$
W(F)=g(F,{\rm cof} F,\det F)b
$
pour tout $F\in\MM^{3\times 3}$ o\`u ${\rm cof}F$ est la matrice des cofacteurs de $F$. La polyconvexit\'e implique la quasiconvexit\'e mais la r\'eciproque est fausse (voir \cite{alibert-dacorogna92,zhang92}). On dit que $W$ est rang-1 convexe si (\ref{CondConvexity}) a lieu pour tout $t\in]0,1[$ et tous $F,F^\prime\in\MM^{m\times N}$ avec ${\rm rang}(F-F^\prime)\leq1$. Si $W$ est quasiconvexe et finie alors $W$ est rang-1 convexe. Si $m\geq 3$ et $N\geq 2$, la rang-1 convexit\'e n'implique pas la quasiconvexit\'e (voir \cite{sverak94}). (Le cas $m=2$ et $N\geq 2$ est encore ouvert.)

A cause du r\'esultat suivant (voir \cite[Corollary 3.2]{ball-murat84}) la quasiconvexit\'e est un concept central du calcul des variations. 
\begin{theorem}\label{CondNessDeQuasi} 
Si $I$ est sfsci dans $W^{1,p}(\Omega;\RR^m)$ alors $W$ est quasiconvexe.
\end{theorem}
 Un probl\`eme (ouvert) important est de savoir si la quasiconvexit\'e est aussi une condition suffisante pour que $I$ soit sfsci dans $W^{1,p}(\Omega;\RR^m)$. Du point de vue de l'hyper\'elasticit\'e, la question (ouverte) est de savoir s'il existe une r\'eciproque du th\'eor\`eme \ref{CondNessDeQuasi} qui est compatible avec (\ref{Cst1}) et (\ref{Cst2}). 

En 1984, Ball et Murat ont d\'emontr\'e le th\'eor\`eme suivant (voir \cite[Theorem 4.5(i)]{ball-murat84} voir aussi \cite{ball77} pour des r\'esultats plus g\'en\'eraux faisant intervenir la polyconvexit\'e).
\begin{theorem}\label{BallSciTheorem}
Soit $W:\MM^{N\times N}\to[0,+\infty]$ d\'efinie par $W(F):=|F|^p+h(\det F)$ avec $h:\RR\to[0,+\infty]$ et $p\geq N$. Si $h$ est sci et convexe {(}ce qui implique que $W$ est polyconvexe{)} alors $I$ est sfsci dans $W^{1,p}(\Omega;\RR^m)$.
\end{theorem}
 Le th\'eor\`eme \ref{BallSciTheorem} est compatible avec (\ref{Cst1}) et (\ref{Cst2}) (par exemple, on peut l'utiliser avec $h:\RR\to[0,+\infty]$ donn\'ee par $h(t)={1/ t}$  si $t>0$ et $h(t)=+\infty$ si $t\leq 0$). Cependant, il n'est pas compl\'etement satisfaisant puisqu'il fait appel \`a la polyconvexit\'e qui n'est pas \'equivalente \`a la quasiconvexit\'e.  

En 1984, Acerbi et Fusco ont prouv\'e le th\'eor\`eme suivant (voir \cite[Theorem II.4]{acerbi-fusco84}).
\begin{theorem}\label{AcerbiFuscoTheorem}
Si $W$ est continue, quasiconvexe et satisfait la condition de croissance d'ordre $p$ suivante 
\begin{eqnarray}\label{CondCroiss}
W(F)\leq c(1+|F|^p)\hbox{ pour tout }F\in\MM^{m\times N} \hbox{ avec }c>0\hbox{,}
\end{eqnarray}
alors $I$ est sfsci dans $W^{1,p}(\Omega;\RR^m)$.
\end{theorem}
 Ce th\'eor\`eme ne fait pas intervenir la polyconvexit\'e mais n'est pas applicable en hyper\'elasticit\'e. En effet, il est clair que (\ref{CondCroiss}) est incompatible avec (\ref{Cst1}) et (\ref{Cst2}). A notre connaissance, du point de vue de l'hyper\'elasticit\'e, le th\'eor\`eme \ref{AcerbiFuscoTheorem} n'a pas encore \'et\'e d\'epass\'e.

\subsection{Relaxation}

Lorsque $I$ n'est pas sfsci dans $W^{1,p}(\Omega;\RR^m)$ (ou bien si on ne sait pas d\'emontrer que $I$ est sfsci dans $W^{1,p}(\Omega;\RR^m)$) on ne peut pas, \`a l'aide de la m\'ethode directe du calcul des variations, r\'esoudre le probl\`eme de minimisation suivant 
\begin{equation}\label{P}
\inf\Big\{E(\phi)=I(\phi)-J(\phi):\phi\in\mathcal{A}\Big\}=:\alpha
\end{equation}
o\`u $\mathcal{A}:=\{\phi\in W^{1,p}(\Omega;\RR^m):\phi=\phi_0\hbox{ sur }\partial\Omega\}$ avec $\phi_0\in W^{1,p}(\Omega;\RR^m)$ et $I,J$ sont d\'efinies en (\ref{DefsOfFuncts}). On consid\`ere alors le probl\`eme suivant 
\begin{equation}\label{Prelax}
\inf\Big\{\overline{E}(\phi):\phi\in\mathcal{A}\Big\}=:\overline{\alpha}
\end{equation}
avec $\overline{E}:W^{1,p}(\Omega;\RR^m)\to\RR\cup\{+\infty\}$ d\'efinie par 
\begin{equation}\label{PrelaxBiS}
\overline{E}(\phi):=\inf\left\{\liminf_{n\to+\infty}E(\phi_n):\mathcal{A}\ni\phi_n\wto\phi\right\}.
\end{equation}
En fait, $\overline{E}$ est la r\'egularis\'ee sci par rapport \`a la convergence faible de $W^{1,p}(\Omega;\RR^m)$ de la fonctionnelle valant $E$ sur $\mathcal{A}$ et $+\infty$ sinon, et donc $\overline{E}$ est sfsci dans $W^{1,p}(\Omega;\RR^m)$. Ainsi (\ref{Prelax}) peut \^etre trait\'e par la m\'ethode directe du calcul des variations. Le passage de (\ref{P}) \`a (\ref{Prelax}) est appel\'e {\em relaxation}. Son int\'er\^et  vient des trois propri\'et\'es suivantes :
\begin{enumerate}
\item[$\diamond$] $\alpha=\overline{\alpha}\ ;$
\item[$\diamond$] si $\phi_n\wto\bar\phi$  dans $W^{1,p}(\Omega;\RR^m)$ avec $\{\phi_n\}_{n\geq 1}$ une suite minimisante pour $E$ dans $\mathcal{A}$ alors $\bar\phi$ est un minimiseur de $\overline{E}$ dans $\mathcal{A}$\ ;
\item[$\diamond$] si $\bar\phi$ est un minimiseur de $\overline{E}$ dans $\mathcal{A}$ alors il existe une suite minimisante $\{\phi_n\}_{n\geq 1}$ pour $E$ dans $\mathcal{A}$ telle que $\phi_n\wto\bar\phi$ dans $W^{1,p}(\Omega;\RR^m)$.
\end{enumerate}
Le probl\`eme (\ref{Prelax}) est appel\'e le probl\`eme relax\'e de (\ref{P}) et les solutions de (\ref{Prelax}) sont dites solutions g\'en\'eralis\'ees de (\ref{P}). L'inconv\'enient de la relaxation est que l'expression dans (\ref{PrelaxBiS}) de la nouvelle fonctionnelle \`a minimiser $\overline{E}$ n'est pas une formule explicite. On est ainsi amen\'e \`a \'etudier l'existence  d'une repr\'esentation (plus explicite) pour $\overline{E}$.
 
 Pour simplifier, dans ce qui suit on supposera qu'il n'y a pas de condition au bord, i.e., $\mathcal{A}=W^{1,p}(\Omega;\RR^m)$. Rappelant que $J$ est lin\'eaire et fortement continue dans $W^{1,p}(\Omega;\RR^m)$, il est facile de voir que 
$
\overline{E}=\overline{I}-J
$
avec $\overline{I}:W^{1,p}(\Omega;\RR^m)\to[0,+\infty]$ donn\'ee par 
\begin{equation}\label{DefFunctRelax}
\overline{I}(\phi):=\inf\left\{\liminf_{n\to+\infty}I(\phi_n):W^{1,p}(\Omega;\RR^m)\ni\phi_n\wto\phi\right\}.
\end{equation}
Du point de vue de l'hyper\'elasticit\'e, le fait que la fonctionnelle $I$ n'est pas sfsci dans $W^{1,p}(\Omega;\RR^m)$ signifie que la densit\'e d'\'energie $W$ caract\'erise le comportement du mat\'eriau \`a l'\'echelle microscopique. Le passage de $I$ \`a $\overline{I}$ s'interpr\`ete donc comme un passage ``micro-macro". Comme la fonctionnelle $\overline{I}$ est  cens\'ee  repr\'esenter le comportement macroscopique d'un mat\'eriau hyper\'elastique, il est naturel de se demander  si elle poss\`ede une repr\'esentation int\'egrale du type 
\begin{eqnarray}\label{FunctRelax}
\overline{I}(\phi)=\int_\Omega \overline{W}(\nabla\phi(x))dx\hbox{ pour tout }\phi\in W^{1,p}(\Omega;\RR^m)
\end{eqnarray}
avec $\overline{W}:\MM^{m\times N}\to[0,+\infty]$ (qui sera n\'ecessairement quasiconvexe d'apr\`es le  th\'eor\`eme \ref{CondNessDeQuasi}). Un bon candidat pour $\overline{W}$ est l'enveloppe quasiconvexe de $W$.

\begin{definition}
L'enveloppe quasiconvexe de $W$ est l'unique fonction {\rm(}lorsqu'elle existe{\rm)} $\mathcal{Q} W:\MM^{m\times N}\to[0,+\infty]$ telle que {\rm:}
\begin{enumerate}
\item[$\diamond$] $\mathcal{Q}W$ est Borel mesurable et $\mathcal{Q}W\leq W$\ ;
\item[$\diamond$] pour tout $g:\MM^{m\times N}\to[0,+\infty]$, si $g$ est Borel mesurable, quasiconvexe et $g\leq W$ alors $g\leq \mathcal{Q}W$.
\end{enumerate}
On dit que $\mathcal{Q}W$ est la plus grande fonction quasiconvexe qui est inf\'erieure \`a $W$.
\end{definition} 

Posons $\Aff_0(Y;\RR^m):=\{\varphi\in\Aff(Y;\RR^m):\varphi=0\hbox{ sur }\partial Y\}$ (avec $\Aff(Y;\RR^m)$ d\'esignant l'ensemble des fonctions continues et affines par morceaux de $Y$ dans $\RR^m$, i.e., $\varphi\in\Aff(Y;\RR^m)$ si et seulement si $\varphi$ est continue et il existe une famille finie $\{V_i\}_{i\in I}$ de sous-ensembles ouverts et disjoints de $Y$ telle que $|Y\setminus\cup_{i\in I} V_i|=0$ et pour chaque $i\in I$, $|\partial V_i|=0$ et $\nabla\varphi(x)=F_i$ dans $V_i$  avec $F_i\in\MM^{m\times N}$) et d\'efinissons $\Z W:\MM^{m\times N}\to[0,+\infty]$ par 
\begin{equation}\label{DefOfZW}
\Z W(F):=\inf\left\{\int_YW(F+\nabla\varphi(x))dx:\varphi\in\Aff_0(Y;\RR^m)\right\}.
\end{equation}
En 1982,  Dacorogna a d\'emontr\'e le th\'eor\`eme suivant (voir \cite[Theorem 5]{dacorogna82} voir aussi \cite[Statement III.7]{acerbi-fusco84}).
\begin{theorem}\label{DacorognaTheorem} Si $W$ est continue et satisfait {(\ref{CondCroiss})} alors {(\ref{FunctRelax})} a lieu avec $\overline{W}=\mathcal{Q}W=\Z W.$
\end{theorem}
 Un probl\`eme (ouvert) important (du point de vue de l'hyper\'elasticit\'e) est de savoir si on a (\ref{FunctRelax}) avec $\overline{W}=\mathcal{Q}W=\Z W$ lorsque $W$ satisfait (\ref{Cst1}) et (\ref{Cst2}). Dans \cite{oah-jpm08a} il a \'et\'e montr\'e que (\ref{FunctRelax}) a lieu avec $\overline{W}=\mathcal{Q}W=\Z W$ lorsque $W$ satisfait (\ref{Cst1prime}) et (\ref{Cst2}) (voir \S 2).

\begin{remark}
En 1982, Dacorogna a aussi prouv\'e le th\'eor\`eme suivant de repr\'esen-tation de la quasiconvexifi\'ee (voir \cite[Theorem 6.9 p. 271]{dacorogna08}).
\begin{theorem}\label{QuasiconvexificationFormulaTheorem}
Si $W$ est finie alors $\mathcal{Q}W=\Z W$.
\end{theorem}
(Une preuve de ce th\'eor\`eme est donn\'ee en \S 2.1.3.) 
\end{remark}

\subsection{Passage 3D-2D}

Consid\'erons un mat\'eriau hyper\'elastique occupant dans une configuration de r\'ef\'erence l'ouvert $\Sigma_\eps:=\Sigma\times]-{\eps\over 2},{\eps\over 2}[\subset\RR^3$ o\`u  $\Sigma\subset\RR^2$ est un ouvert born\'e  et $\eps>0$ est ``tr\`es petit" (cela signifie que le mat\'eriau (tridimensionnel) a une tr\`es petite \'epaisseur, il est ``presque" bidimensionnel). {\em Le passage 3D-2D} consiste \`a trouver une mod\'elisation bidimensionnelle du mat\'eriau consid\'er\'e. Dans notre cas, il s'agit de montrer que $I_\eps:W^{1,p}(\Sigma_\eps;\RR^3)\to[0,+\infty]$ d\'efinie par 
\begin{equation}\label{TriDimEnerMem}
I_\eps(\phi):={1\over \eps}\int_{\Sigma_\eps}W(\nabla\phi(x,x_3))dxdx_3
\end{equation} 
(o\`u $W:\MM^{3\times 3}\to[0,+\infty]$ est la densit\'e d'\'energie du mat\'eriau tridimensionnel repr\'esent\'e par $\Sigma_\eps$ et $(x,x_3)$ avec $x\in\Sigma$ et $x_3\in]-{\eps\over 2},{\eps\over 2}[$ d\'esigne un point de $\Sigma_\eps$) ``converge variationnellement" lorsque $\eps\to 0$  (voir la d\'efinition \ref{DefGammaPiConvergence}) vers $I_{\rm mem}:W^{1,p}(\Sigma;\RR^3)\to[0,+\infty]$ donn\'ee par 
\begin{equation}\label{BiDimEnerMem}
I_{\rm mem}(\psi):=\int_\Sigma W_{\rm mem}(\nabla \psi(x))dx
\end{equation}
avec $W_{\rm mem}:\MM^{3\times 2}\to[0,+\infty]$ (qui sera la densit\'e d'\'energie du mat\'eriau bidimensionnel repr\'esent\'e par $\Sigma$), et de donner une formule (qui d\'ependra de $W$) pour $W_{\rm mem}$. La fonctionnelle $I_{\rm mem}$ est appel\'ee l'\'energie de membrane non lin\'eaire associ\'ee au mat\'eriau bidimensionnel mod\'elis\'e par $\Sigma$. 

La limite de $I_\eps$ lorsque $\eps\to 0$ est calcul\'ee au sens de la $\Gamma(\pi)$-convergence (une variante de la $\Gamma$-convergence de De Giorgi (voir \cite{degiorgi75,degiorgi-franzoni75}, voir aussi la d\'efinition \ref{GammaCvgDef}) d\'evelopp\'ee par Anzellotti, Baldo et Percivale (voir \cite{anzellotti-baldo-percivale94})). Soit $\pi=\{\pi_\eps\}_\eps$ avec  $\pi_\eps:W^{1,p}(\Sigma_\eps;\RR^3)\to W^{1,p}(\Sigma;\RR^3)$ d\'efinie par 
\begin{equation}\label{PourAJOUTRmK}
\pi_\eps(\phi):={1\over\eps}\int_{-{\eps\over 2}}^{\eps\over 2}\phi(\cdot,x_3)dx_3.
\end{equation}

\begin{definition}\label{DefGammaPiConvergence} On dit que $I_\eps$ $\Gamma(\pi)$-converge vers $I_{\rm mem}$ et on \'ecrit 
$$
I_{\rm mem}=\Gamma(\pi)\hbox{\rm -}\lim_{\eps\to 0}I_\eps
$$
si les deux assertions suivantes sont satisfaites {\rm:}
\begin{enumerate}
\item[$\diamond$] pour tout $\psi\in W^{1,p}(\Sigma;\RR^3)$ et tout $\{\phi_\eps\}_\eps\subset W^{1,p}(\Sigma_\eps;\RR^3)$, 
$$
\hbox{si }\pi_\eps(\phi_\eps)\wto \psi\hbox{ dans }W^{1,p}(\Sigma;\RR^3)\hbox{ alors }I_{\rm mem}(\psi)\leq\liminf_{\eps\to 0}I_\eps(\phi_\eps)\ ;
$$
\item[$\diamond$] pour tout $\psi\in W^{1,p}(\Sigma;\RR^3)$, il existe $\{\phi_\eps\}_\eps\subset W^{1,p}(\Sigma_\eps;\RR^3)$ tel que 
$$
\pi_\eps(\phi_\eps)\wto \psi\hbox{ dans }W^{1,p}(\Sigma;\RR^3)\hbox{ et }I_{\rm mem}(\psi)\geq\limsup_{\eps\to 0}I_\eps(\phi_\eps).
$$
\end{enumerate}
\end{definition}

A notre connaissance, la $\Gamma(\pi)$-convergence de (\ref{TriDimEnerMem}) vers (\ref{BiDimEnerMem}) a \'et\'e \'etudi\'ee pour la premi\`ere fois par Percivale en 1991 (voir \cite[\S 3]{percivale91}). Dans son travail, il a essay\'e de prendre en compte les conditions  (\ref{Cst1}) et (\ref{Cst2}). Mais, comme ses r\'esultats contenaient quelques erreurs, il ne les a pas publi\'es. Malgr\'e tout, Percivale a introduit la ``bonne" formule pour $W_{\rm mem}$, i.e., $W_{\rm mem}=\mathcal{Q} W_0$ (l'enveloppe quasiconvexe de $W_0$) avec $W_0:\MM^{3\times 2}\to[0,+\infty]$ donn\'ee par 
\begin{equation}\label{W_0}
W_0(\xi):=\inf\Big\{W(\xi\mid\zeta):\zeta\in\RR^3\Big\}
\end{equation}
avec $(\xi\mid \zeta)$ d\'esignant la matrice de $\MM^{3\times 3}$ correspondant \`a $(\xi,\zeta)\in\MM^{3\times 2}\times\RR^3$. 

En 1993, ind\'ependamment de Percivale, Le Dret et Raoult ont d\'emontr\'e le th\'eor\`eme suivant (voir \cite{ledret-raoult93} et \cite[Theorem 2]{ledret-raoult95}).

\begin{theorem}\label{LeDretRaoult}
Si $W$ est continue et satisfait {(\ref{CondCroiss})} alors $I_{\rm mem}=\Gamma(\pi)\hbox{\rm -}\lim_{\eps\to 0}I_\eps$ avec $W_{\rm mem}=\mathcal{Q} W_0$.
\end{theorem}

 Bien que ce th\'eor\`eme ne soit pas compatible avec les contraintes de l'hyper\'elasti-cit\'e (\ref{Cst1}) et (\ref{Cst2}), il \'etablit un cadre math\'ematique convenable pour \'etudier la r\'eduction de dimension d'un point de vue variationnel (ce th\'eor\`eme est en fait le point de d\'epart de beaucoup de travaux sur le sujet).  

Apr\`es Percivale, en 1996, Ben Belgacem a aussi consid\'er\'e les conditions (\ref{Cst1}) et (\ref{Cst2}). Dans \cite[Theorem 1]{benbelgacem97}, il annon\c cait avoir r\'eussi \`a traiter la $\Gamma(\pi)$-convergence de (\ref{TriDimEnerMem}) vers (\ref{BiDimEnerMem}) en pr\'esence de (\ref{Cst1}) et (\ref{Cst2}). Dans \cite{benbelgacem00}, qui est l'article correspondant \`a la note \cite{benbelgacem97}, l'\'enonc\'e  \cite[Theorem 1]{benbelgacem97} n'est pas compl\'etement d\'emontr\'e (cependant, une preuve plus d\'etaill\'ee, mais pas enti\`erement compl\`ete, peut \^etre trouv\'ee dans sa th\`ese \cite{benbelgacem96}). De plus, pour Ben Belgacem, $W_{\rm mem}=\mathcal{Q}\mathcal{R}W_0$ (l'enveloppe quasiconvexe de l'enveloppe rang-1 convexe de $W_0$). En fait, il a \'et\'e montr\'e dans \cite{oah-jpm07,oah-jpm06} que $\mathcal{Q}\mathcal{R}W_0=\mathcal{Q}W_0$ (voir la remarque \ref{QRW_0=QW_0}). N\'eanmoins, les travaux de Ben Belgacem apportent des avanc\'ees substantielles. En particulier, ils mettent en lumi\`ere l'importance des th\'eor\`emes d'approximation pour des fonctions de Sobolev par des immersions lisses dans l'\'etude du passage 3D-2D en pr\'esence de (\ref{Cst1}) et (\ref{Cst2}).

Dans \cite{oah-jpm06} nous avons trait\'e la $\Gamma(\pi)$-convergence de (\ref{TriDimEnerMem}) vers (\ref{BiDimEnerMem}) en pr\'esence de (\ref{Cst1prime}) et (\ref{Cst2}) (voir le th\'eor\`eme \ref{ThMemb1} et la remarque \ref{ample-integrand}). Utilisant la m\'ethode developp\'ee dans \cite{oah-jpm06} et quelques r\'esultats de Ben Belgacem, dans \cite{oah-jpm08b} nous avons \'etudi\'e la $\Gamma(\pi)$-convergence de (\ref{TriDimEnerMem}) vers (\ref{BiDimEnerMem}) sous les contraintes (\ref{Cst1}) et (\ref{Cst2}) (voir le th\'eor\`eme \ref{ThMemb2}).

\section{Relaxation avec contraintes de type d\'eterminant}

\subsection{Th\'eor\`emes g\'en\'eraux}

Dans \cite{oah-jpm07,oah-jpm08a} nous avons mis au point le th\'eor\`eme de repr\'esentation int\'egrale suivant pour $\overline{I}$ d\'efinie en (\ref{DefFunctRelax}). Ce th\'eor\`eme contient et d\'epasse le th\'eor\`eme \ref{DacorognaTheorem} (voir la remarque \ref{ThmDacorognaCoDe}). En particulier, il est compatible avec (\ref{Cst1prime}) et (\ref{Cst2}) (voir \S 2.3).
\begin{theorem}\label{GeneralThRelax}
Si $\Z W$ d\'efinie en {(\ref{DefOfZW})} satisfait la condition de croissance d'ordre $p$ suivante 
\begin{eqnarray}\label{Cond-CroissZW}
\Z W(F)\leq c(1+|F|^p)\hbox{ pour tout }F\in\MM^{m\times N} \hbox{ avec } c>0\hbox{,}
\end{eqnarray}
alors {(\ref{FunctRelax})} a lieu avec $\overline{W}=\mathcal{Q}W=\Z W$.
\end{theorem}
\begin{proof}
Soit $\Z I:W^{1,p}(\Omega;\RR^m)\to[0,+\infty]$ d\'efinie par 
$$
\Z I(\phi):=\int_\Omega\Z W(\nabla\phi(x))dx
$$
et soient $\overline{I}_{\rm aff},\overline{\Z I}_{\rm aff},\overline{\Z I}:W^{1,p}(\Omega;\RR^m)\to[0,+\infty]$ donn\'ees par :
\begin{enumerate}
\item[$\diamond$] $\displaystyle\overline{I}_{\rm aff}(\phi):=\inf\left\{\liminf_{n\to+\infty}I(\phi_n):\Aff(\Omega;\RR^m)\ni\phi_n\wto\phi\right\};$
\item[$\diamond$] $\displaystyle\overline{\Z I}_{\rm aff}(\phi):=\inf\left\{\liminf_{n\to+\infty}\Z I(\phi_n):\Aff(\Omega;\RR^m)\ni\phi_n\wto\phi\right\};$
\item[$\diamond$] $\displaystyle\overline{\Z I}(\phi):=\inf\left\{\liminf_{n\to+\infty}\Z I(\phi_n):W^{1,p}(\Omega;\RR^m)\ni\phi_n\wto\phi\right\}.$
\end{enumerate}
On a besoin du lemme  suivant qui est valable sans l'hypoth\`ese (\ref{Cond-CroissZW}).  
\begin{lemma}\label{FundamentalLemmma1}
$\overline{I}_{\rm aff}=\overline{\Z I}_{\rm  aff}$.
\end{lemma}
(Pour une preuve du lemme \ref{FundamentalLemmma1} voir \S 2.1.1.) Comme $\Z W$ satisfait (\ref{Cond-CroissZW}) et $\Aff(\Omega;\RR^m)$ est fortement dense dans $W^{1,p}(\Omega;\RR^m)$ (voir la remarque \ref{Ekeland-TemamDensityRemark}) on a $\overline{\Z I}_{\rm  aff}=\overline{\Z I}$. Donc $\overline{I}_{\rm aff}=\overline{\Z I}$ par le lemme \ref{FundamentalLemmma1}. De plus, $\overline{I}\leq\overline{I}_{\rm aff}$ et $\overline{\Z I}\leq \overline{I}$, d'o\`u $\overline{I}=\overline{I}_{\rm aff}=\overline{\Z I}$. D'autre part, Fonseca a d\'emontr\'e la proposition suivante (voir \cite{fonseca88}).
\begin{proposition}\label{FonsecaProperties}
$\Z W$ satisfait les quatre propri\'et\'es qui suivent.
\begin{enumerate} 
\item[(a)] Pour tout ouvert born\'e $D\subset\RR^N$ avec $|\partial D|=0$ et tout $F\in\MM^{m\times N}$,
$$
\Z W(F)=\inf\left\{{1\over |D|}\int_D W(F+\nabla\varphi(x))dx:\varphi\in\Aff_0(D;\RR^m)\right\}.
$$
Ainsi $\displaystyle\Z W(F)\leq{1\over |D|}\int_D W(F+\nabla\varphi(x))dx$
pour tout $\varphi\in\Aff_0(D;\RR^m).$
\item[(b)] Si $\Z W$ est finie alors $\Z W$ est rang-1 convexe.
\item[(c)] Si $\Z W$ est finie alors $\Z W$ est continue.
\item[(d)]  Pour tout ouvert born\'e $D\subset\RR^N$ avec $|\partial D|=0$, tout $F\in\MM^{m\times N}$ et tout $\varphi\in\Aff_0(D;\RR^m)$,
$$
\Z W(F)\leq{1\over |D|}\int_D \Z W(F+\nabla\varphi(x))dx.
$$
\end{enumerate}
\end{proposition}
(Pour une preuve de la proposition \ref{FonsecaProperties} voir \S 2.1.2.) Ici on utilise seulement la proposition \ref{FonsecaProperties}(c), i.e., $Z W$ est continue puisque $\Z W$ satisfait (\ref{Cond-CroissZW}). En fait, on peut montrer le r\'esultat suivant qui est une extension facile du th\'eor\`eme de repr\'esentation de la quasiconvexifi\'ee de Dacorogna (voir le th\'eor\`eme \ref{QuasiconvexificationFormulaTheorem}).

\newtheorem*{QuasiconvexificationFormulaTheorem}{\bf Th\'eor\`eme \ref{QuasiconvexificationFormulaTheorem}-bis}

\begin{QuasiconvexificationFormulaTheorem}
{\em Si $\Z W$ est finie alors $\Z W$ est continue et $\mathcal{Q} W=\Z W$.}  
\end{QuasiconvexificationFormulaTheorem}

(Pour une preuve du th\'eor\`eme \ref{QuasiconvexificationFormulaTheorem}-bis voir \S 2.1.4.) Ainsi, $\Z W$ est continue, quasiconvexe et satisfait (\ref{Cond-CroissZW}), donc $\overline{\Z I}=\Z I$ par le th\'eor\`eme \ref{AcerbiFuscoTheorem}, et le th\'eor\`eme \ref{GeneralThRelax} suit.
\end{proof}

\begin{remark}\label{ThmDacorognaCoDe}
Si $W$ v\'erifie {(\ref{CondCroiss})} alors il est clair que $\Z W$ satisfait (\ref{Cond-CroissZW}). On obtient donc le th\'eor\`eme suivant (qui est un petit peu plus g\'en\'eral que le th\'eor\`eme de relaxation de Dacorogna, voir le th\'eor\`eme \ref{DacorognaTheorem}).

\newtheorem*{DacorognaTheorem}{\bf Th\'eor\`eme \ref{DacorognaTheorem}-bis}

\begin{DacorognaTheorem}
{\em Si $W$ v\'erifie {(\ref{CondCroiss})} alors {(\ref{FunctRelax})} a lieu avec $\overline{W}=\mathcal{Q}W=\Z W.$}
\end{DacorognaTheorem}
\end{remark}

\begin{remark}
En analysant la d\'emonstration du th\'eor\`eme \ref{GeneralThRelax}, on voit que l'on a en fait d\'emontr\'e le r\'esultat suivant.

\newtheorem*{GeneralThRelax}{\bf Th\'eor\`eme \ref{GeneralThRelax}-bis}
\begin{GeneralThRelax}
Si $\Z W$ satisfait (\ref{Cond-CroissZW}) alors $\overline{I}(\phi)=\overline{I}_{\rm aff}(\phi)=\int_\Omega\overline{W}(\nabla\phi(x))dx$ pour tout $\phi\in W^{1,p}(\Omega;\RR^m)$ avec $\overline{W}=\mathcal{Q} W=\Z W$.
\end{GeneralThRelax}
\end{remark}
\begin{remark} En utilisant la m\^eme m\'ethode, dans le cas $p=\infty$ on peut montrer le th\'eor\`eme suivant.
\begin{theorem}
Si $\Z W$ est finie alors $\overline{I}(\phi)=\overline{I}_{\rm aff}(\phi)=\int_\Omega\overline{W}(\nabla\phi(x))dx$ pour tout $\phi\in W^{1,\infty}(\Omega;\RR^m)$ avec $\overline{W}=\mathcal{Q} W=\Z W$.
\end{theorem}
\end{remark}
Un r\'esultat analogue au th\'eor\`eme \ref{GeneralThRelax} a \'et\'e prouv\'e par Ben Belgacem  (voir \cite[Theorem 3.1]{benbelgacem00}).  Soit la suite  $\{\R_i W\}_{i\geq 0}$ d\'efinie par $\R_0 W=W$ et pour tout $i\geq 1$ et tout $F\in\MM^{m\times N}$,
\begin{equation}\label{Kohn-StrangFormula}
\R_{i+1}W(F):=\infff\limits_{t\in[0,1]}\big\{(1-t)\R_i W(F-t a\otimes b)+t\R_i W(F+(1-t)a\otimes b)\big\}
\end{equation}
avec $a\otimes b\in\MM^{m\times N}$ donn\'e par $(a\otimes b)x:=\langle a,x\rangle b$ pour tout $x\in\RR^N$, o\`u $\langle\cdot,\cdot\rangle$ d\'esigne le produit scalaire dans $\RR^N$. D'apr\`es Kohn et Strang, $\R_{i+1}W\leq\R_i W$ pour tout $i\geq 0$ et $\R W=\inf_{i\geq 0}\R_i W$ (voir \cite{kohn-strang86}) o\`u $\R W$ d\'esigne l'enveloppe rang-1 convexe de $W$, i.e., la plus grande fonction rang-1 convexe qui est inf\'erieure \`a $W$. Le th\'eor\`eme de Ben Belgacem s'\'enonce comme suit.

\begin{theorem}\label{RelaxBenBelgacemTh}
Si {:}
\begin{enumerate}
\item[$\diamond$] $\OO_W:={\rm int}\{F\in\MM^{m\times N}:\Z\R_i W(F)\leq\R_{i+1}W(F)\hbox{ pour tout }i\geq 0\}$ est dense dans $\MM^{m\times N}$ ;
\item[$\diamond$] pour tout $i\geq 1$, tout $F\in\MM^{m\times N}$ et tout $\{F_n\}_n\subset \OO_W$,
$$
\hbox{si } F_n\to F \hbox{ alors }\R_i W(F)\geq\limsup_{n\to+\infty}\R_i W(F_n)\hbox{,}
$$
\end{enumerate}
et si $\R W$ satisfait la  condition de croissance d'ordre $p$ suivante 
\begin{eqnarray}\label{RWGrowth-Cond}
\R W(F)\leq c(1+|F|^p) \hbox{ pour tout }F\in\MM^{m\times N} \hbox{ avec }c>0\hbox{,}
\end{eqnarray}
alors {(\ref{FunctRelax})} a lieu avec $\overline{W}=\mathcal{Q}\R W$.
\end{theorem}
 Le th\'eor\`eme \ref{RelaxBenBelgacemTh} est aussi compatible avec (\ref{Cst1prime}) et (\ref{Cst2}) (voir \cite{benbelgacem00} et \cite[Chapitre 1]{benbelgacem96}). Ben Belgacem est le premier a avoir mis au point un th\'eor\`eme de repr\'esentation int\'egrale pour $\overline{I}$ qui contient et d\'epasse le th\'eor\`eme \ref{DacorognaTheorem}. (En fait, la premi\`ere tentative de d\'epassement du th\'eor\`eme \ref{DacorognaTheorem} est due \`a Percivale (voir \cite[\S 2]{percivale91})). De mani\`ere g\'en\'erale, comme la rang-1 convexit\'e et la quasiconvexit\'e ne co\"\i ncident pas, les th\'eor\`emes \ref{GeneralThRelax} et \ref{RelaxBenBelgacemTh} ne sont pas identiques. Cependant, on a la proposition ``mixte" suivante.
\begin{proposition}\label{PropMixte}
Si $\R W$ satisfait {(\ref{RWGrowth-Cond})} et si $\Z W$ est finie alors {(\ref{FunctRelax})} a lieu avec $\overline{W}=\mathcal{Q}W$.
\end{proposition}
\begin{proof} Comme $\Z W$ est finie, de la proposition \ref{FonsecaProperties}(b) on d\'eduit que $\Z W$ est rang-1 convexe, donc $\Z W\leq \R W$. Ainsi $\Z W$ satisfait (\ref{Cond-CroissZW}) puisque $\R W$ satisfait (\ref{RWGrowth-Cond}) et la proposition \ref{PropMixte} suit par le th\'eor\`eme \ref{GeneralThRelax}.
\end{proof}

\begin{remark}\label{QW-QRW-LiNk}
Si $\Z W$ est finie alors $\mathcal{Q}\R W=\mathcal{Q} W=\Z W$. En effet,  du th\'eor\`eme \ref{QuasiconvexificationFormulaTheorem}-bis on d\'eduit que $\Z W$ est continue et $\mathcal{Q}W=\Z W$. Ainsi $\mathcal{Q}W$ est continue et quasiconvexe donc rang-1 convexe, d'o\`u $\mathcal{Q} W\leq \R W$ et par cons\'equent $\mathcal{Q} W\leq \mathcal{Q}\R W$. D'autre part, $\mathcal{Q}\R W\leq\mathcal{Q}W$ puisque $\R W\leq W$, ce qui donne le r\'esultat.
\end{remark}

\subsubsection{D\'emonstration du lemme \ref{FundamentalLemmma1}} Il est facile de voir que $\overline{\Z I}_{\rm  aff}\leq\overline{I}_{\rm aff}$, donc tout revient \`a montrer que $\overline{I}_{\rm aff}\leq\overline{\Z I}_{\rm  aff}$. Pour cela, il suffit de prouver que 
\begin{eqnarray}\label{step1VitAli}
\hbox{si }\phi\in\Aff(\Omega;\RR^m)\hbox{ alors }
\overline{I}_{\rm aff}(\phi)\leq \int_\Omega \Z W(\nabla \phi(x))dx.
\end{eqnarray}
Soit $\phi\in\Aff(\Omega;\RR^m)$. Par d\'efinition, il existe une famille finie $\{V_i\}_{i\in I}$ de sous-ensembles ouverts et disjoints de $\Omega$ telle que $|\Omega\setminus\cup_{i\in I}V_i|=0$ et pour chaque $i\in I$, $|\partial V_i|=0$ et $\nabla\phi(x)=F_i$ dans $V_i$ avec $F_i\in\MM^{m\times N}$.   \'Etant donn\'es $\delta>0$ et $i\in I$, on  consid\`ere $\varphi_i\in \Aff_0(Y;\RR^m)$ tel que 
\begin{equation}\label{Zk}
\int_Y W(F_i+\nabla\varphi_i(y))dy\leq\Z W(F_i)+{\delta\over|\Omega|}.
\end{equation}
Soit $n\geq 1$. Par le th\'eor\`eme de recouvrement de Vitali, il existe une famille au plus d\'enombrable $\{a_{i,j}+\alpha_{i,j}Y\}_{j\in J_{i}}$ de sous-ensembles disjoints de  $V_i$, o\`u $a_{i,j}\in\RR^N$ et $0<\alpha_{i,j}<{1\over n}$, telle que
$
\big|V_i\setminus\cup_{j\in J_{i}}(a_{i,j}+\alpha_{i,j}Y)\big|=0
$
 (donc $\sum_{j\in J_i}\alpha_{i,j}^N=|V_i|$). D\'efinissons $\phi_n:\Omega\to\RR^m$ par 
$$
\phi_n(x):=
\alpha_{i,j}\varphi_{i}\left({x-a_{i,j}\over \alpha_{i,j}}\right)\hbox{ si }x\in a_{i,j}+\alpha_{i,j}Y.
$$
Puisque $\varphi_i\in\Aff_0(Y;\RR^m)$, il existe une famille finie $\{Y_{i,l}\}_{l\in L_i}$ de sous-ensembles ouverts et disjoints de $Y$ telle que $|Y\setminus\cup_{l\in L_i}Y_{i,l}|=0$ et pour chaque $l\in L_i$, $|\partial Y_{i,l}|=0$ et $\nabla\varphi_i(y)=G_{i,l}$ dans $Y_{i,l}$ avec $G_{i,l}\in\MM^{m\times N}$. Posons
$
U_{i,l,n}:=\cup_{j\in J_i}a_{i,j}+\alpha_{i,j}Y_{i,l}.
$
Alors $|\Omega\setminus\cup_{i\in I}\cup_{l\in L_i}U_{i,l,n}|=0$ et pour chaque $i\in I$ et chaque $l\in L_i$, $|\partial U_{i,l,n}|=0$ et $\nabla\phi_n(x)=G_{i,l}$ dans $U_{i,l,n}$. D'o\`u $\phi_n\in\Aff_0(\Omega;\RR^m)$. D'autre part, $\|\phi_n\|_{L^\infty(\Omega;\RR^m)}\leq {1\over n}\max_{i\in I}\|\varphi_i\|_{L^\infty(Y;\RR^m)}$ et $\|\nabla\phi_n\|_{L^\infty(\Omega;\MM^{m\times N})}\leq\max_{i\in I}\|\nabla\varphi_i\|_{L^\infty(Y;\MM^{m\times N})}$, donc (\`a une sous-suite pr\`es) $\phi_n\stackrel{*}{\wto}0$ dans $W^{1,\infty}(\Omega;\RR^m)$, o\`u ``$\stackrel{*}{\wto}$" d\'esigne la convergence $*$-faible dans  $W^{1,\infty}(\Omega;\RR^m)$. D'o\`u $\phi_n\wto 0$ dans $W^{1,p}(\Omega;\RR^m)$. De plus, on a 
\begin{eqnarray*}
\int_\Omega W\left(\nabla \phi(x)+\nabla\phi_n(x)\right)dx&=&\sum_{i\in I}\int_{V_i} W\left(F_i+\nabla\phi_n(x)\right)dx\\
&=&\sum_{i\in I}\sum_{j\in J_i}\alpha_{i,j}^N\int_{Y}W\left(F_i+\nabla\varphi_{i}(y)\right)dy\\
&=&\sum_{i\in I}|V_i|\int_{Y}W\left(F_i+\nabla\varphi_{i}(y)\right)dy.
\end{eqnarray*}
Comme $\phi+\phi_n\in\Aff(\Omega;\RR^m)$ et $\phi+\phi_n\wto \phi$ dans $W^{1,p}(\Omega;\RR^m)$, utilisant (\ref{Zk}) on d\'eduit que 
\begin{eqnarray*}
\overline{I}_{\rm aff}(\phi)\leq\liminf_{n\to+\infty}\int_\Omega W\left(\nabla \phi(x)+\nabla\phi_n(x)\right)dx&\leq&\sum_{i\in I}|V_i|\Z W(F_i)+\delta\\
&=&\int_\Omega\Z W(\nabla \phi(x))dx + \delta,
\end{eqnarray*}
et (\ref{step1VitAli}) suit.\hfill$\square$

\subsubsection{D\'emonstration de la proposition \ref{FonsecaProperties}}
Dans ce qui suit on d\'emontre la proposition \ref{FonsecaProperties}.
\begin{definition}\label{local-rank1}
Soit $U\subset \MM^{m\times N}$ un ouvert. 
\begin{itemize}
\item[$\diamond$] On dit que $W$ est rang-1 convexe en $F$ dans $U$ si pour chaque $\theta\in[0,1[$, chaque $a\in\RR^N$ et chaque $b\in\RR^m$ tels que $F+\mu a\otimes b\in U$ pour tout $\mu\in[-\frac{\theta}{1-\theta},1]$, on a 
\begin{eqnarray*}\label{rank-one convex}
W(F)\le \theta W(F+a\otimes b)+(1-\theta)W(F-\frac{\theta}{1-\theta}a\otimes b)
\end{eqnarray*}
(avec $a\otimes b\in\RR^N\otimes\RR^m\subset\MM^{m\times N}$ donn\'e par $(a\otimes b)x=\langle a,x\rangle b$ pour tout $x\in\RR^N$, o\`u $\langle\cdot,\cdot\rangle$ d\'esigne le produit scalaire dans $\RR^N$).
\item[$\diamond$] On dit que $W$ est rang-1 convexe dans $U$ si $W$ est rang-1 convexe en $F$ dans $U$ pour tout $F\in U$.
\end{itemize}
\end{definition}
\begin{remark}\label{equivalence def rank one convex} 
Lorsque $U=\MM^{m\times N}$ la d\'efinition \ref{local-rank1} co\"\i ncide avec la d\'efinition classique, i.e., $W$ est rang-1 convexe si \eqref{CondConvexity} a lieu pour tout $F,F^\prime\in\MM^{m\times N}$ avec $\rank(F-F^\prime)\le 1$.
\end{remark}
La proposition \ref{FonsecaProperties}(b) se d\'eduit de la proposition suivante.
\begin{proposition}\label{prop_fonseca_ball_morrey}
La fonction $\Z W$ est rang-1 convexe dans $\rm{int}(\rm{dom}\Z W)$.
\end{proposition}
Pour montrer la proposition \ref{prop_fonseca_ball_morrey} nous aurons besoin de la proposition \ref{FonsecaProperties}(a) et (d).
\begin{proof}[D\'emonstration de la proposition \ref{FonsecaProperties}(a)]\label{invariance de ZW} Soient $D\subset\RR^N$ un ouvert born\'e tel que $\vert\partial D\vert=0$, $F\in\MM^{m\times N}$ et $\varphi\in\Aff_0(D;\RR^m)$. Par le th\'eor\`eme de recouvrement de Vitali, il existe une famille au plus d\'enombrable $\{a_i+\eps_i D\}_{i\in I}$ de sous-ensembles ouverts et disjoints de $Y$, o\`u $a_i\in\RR^N$ et $0<\eps_i<1$, telle que $\vert Y\setminus \cup_{i\in I}(a_i+\eps_i D)\vert=0$ $(\mbox{donc }\sum_{i\in I}\eps_i^N=\frac{1}{\vert D\vert})$. D\'efinissons $\phi\in \Aff_0(Y;\RR^m)$ par 
\[
\phi(x)= \eps_i\varphi\left(\frac{x-a_i}{\eps_i}\right)\mbox{ si }x\in a_i+\eps_iD.
\]
Alors 
\begin{eqnarray*}
\Z W(F)\le\int_Y W(F+\nabla \phi(x))dx&=&\sum_{i\in I}\int_{a_i+\eps_i D}W\left(F+\nabla \varphi\left(\frac{x-a_i}{\eps_i}\right)\right)dx\\
&=&\sum_{i\in I}\eps_i^N\int_{D}W(F+\nabla \varphi(x))dx\\
&=&\frac{1}{\vert D\vert}\int_D W(F+\nabla \varphi(x))dx.
\end{eqnarray*}
Il suit que 
$$
\Z W(F)\le \inf\left\{\frac{1}{\vert D\vert}\int_D W(F+\nabla \varphi(x))dx:\varphi\in\Aff_0(D;\RR^m)\right\}.
$$
En \'echangeant le r\^ole de $D$ et $Y$ on obtient l'in\'egalit\'e oppos\'ee, ce qui termine la preuve.
\end{proof}
\begin{proof}[D\'emonstration de la proposition \ref{FonsecaProperties}(d)] Soient $D\subset\RR^N$ un ouvert born\'e tel que $\vert\partial D\vert=0$, $F\in\MM^{m\times N}$ et $\varphi\in\Aff_0(D;\RR^m)$. Par d\'efinition,  il existe une famille finie $\{V_i\}_{i\in I}$ de sous-ensembles ouverts et disjoints  de $D$ telle que $\vert D\setminus \cup_{i\in I} V_i\vert=0$ et pour chaque  $i\in I$, $|\partial V_i|=0$ et $\nabla \varphi(x)=F_i$ dans $V_i$ avec $F_i\in\MM^{m\times N}$. \'Etant donn\'e $\eps>0$ et $i\in I$, par la proposition \ref{FonsecaProperties}(a), il existe $\varphi_{\eps,i}\in\Aff_0(V_i;\RR^m)$ tel que 
\begin{equation}\label{RaJoUTEq1}
\frac{1}{\vert V_i\vert}\int_{V_i}W(F+F_i+\nabla\varphi_{\eps,i}(x))dx\le \Z W(F+F_{i})+\eps.
\end{equation}
D\'efinissons $\varphi_{\eps}\in\Aff_0(D;\RR^m)$ par $\varphi_\eps(x)=\varphi(x)+\varphi_{\eps,i}(x)$ si $x\in V_i$. Utilisant \`a nouveau la proposition \ref{FonsecaProperties}(a) et (\ref{RaJoUTEq1}) on d\'eduit que 
\begin{eqnarray*}
\Z W(F)\le \frac{1}{\vert D\vert}\int_D W(F+\nabla\varphi_\eps(x))dx&=& \frac{1}{\vert V_i\vert}\sum_{i\in I} \int_{V_i} W(F+F_i+\nabla\varphi_{\eps,i}(x))dx\\
&\le& \frac{1}{\vert D\vert}\int_D \Z W(F+\nabla \varphi(x))dx+\eps\end{eqnarray*}
et le r\'esultat suit en faisant $\eps\to 0$.
\end{proof}
La d\'emonstration suivante est d\^u \`a Fonseca (voir \cite{fonseca88}) : elle repose sur la construction d'une certaine fonction continue et affine par morceaux (pour des constructions similaires voir  Ball \cite{ball77} et Morrey \cite{morrey52}).
\begin{proof}[D\'emonstration de la Proposition \ref{prop_fonseca_ball_morrey}]
Pour simplifier on suppose que $N=3$ (la g\'en\'eralisation \`a $N$ quelconque est ais\'ee). Soient $F\in \mbox{int}(\mbox{dom}\Z W)$, $\theta\in [0,1]$, $a\in\RR^3$ et $b\in\RR^m$ tels que $F+\mu a\otimes b\in \mbox{int}(\mbox{dom}\Z W)$ pour tout $\mu\in [-\frac{\theta}{1-\theta},1]$. Sans perdre de g\'en\'eralit\'e on peut supposer que $|a|=1$. Soient $(\tau,\eta,a)$ une base orthonorm\'ee de $\RR^3$ et $k_0\in\NN$ tels que 
\begin{equation}\label{intdomain fonseca proof}
F+\frac{2\theta}{k}\rho\otimes b\in \mbox{int}(\mbox{dom}\Z W)
\end{equation}
pour tout $k\ge k_0$ et tout $\rho\in\RR^3$ satisfaisant $|\rho|=1$. \'Etant donn\'e $k>k_0$ on d\'efinit le parall\'elipip\`ede  $P_k\subset\RR^3$ et le rectangle $R_k\subset P_k$ par :
\begin{itemize}
    \item[]$P_k:=\{x\in\RR^3:\vert\langle \tau,x\rangle\vert\le\frac{k}{2},\vert\langle \eta,x\rangle\vert\le\frac{k}{2},-(1-\theta)\le\langle a,x\rangle\le\theta\}$ ;
    \item[]$R_k:=\{x\in P_k:\vert\langle \tau,x\rangle\vert\le\frac{k-k_0}{2},\vert\langle \eta,x\rangle\vert\le\frac{k-k_0}{2},\langle a,x\rangle=0\}$,
\end{itemize}
(o\`u $\langle\cdot,\cdot\rangle$ est le produit scalaire dans $\RR^3$) et on consid\`ere $P_k^+,P_k^-,T_k^1,T_k^2,T_k^3,T_k^4\subset P_k$ (voir la figure \ref{figfonseca}) donn\'es par :
\begin{itemize}
    \item[]$P_k^+:=\mbox{co}(\{A_1,A_2,A_3,A_4,B_1,B_2,B_3,B_4\})$ ;
    \item[]$P_k^-:=\mbox{co}(\{A_5,A_6,A_7,A_8,B_1,B_2,B_3,B_4\})$ ;
    \item[]$T_k^1:=\mbox{co}(\{A_1,A_4,A_5,A_8,B_1,B_4\})$ ;
    \item[]$T_k^2:=\mbox{co}(\{A_1,A_2,A_5,A_6,B_1,B_2\})$ ;
    \item[]$T_k^3:=\mbox{co}(\{A_2,A_3,A_6,A_7,B_2,B_3\})$ ;
    \item[]$T_k^4:=\mbox{co}(\{A_3,A_4,A_7,A_8,B_3,B_4\})$,
\end{itemize}    
o\`u $A_1,A_2,A_3,A_4,A_5,A_6,A_7,A_8$ (resp. $B_1,B_2,B_3,B_4$) sont les sommets de $P_k$ (resp. $R_k$) et co$(X)$ d\'esigne l'enveloppe convexe de l'ensemble $X\subset\MM^{m\times N}$. 

\begin{figure}[H]
\begin{center}
\scalebox{0.8} % Change this value to rescale the drawing.
{
\begin{pspicture}(0,-5.41)(11.38,5.41)
\psline[linewidth=0.04cm](3.6,4.31)(8.4,4.31)
\psline[linewidth=0.04cm](2.0,1.11)(6.8,1.11)
\psline[linewidth=0.04cm](3.6,4.31)(2.0,1.11)
\psline[linewidth=0.04cm](8.4,4.31)(6.8,1.11)
\psline[linewidth=0.04cm](6.8,-3.69)(8.4,-0.49)
\psline[linewidth=0.04cm,linestyle=dashed,dash=0.16cm 0.16cm](6.8,1.11)(6.4,-0.09)
\psline[linewidth=0.04cm,linestyle=dashed,dash=0.16cm 0.16cm](2.4,-0.09)(4.0,3.11)
\psline[linewidth=0.04cm,linestyle=dashed,dash=0.16cm 0.16cm](6.4,-0.09)(8.0,3.11)
\psline[linewidth=0.04cm,linestyle=dashed,dash=0.16cm 0.16cm](8.0,3.11)(8.4,4.31)
\psline[linewidth=0.04cm,linestyle=dashed,dash=0.16cm 0.16cm](6.4,-0.09)(2.4,-0.09)
\psline[linewidth=0.04cm,linestyle=dashed,dash=0.16cm 0.16cm](2.0,1.11)(2.4,-0.09)
\psline[linewidth=0.04cm,linestyle=dashed,dash=0.16cm 0.16cm](3.6,4.31)(4.0,3.11)
\psline[linewidth=0.04cm,linestyle=dashed,dash=0.16cm 0.16cm](8.0,3.11)(4.0,3.11)
\psline[linewidth=0.04cm](2.0,1.11)(2.0,-3.69)
\psline[linewidth=0.04cm](6.8,1.11)(6.8,-3.69)
\psline[linewidth=0.04cm](2.0,-3.69)(6.8,-3.69)
\psline[linewidth=0.04cm](8.4,4.31)(8.4,-0.49)
\psline[linewidth=0.04cm,linestyle=dashed,dash=0.16cm 0.16cm](8.0,3.11)(8.4,-0.49)
\psline[linewidth=0.04cm,linestyle=dashed,dash=0.16cm 0.16cm](6.4,-0.09)(6.8,-3.69)
\psline[linewidth=0.02cm,linestyle=dashed,dash=0.16cm 0.16cm](3.6,4.31)(3.6,-0.49)
\psline[linewidth=0.04cm,linestyle=dashed,dash=0.16cm 0.16cm](2.4,-0.09)(2.0,-3.69)
\psline[linewidth=0.02cm,linestyle=dashed,dash=0.16cm 0.16cm](3.6,-0.49)(2.0,-3.69)
\psline[linewidth=0.02cm,linestyle=dashed,dash=0.16cm 0.16cm](3.6,-0.49)(8.4,-0.49)
\psline[linewidth=0.04cm,linestyle=dashed,dash=0.16cm 0.16cm](4.0,3.11)(3.6,-0.49)
\psline[linewidth=0.04cm,arrowsize=0.05291667cm 2.0,arrowlength=1.4,arrowinset=0.4]{<-}(5.2,3.51)(5.2,1.51)
\psline[linewidth=0.04cm,arrowsize=0.05291667cm 2.0,arrowlength=1.4,arrowinset=0.4]{<-}(6.8,1.51)(5.2,1.51)
\psline[linewidth=0.04cm,arrowsize=0.05291667cm 2.0,arrowlength=1.4,arrowinset=0.4]{<-}(4.4,0.31)(5.2,1.51)
\psline[linewidth=0.04cm,arrowsize=0.05291667cm 2.0,arrowlength=1.4,arrowinset=0.4]{<->}(1.6,1.11)(1.6,-0.09)
\psline[linewidth=0.04cm,arrowsize=0.05291667cm 2.0,arrowlength=1.4,arrowinset=0.4]{<->}(1.6,-0.09)(1.6,-3.69)
\psline[linewidth=0.04cm,arrowsize=0.05291667cm 2.0,arrowlength=1.4,arrowinset=0.4]{<->}(2.0,-4.09)(6.8,-4.09)
\psline[linewidth=0.04cm,arrowsize=0.05291667cm 2.0,arrowlength=1.4,arrowinset=0.4]{<->}(0.42,1.11)(0.42,-3.69)
\usefont{T1}{ptm}{m}{n}
\rput(1.0,0.5){$\theta$}
\usefont{T1}{ptm}{m}{n}
\rput(1.0,-2.0){$1-\theta$}
\usefont{T1}{ptm}{m}{n}
\rput(0.23,-1.185){$1$}
\usefont{T1}{ptm}{m}{n}
\rput(4.63,-4.385){$k$}
\psline[linewidth=0.04cm,arrowsize=0.05291667cm 2.0,arrowlength=1.4,arrowinset=0.4]{<->}(7.2,-3.77)(8.8,-0.49)
\usefont{T1}{ptm}{m}{n}
\rput(8.3,-1.985){$k$}
\usefont{T1}{ptm}{m}{n}
\rput(5.4,3.3){$a$}
\usefont{T1}{ptm}{m}{n}
\rput(4.7,0.415){$\tau$}
\usefont{T1}{ptm}{m}{n}
\rput(6.57,1.76){$\eta$}
\psline[linewidth=0.04cm,arrowsize=0.05291667cm 2.0,arrowlength=1.4,arrowinset=0.4]{<->}(7.2,1.51)(8,1.51)

\usefont{T1}{ptm}{m}{n}
\rput(7.6,1.215){$\frac{k_0}{2}$}

\psline[linewidth=0.04cm,arrowsize=0.05291667cm 2.0,arrowlength=1.4,arrowinset=0.4]{<->}(5.6,-0.09)(5.2,-0.89)
\usefont{T1}{ptm}{m}{n}
\rput(4.59,-2.0){$P_k^-$}
\usefont{T1}{ptm}{m}{n}
\rput(5.03,4.015){$P_k^+$}

\usefont{T1}{ptm}{m}{n}
\rput(5.02,-0.385){$\frac{k_0}{2}$}

\usefont{T1}{ptm}{m}{n}
\rput(6.4,0.82){$A_1$}
\usefont{T1}{ptm}{m}{n}
\rput(8.7,4.5){$A_2$}
\usefont{T1}{ptm}{m}{n}
\rput(3.50,4.60){$A_3$}
\usefont{T1}{ptm}{m}{n}
\rput(1.75,1.4){$A_4$}
\usefont{T1}{ptm}{m}{n}
\rput(7.10,-3.985){$A_5$}
\usefont{T1}{ptm}{m}{n}
\rput(8.85,-0.345){$A_6$}
\usefont{T1}{ptm}{m}{n}
\rput(3.25,-0.385){$A_7$}
\usefont{T1}{ptm}{m}{n}
\rput(1.74,-3.985){$A_8$}
\usefont{T1}{ptm}{m}{n}
\rput(6.1,0.14){$B_1$}
\usefont{T1}{ptm}{m}{n}
\rput(7.7,3.615){$B_2$}
\usefont{T1}{ptm}{m}{n}
\rput(4.28,3.4){$B_3$}
\usefont{T1}{ptm}{m}{n}
\rput(2.83,0.14){$B_4$}
\psline[linewidth=0.02cm,arrowsize=0.05291667cm 2.0,arrowlength=1.4,arrowinset=0.4]{<-}(4.8,-2.89)(5.6,-5.29)
\psline[linewidth=0.02cm,arrowsize=0.05291667cm 2.0,arrowlength=1.4,arrowinset=0.4]{<-}(8.0,0.31)(10.4,0.31)
\psline[linewidth=0.02cm,arrowsize=0.05291667cm 2.0,arrowlength=1.4,arrowinset=0.4]{<-}(3.2,0.71)(2.0,3.91)
\psline[linewidth=0.02cm,arrowsize=0.05291667cm 2.0,arrowlength=1.4,arrowinset=0.4]{<-}(6.4,3.40)(6.8,5.11)
\usefont{T1}{ptm}{m}{n}
\rput(5.9,-5.185){$T_k^1$}
\usefont{T1}{ptm}{m}{n}
\rput(10.7,0.415){$T_k^2$}
\usefont{T1}{ptm}{m}{n}
\rput(6.95,5.30){$T_k^3$}
\usefont{T1}{ptm}{m}{n}
\rput(2.15,4.14){$T_k^4$}
\end{pspicture} 
}
\end{center}
\caption{Les ensembles  $P_k^+,P_k^-,T_k^1,T_k^2,T_k^3,T_k^4$.}
\label{figfonseca}
\end{figure}
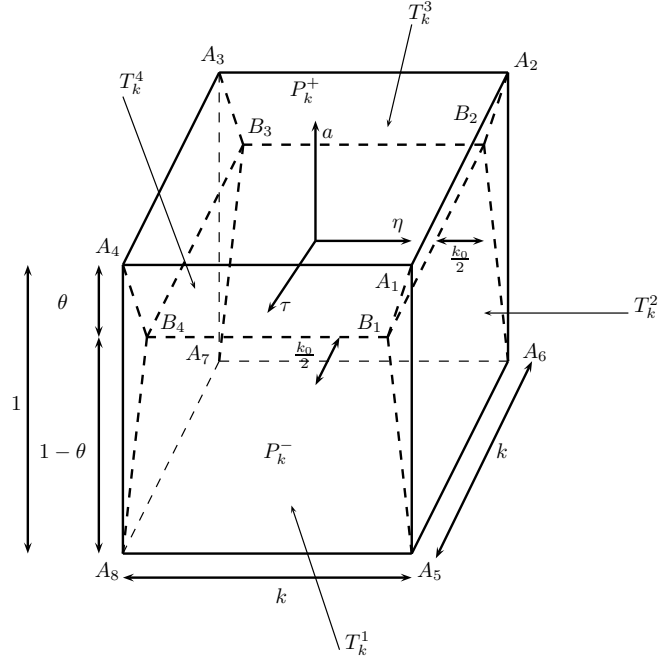

D\'efinissons $\phi_k\in\Aff_0(P_k;\RR)$ par 
$$
\phi_k(x):=\left\{
\begin{array}{ll}
\langle a,x\rangle-\theta&\hbox{si }x\in P_k^+\\
\langle-\frac{\theta}{1-\theta} a,x\rangle-\theta&\hbox{si }x\in P_k^-\\
\langle\frac{2\theta}{k_0} \tau,x\rangle-\frac{k}{k_0}\theta&\hbox{si }x\in T_k^1\\
\langle\frac{2\theta}{k_0} \eta,x\rangle-\frac{k}{k_0}\theta&\hbox{si }x\in T_k^2\\
\langle-\frac{2\theta}{k_0} \tau,x\rangle-\frac{k}{k_0}\theta&\hbox{si }x\in T_k^3\\
\langle-\frac{2\theta}{k_0} \eta,x\rangle-\frac{k}{k_0}\theta&\hbox{si }x\in T_k^4,
\end{array}
\right.
$$
et posons $\varphi_k:=\phi_k b\in\Aff_0(P_k;\RR^m)$. Utilisant la proposition \ref{FonsecaProperties}(d) on voit que 
\begin{eqnarray*}
\Z W(F)&\le& \frac{1}{\vert P_k\vert}\int_{P_k}\Z W(F+\nabla\varphi_k(x))dx\\
&=&\frac{1}{k^2}\Big(\Z W(F+a\otimes b)\vert P_k^+\vert+\Z W\big(F-\frac{\theta}{1-\theta} a\otimes b\big)\vert P_k^-\vert\\
&&+\sum_{i=1}^{4}\int_{T_k^i}\Z W(F+\nabla \varphi_k(x))dx\Big),
\end{eqnarray*}
d'o\`u 
\begin{equation}\label{pass limit fonseca proof}
\Z W(F)\le \Z W(F+a\otimes b)\frac{\vert P_k^+\vert}{k^2}+\Z W\big(F-\frac{\theta}{1-\theta} a\otimes b\big)\frac{\vert P_k^-\vert}{k^2}+4C\frac{\vert T_k^1\vert}{k^2}
\end{equation}
avec $C:=\max\{\vert \Z W(F+\nabla\varphi_k(x))\vert:x\in \cup_{i=1}^{4}T_k^i\}$ ($C<+\infty$ grace \`a (\ref{intdomain fonseca proof})). Posons $T_k^{1,+}:=T_k^1\cap\{x\in\RR^3:\langle a,x\rangle\ge 0\}$ et $T_k^{1,-}:=T_k^1\cap\{x\in\RR^3:\langle a,x\rangle\le 0\}$. Alors :
\begin{itemize}
\item[$\diamond$] $\max\{\vert T_k^{1,+}\vert,\vert T_k^{1,-}\vert\}\le \max\big\{\frac{\theta k k_0}{2},\frac{(1-\theta) k k_0}{2}\big\}$ ;
\item[$\diamond$] $\vert P_k^+\vert=\theta k^2-4\vert T_k^{1,+}\vert$ ;
\item[$\diamond$] $\vert P_k^-\vert=(1-\theta) k^2-4\vert T_k^{1,-}\vert$,
\end{itemize}
donc $\frac{\vert P_k^+\vert}{k^2}\to \theta$, $\frac{\vert P_k^-\vert}{k^2}\to 1-\theta$ et $\frac{\vert T_k^{1}\vert}{k^2}\to 0$ lorsque $k\to +\infty$, et le r\'esultat suit en faisant $k\to +\infty$ dans \eqref{pass limit fonseca proof}.
\end{proof}

\begin{definition} 
On dit que $U\subset\MM^{m\times N}$ est lamination-convexe si pour tout $F,G\in U$ satisfaisant rg$(F-G)\le 1$ on a $[F,G]=\{\lambda F+(1-\lambda)G:\lambda\in [0,1]\}\subset U$.
\end{definition}
La proposition \ref{FonsecaProperties}(c) est une cons\'equence imm\'ediate du r\'esultat suivant qui se d\'eduit ais\'ement de la proposition \ref{prop_fonseca_ball_morrey} et de \cite[Theorem 2.31 p. 47]{dacorogna08} (en remarquant qu'une fonction rang-1 convexe est s\'epar\'ement convexe).
\begin{corollary}\label{lamination} 
Si $\mbox{dom}\Z W$ est ouvert et lamination-convexe alors $\Z W$ est rang-1 convexe et continue sur $\mbox{dom}\Z W$.
\end{corollary}

\subsubsection{D\'emonstration du th\'eor\`eme \ref{QuasiconvexificationFormulaTheorem}} Puisque $\Z W\ge \mathcal{Q} W$ il suffit de montrer que $\Z W$ est quasiconvexe.  Soient $F\in\MM^{m\times N}$ et $\phi\in W^{1,\infty}_0(Y;\RR^m)$ tels que $\int_Y W(F+\nabla\phi(x))dx<+\infty$. Par la remarque \ref{Ekeland-TemamDensityRemark}, il existe $\{\phi_n\}_{n\ge 1}\subset\Aff_0(Y;\RR^m)$ tel que $\nabla \phi_n(x)\to \nabla \phi(x)$ p.p. dans $Y$ et $\sup_{n \ge 1}\|\nabla \phi_n\|_{L^\infty(Y;\MM^{m\times N})}=:C<+\infty$. Puisque $W$ est finie, $\Z W$ l'est aussi, et donc $\Z W$ est continue par la  proposition \ref{FonsecaProperties}(c). D'o\`u : 
\begin{itemize}
\item[$\diamond$] pour chaque $n\geq 1$, $\Z W(F+\nabla \phi_n(x))\le \sup_{\vert\zeta\vert\le C} \Z W(F+\zeta)<+\infty$ p.p. dans $Y$ ;
\item[$\diamond$] $\Z W(F+\nabla\phi_n(x))\to \Z W(F+\nabla\phi(x))$ p.p. dans $Y$.
\end{itemize}
Il suit que 
$$
\lim_{n\to +\infty} \int_Y \Z W(F+\nabla \phi_n(x))dx=\int_Y \Z W(F+\nabla \phi(x))dx.
$$
par le th\'eor\`eme de convergence domin\'ee de Lebesgue. \'Etant donn\'e $\eps>0$, consid\'e-rons $n\ge 1$ tel que 
\begin{equation}\label{RajoutEqUaTioN1}
\int_Y \Z W(F+\nabla \phi(x))dx+{\eps}
\geq\int_Y \Z W(F+\nabla \phi_{n}(x))dx=\sum_{i\in I}|V_i|\Z W(F+F_i),
\end{equation}
o\`u $\{V_i\}_{i\in I}$ est une famille finie de sous-ensembles ouverts et disjoints de $Y$  telle que $|Y\setminus\cup_{i\in I} V_i|=0$ et pour chaque $i\in I$, $|\partial V_i|=0$ et $\nabla\phi_n(x)=F_i$ dans $V_i$  avec $F_i\in\MM^{m\times N}$. Par la proposition \ref{FonsecaProperties}(a),  pour chaque $i\in I$, il existe $\varphi_i\in \Aff_0(V_i;\RR^m)$ tel que 
\begin{equation}\label{RajoutEqUaTioN2}
\Z W(F+F_i)\ge  {1\over\vert V_i\vert}\int_{V_i} W(F+\nabla\phi_n(x)+\nabla\varphi_{i}(x))dx-{\eps}.
\end{equation}
D\'efinissons $\varphi_n\in\Aff_0(Y;\RR^m)$ par $\varphi_n(x)=\phi_n(x)+\varphi_{i}(x)\hbox{ si } x\in  V_i$. Utilisant (\ref{RajoutEqUaTioN1}) et (\ref{RajoutEqUaTioN2}) on d\'eduit que 
\[
\int_Y \Z W(F+\nabla \phi(x))dx+2\eps\geq
\int_Y W(F+\nabla \varphi_n(x))dx\ge \Z W(F),
\]
et le r\'esultat suit en faisant  $\eps\to 0$.
\hfill$\square$

\subsubsection{D\'emonstration du th\'eor\`eme \ref{QuasiconvexificationFormulaTheorem}-bis}
On a toujours $\Z W\geq\mathcal{Q}W$. D'autre part, puisque $\Z W$ est finie, $\Z W$ est continue par la proposition \ref{FonsecaProperties}(c). Utilisant le th\'eor\`eme \ref{QuasiconvexificationFormulaTheorem} on d\'eduit que $\mathcal{Q}\Z W=\Z\Z W$. Or $\Z\Z W=\Z W$ par la proposition \ref{FonsecaProperties}(d), donc $\mathcal{Q} \Z W=\Z W$, i.e., $\Z W$ est quasiconvexe, d'o\`u $\Z W\leq\mathcal{Q} W$.
\hfill$\square$

\begin{remark}
D\'efinissons $\hat \Z W:\MM^{m\times N}\to[0,+\infty]$ par 
\[
\hat\Z W(F):=\inf\left\{\int_YW(F+\nabla\varphi(x):\varphi\in W^{1,\infty}_0(Y;\RR^m)\right\}
\]
(on a toujours $W\geq\Z W\geq\hat \Z W\geq\mathcal{Q} W$). Le r\'esultat suivant est une cons\'equence imm\'ediate du th\'eor\`eme \ref{QuasiconvexificationFormulaTheorem}-bis. 
\begin{corollary}
Si $\Z W$ est finie alors $\mathcal{Q}W=\Z W=\hat \Z W$.
\end{corollary}
\end{remark}

Voici deux exemples qui montrent que le th\'eor\`eme \ref{GeneralThRelax} peut \^etre appliqu\'e \`a des $W$ qui ne satisfont pas (\ref{CondCroiss}) (voir \S 2.2 et \S 2.3).

\subsection{Contrainte produit vectoriel non nul}

On suppose que $W:\MM^{3\times 2}\to[0,+\infty]$ satisfait la condition suivante 
\begin{eqnarray}\label{HypExample:m=N+1}
&&\hbox{il existe }\alpha,\beta>0\hbox{ tels que pour tout }F=(F_1\mid F_2)\in\MM^{3\times 2},\\
&&\hbox{si }|F_1\land F_2|\geq\alpha\hbox{ alors }W(F)\leq\beta(1+|F|^p),\nonumber
\end{eqnarray}
avec $F_1\land F_2$ d\'esigne le produit vectoriel de $F_1$ par $F_2$. (Par exemple, on peut prendre  $W$ donn\'ee par 
$
W(F):=|F|^p+h(|F_1\land F_2|)
$
avec $h:[0,+\infty[\to[0,+\infty]$ une fonction Borel mesurable satisfaisant la propri\'et\'e suivante 
\begin{eqnarray}\label{Example:m=N+1}
\hbox{pour tout }\delta>0\hbox{ il existe }r_\delta>0\hbox{ tel que }h(t)\leq r_\delta\hbox{ pour tout }t\geq\delta
\end{eqnarray}
(par exemple, $h(0)=+\infty$ et $h(t)={1\over t^\alpha}$ si $t>0$ avec $\alpha>0$). Il est clair que $W$ ainsi d\'efinie ne satisfait pas (\ref{CondCroiss}).) Le r\'esultat suivant  a \'et\'e d\'emontr\'e dans \cite{oah-jpm07} (voir aussi \cite{oah-jpm08a}).
\begin{theorem}\label{ThExample:m=N+1}
Si $W$ v\'erifie {(\ref{HypExample:m=N+1})} alors $\Z W$ satisfait {(\ref{Cond-CroissZW})}.
\end{theorem}
\begin{proof}[Sch\'ema de d\'emonstration]
Utilisant la proposition \ref{FonsecaProperties}(a) avec $D\subset\RR^2$ et $\varphi\in\Aff_0(D;\RR^3)$ bien choisis, on montre d'abord que si $W$ v\'erifie {\rm(\ref{HypExample:m=N+1})} alors $\Z W$ satisfait la condition suivante 
\begin{eqnarray}\label{IntermediateCondition}
&&\hbox{il existe }\gamma>0\hbox{ tel que pour tout }F=(F_1\mid F_2)\in\MM^{3\times 2},\\
&&\hbox{si }\min\{|F_1+F_2|,|F_1-F_2|\}\geq \alpha\hbox{ alors }\Z W(F)\leq\gamma(1+|F|^p)\nonumber
\end{eqnarray}
avec $\alpha>0$ donn\'e par (\ref{HypExample:m=N+1}). On prouve ensuite, en utilisant la proposition \ref{FonsecaProperties}(d) avec $D\subset\RR^2$ et $\varphi\in\Aff_0(D;\RR^3)$ bien choisis, que si $\Z W$ v\'erifie (\ref{IntermediateCondition}) alors $\Z W$ satisfait (\ref{Cond-CroissZW}). 
\end{proof}

 Le corollaire suivant est une cons\'equence imm\'ediate des th\'eor\`emes \ref{ThExample:m=N+1} et \ref{GeneralThRelax}.
\begin{corollary}
Si $W$ satisfait {(\ref{HypExample:m=N+1})} alors  {(\ref{FunctRelax})} a lieu avec $\overline{W}=\mathcal{Q}W=\Z W$.
\end{corollary}

\begin{proof}[D\'emonstration du th\'eor\`eme \ref{ThExample:m=N+1}]
On proc\`ede en deux \'etapes.

\subsubsection*{\'Etape 1}{\em Montrons que $\Z W$ satisfait (\ref{IntermediateCondition})}. Soit $F=(F_1\mid F_2)\in\MM^{3\times 2}$ tel que $\min\{|F_1+F_2|,|F_1-F_2|\}\geq\alpha$. Alors, l'une des trois possibilit\'es suivantes a lieu :
\begin{eqnarray}\label{Possibility1Vect}
&&|F_1\land F_2|\not=0\ ;\\ 
&&|F_1\land F_2|=0\hbox{ avec }F_1\not=0\ ;\label{Possibility2Vect}\\
&&|F_1\land F_2|=0\hbox{ avec }F_2\not=0. \label{Possibility3Vect}
\end{eqnarray}
Posons
$
D:=\{(x_1,x_2)\in\RR^2:x_1-1<x_2<x_1+1\hbox{ et }-x_1-1<x_2<1-x_1\}
$ 
et d\'efinissons $\psi\in\Aff_0(D;\RR)$ par 
$$
\psi(x_1,x_2):=\left\{
\begin{array}{ll}
-x_1+(x_2+1)&\hbox{si }(x_1,x_2)\in\Delta_1\\
(1-x_1)-x_2&\hbox{si }(x_1,x_2)\in\Delta_2\\
x_1+(1-x_2)&\hbox{si }(x_1,x_2)\in\Delta_3\\
(x_1+1)+x_2&\hbox{si }(x_1,x_2)\in\Delta_4
\end{array}
\right.
$$
avec :
\begin{itemize} 
\item[]$\Delta_1:=\{(x_1,x_2)\in D:x_1\geq 0\hbox{ et } x_2\leq 0\}$ ;
\item[]$\Delta_2:=\{(x_1,x_2)\in D:x_1\geq 0\hbox{ et }x_2\geq 0\}$ ;
\item[]$\Delta_3:=\{(x_1,x_2)\in D:x_1\leq 0\hbox{ et }x_2\geq 0\}$ ;
\item[]$\Delta_4:=\{(x_1,x_2)\in D:x_1\leq 0\hbox{ et }x_2\leq 0\}$.
\end{itemize}
Consid\'erons $\varphi\in\Aff_0(D;\RR^3)$ donn\'ee par 
$$
\varphi(x):=\psi(x)\nu\hbox{ avec }\left\{\begin{array}{ll}\nu={F_1\land F_2\over|F_1\land F_2|}&\hbox{si on a (\ref{Possibility1Vect})}\\
|\nu|=1 \hbox{ et }\langle F_1,\nu\rangle=0&\hbox{si on a (\ref{Possibility2Vect})}\\
|\nu|=1 \hbox{ et }\langle F_2,\nu\rangle=0&\hbox{si on a (\ref{Possibility3Vect}).}\end{array}\right.
$$
Alors 
$$
F+\nabla\varphi(x)=\left\{
\begin{array}{ll}
(F_1-\nu\mid F_2+\nu)&\hbox{si }x\in{\rm int}(\Delta_1)\\
(F_1-\nu\mid F_2-\nu)&\hbox{si }x\in{\rm int}(\Delta_2)\\
(F_1+\nu\mid F_2-\nu)&\hbox{si }x\in{\rm int}(\Delta_3)\\
(F_1+\nu\mid F_2+\nu)&\hbox{si }x\in{\rm int}(\Delta_4)
\end{array}
\right.
$$
(o\`u ${\rm int}(E)$ d\'esigne l'int\'erieur de l'ensemble $E$). Utilisant la proposition \ref{FonsecaProperties}(a) on d\'eduit que 
\begin{eqnarray}\label{Z_1}
\Z W(F)&\leq&{1\over 4}\left(W(F_1-\nu\mid F_2+\nu)+W(F_1-\nu\mid F_2-\nu)\right.\\
&&\left.+\ W(F_1+\nu\mid F_2-\nu)+W(F_1+\nu\mid F_2+\nu)\right).\nonumber
\end{eqnarray}
Or
$
|(F_1-\nu)\land(F_2+\nu)|^2=|F_1\land F_2+(F_1+F_2)\land\nu|^2
=|F_1\land F_2|^2+|(F_1+F_2)\land\nu|^2
\geq|(F_1+F_2)\land\nu|^2,
$
donc 
$$
|(F_1+\nu)\land(F_2-\nu)|\geq |(F_1+F_2)\land\nu|=|F_1+F_2|.
$$
De la m\^eme fa\c con, on obtient : 
\begin{itemize}
\item[]$|(F_1-\nu)\land(F_2-\nu)|\geq |F_1-F_2|$ ;
\item[]$|(F_1+\nu)\land(F_2-\nu)|\geq |F_1+F_2|$ ;
\item[]$|(F_1+\nu)\land(F_2+\nu)|\geq |F_1-F_2|$.
\end{itemize}
Ainsi $|(F_1-\nu)\land(F_2+\nu)|\geq\alpha$, $|(F_1-\nu)\land(F_2-\nu)|\geq\alpha$, $|(F_1+\nu)\land(F_2-\nu)|\geq\alpha$ et $|(F_1+\nu)\land(F_2+\nu)|\geq\alpha$ car $\min\{|F_1+F_2|,|F_1-F_2|\}\geq\alpha$. Utilisant {\rm(\ref{HypExample:m=N+1})} il suit que 
\begin{eqnarray*}
W(F_1-\nu\mid F_2+\nu)&\leq&\beta(1+|(F_1-\nu\mid F_2+\nu)|^p)\\
&\leq&\beta 2^p(1+|(F_1\mid F_2)|^p+|(-\nu\mid \nu)|^p)\\
&\leq&\beta 2^{2p+1}(1+|F|^p).
\end{eqnarray*}
De la m\^eme mani\`ere, on a : 
\begin{itemize}
\item[]$W(F_1-\nu\mid F_2-\nu)\leq \beta 2^{2p+1}(1+|F|^p)$ ;
\item[]$W(F_1+\nu\mid F_2-\nu)\leq \beta 2^{2p+1}(1+|F|^p)$ ;
\item[]$W(F_1+\nu\mid F_2+\nu)\leq \beta 2^{2p+1}(1+|F|^p)$,
\end{itemize}
et utilisant (\ref{Z_1}) on conclut que $\Z W(F)\leq \beta 2^{2p+1}(1+|F|^p)$.

\subsubsection*{\'Etape 2}{\em Montrons que $\Z W$ v\'erifie (\ref{Cond-CroissZW})}. Soit $F=(F_1\mid F_2)\in\MM^{3\times 2}$. Alors, l'une des quatre possibilit\'es suivantes a lieu :
\begin{eqnarray}
&&|F_1\land F_2|\not=0\ ;\label{Possibility(i)Vect}\\
&&|F_1\land F_2|=0\hbox{ avec }F_1=F_2=0\ ;\label{Possibility(ii)Vect}\\
&&|F_1\land F_2|=0\hbox{ avec }F_1\not=0\ ;\label{Possibility(iii)Vect}\\
&&|F_1\land F_2|=0\hbox{ avec }F_2\not=0. \label{Possibility(iv)Vect}
\end{eqnarray}
D\'efinissons $\psi\in\Aff_0(Y;\RR)$ par 
$$
\psi(x_1,x_2):=\left\{
\begin{array}{ll}
x_2&\hbox{si }(x_1,x_2)\in\Delta_1\\
(1-x_1)&\hbox{si }(x_1,x_2)\in\Delta_2\\
(1-x_2)&\hbox{si }(x_1,x_2)\in\Delta_3\\
x_1&\hbox{si }(x_1,x_2)\in\Delta_4
\end{array}
\right.
$$
avec :
\begin{itemize} 
\item[]$\Delta_1:=\{(x_1,x_2)\in Y:x_2\leq x_1\leq -x_2+1\}$ ;
\item[]$\Delta_2:=\{(x_1,x_2)\in Y:-x_1+1\leq x_2\leq x_1\}$ ;
\item[]$\Delta_3:=\{(x_1,x_2)\in Y:-x_2+1\leq x_1\leq x_2\}$ ;
\item[]$\Delta_4:=\{(x_1,x_2)\in Y:x_1\leq x_2\leq -x_1+1\}$.
\end{itemize}
Consid\'erons $\varphi\in\Aff_0(Y;\RR^3)$ donn\'ee par 
$$
\varphi(x):=\psi(x)\nu\hbox{ avec }
\left\{
\begin{array}{ll}
\nu={\alpha(F_1\land F_2)\over|F_1\land F_2|}&\hbox{si on a (\ref{Possibility(i)Vect})}\\  
|\nu|=\alpha&\hbox{si on a (\ref{Possibility(ii)Vect})}\\
|\nu|=\alpha\hbox{ et }\langle F_1,\nu\rangle=0& \hbox{si on a (\ref{Possibility(iii)Vect})}\\
|\nu|=\alpha\hbox{ et }\langle F_2,\nu\rangle=0& \hbox{si on a (\ref{Possibility(iv)Vect}).}
\end{array}
\right.
$$
Alors 
$$
F+\nabla\varphi(x)=\left\{
\begin{array}{ll}
(F_1\mid F_2+\nu)&\hbox{si }x\in{\rm int}(\Delta_1)\\
(F_1-\nu\mid F_2)&\hbox{si }x\in{\rm int}(\Delta_2)\\
(F_1\mid F_2-\nu)&\hbox{si }x\in{\rm int}(\Delta_3)\\
(F_1+\nu\mid F_2)&\hbox{si }x\in{\rm int}(\Delta_4).
\end{array}
\right.
$$
Utilisant la proposition \ref{FonsecaProperties}(d) on d\'eduit que 
\begin{eqnarray}\label{ZzZ}
\Z W(F)&\leq&{1\over 4}\left(\Z W(F_1\mid F_2+\nu)+\Z W(F_1-\nu\mid F_2)\right.\\
&&\left.+\ \Z W(F_1\mid F_2-\nu)+\Z W(F_1+\nu\mid F_2)\right).\nonumber
\end{eqnarray}
Or
$
|F_1+(F_2+\nu)|^2=|(F_1+F_2)+\nu|^2
=|F_1+F_2|^2+|\nu|^2
=|F_1+F_2|^2+\alpha^2
\geq \alpha^2,
$
d'o\`u
$
|F_1+(F_2+\nu)|\geq \alpha.
$
De la m\^eme fa\c con, on obtient
$
|F_1-(F_2+\nu)|\geq \alpha,
$
donc :
$$
\min\{|F_1+(F_2+\nu)|,|F_1-(F_2+\nu)|\}\geq \alpha.
$$
De la m\^eme mani\`ere, on a : 
\begin{itemize}
\item[]$\min\{|(F_1-\nu)+F_2|,|(F_1-\nu)-F_2|\}\geq\alpha$ ;
\item[]$\min\{|F_1+(F_2-\nu)|,|F_1-(F_2-\nu)|\}\geq\alpha$ ;
\item[]$\min\{|(F_1+\nu)+F_2|,|(F_1+\nu)-F_2|\}\geq\alpha$. 
\end{itemize}
Comme $\Z W$ satisfait (\ref{IntermediateCondition}) il suit que 
\begin{eqnarray*}
\Z W(F_1\mid F_2+\nu)&\leq&\gamma(1+|(F_1\mid F_2+\nu)|^p)\\
&\leq&\gamma 2^p(1+|(F_1\mid F_2)|^p+|(0\mid\nu)|^p)\\
&\leq&\max\{1,\alpha^p\}\gamma2^{p+1}(1+|F|^p).
\end{eqnarray*}
De la m\^eme fa\c con, on obtient :
\begin{itemize}
\item[]$\Z W(F_1-\nu\mid F_2)\leq\max\{1,\alpha^p\}\gamma2^{p+1}(1+|F|^p)$ ;
\item[]$\Z W(F_1\mid F_2-\nu)\leq\max\{1,\alpha^p\}\gamma2^{p+1}(1+|F|^p)$ ;
\item[]$\Z W(F_1+\nu\mid F_2)\leq\max\{1,\alpha^p\}\gamma2^{p+1}(1+|F|^p)$,
\end{itemize} 
et utilisant (\ref{ZzZ}) on conclut que $\Z W(F)\leq\max\{1,\alpha^p\}\gamma2^{p+1}(1+|F|^p)$. 
\end{proof}

\subsection{Contrainte d\'eterminant non nul}

 On suppose que $W:\MM^{3\times 3}\to[0,+\infty]$ satisfait les deux conditions suivantes : 
\begin{eqnarray}\label{Hyp1Example:m=N}
&&\hbox{pour tout }\delta>0\hbox{ il existe }c_\delta>0 \hbox{ tel que pour tout }F\in\MM^{3\times 3},\\
&&\hbox{si }|\det F|\geq\delta\hbox{ alors }W(F)\leq c_\delta(1+|F|^p)\ ;\nonumber\\
&&W(PFQ)=W(F)\hbox{ pour tout }F\in\MM^{3\times 3}\hbox{ et tout }P,Q\in\SO(3) \label{Hyp2Example:m=N}
\end{eqnarray}
avec $\SO(3):=\{Q\in\MM^{3\times 3}:Q^{\rm T}Q=QQ^{\rm T}=I_3\hbox{ et }\det Q=1\}$, o\`u  $I_3$ d\'esigne la matrice identit\'e de $\MM^{3\times 3}$ et $Q^{\rm T}$ est la matrice transpos\'ee de $Q$. (Par exemple, on peut prendre  $W$ donn\'ee par 
$
W(F):=|F|^p+h(|\det F|)
$
avec $h:[0,+\infty[\to[0,+\infty]$ v\'erifiant (\ref{Example:m=N+1}). Noter que $W$ ainsi d\'efinie ne satisfait pas (\ref{CondCroiss}) et est compatible avec (\ref{Cst1prime}) et (\ref{Cst2}).) Le r\'esultat suivant a \'et\'e d\'emontr\'e dans \cite{oah-jpm08a}.
\begin{theorem}\label{ThExample:m=N}
Si $W$ v\'erifie {(\ref{Hyp1Example:m=N})} et {(\ref{Hyp2Example:m=N})} alors $\Z W$ satisfait {(\ref{Cond-CroissZW})}.
\end{theorem}
\begin{proof}[Sch\'ema de d\'emonstration]
Utilisant successivement la proposition \ref{FonsecaProperties}(a) et la proposition \ref{FonsecaProperties}(d) avec $D\subset\RR^3$ et $\varphi\in\Aff_0(D;\RR^3)$ bien choisis, on montre que si $W$ v\'erifie (\ref{Hyp1Example:m=N}) alors $\Z W$ est finie. De la proposition \ref{FonsecaProperties}(c) on d\'eduit que $\Z W$ est continue. Il suit que 
\begin{eqnarray}\label{Cond1:m=N}
&&\Z W(F)\leq c^\prime\hbox{ pour tout }F\in\MM^{3\times 3}\hbox{ tel que }|F|^2\leq 3\hbox{ avec }c^\prime>0.
\end{eqnarray}
De plus (par (\ref{Hyp1Example:m=N}) avec $\delta=1$) il est clair que 
\begin{eqnarray}\label{Cond2:m=N}
&&\Z W(F)\leq c_1(1+|F|^p)\hbox{ pour tout }F\in\MM^{3\times 3}\hbox{ tel que }|\det F|\geq 1\\
&&\hbox{avec }c_1>0.\nonumber
\end{eqnarray}
Utilisant \`a nouveau la proposition \ref{FonsecaProperties}(a) (avec $D\subset\RR^3$ et $\varphi\in\Aff_0(D;\RR^3)$ bien choisis) et la rang-1 convexit\'e de $\Z W$ ($\Z W$ est rang-1 convexe par la proposition \ref{FonsecaProperties}(b) puisque $\Z W$ est finie), on prouve que sous (\ref{Hyp1Example:m=N}) on a 
\begin{eqnarray}\label{Cond3:m=N}
&&\Z W(F)\leq c^{\prime\prime}(1+|F|^p)\hbox{ pour tout }F\in\MM^{3\times 3}\hbox{ tel que }F \hbox{ est diagonale,}\\
&&|\det F|\leq 1\hbox{ et }|F|^2\geq 3\hbox{ avec }c^{\prime\prime}>0.\nonumber
\end{eqnarray}
Combinant (\ref{Cond1:m=N}), (\ref{Cond2:m=N}) et (\ref{Cond3:m=N}), on voit que l'on a d\'emontr\'e que si $W$ v\'erifie (\ref{Hyp1Example:m=N}) alors $\Z W$ satisfait la condition suivante 
\begin{eqnarray}\label{Cond4:m=N}
&&\hbox{il existe }c>0\hbox{ tel que pour tout }F\in\MM^{3\times 3},\\
&&\hbox{si }F\hbox{ est diagonale alors }\Z W(F)\leq c(1+|F|^p).\nonumber
\end{eqnarray}
D'autre part, $\Z W(PFQ)=\Z W(F)$ pour tout $F\in\MM^{3\times 3}$ et tous $P,Q\in\SO(3)$ puisque $W$ v\'erifie (\ref{Hyp2Example:m=N}) et, lorsque $F$ est inversible,  $F=PQ^T\hat F Q$ avec $P,Q\in\SO(3)$ et $\hat F$ diagonale, donc pour toute matrice inversible $F$ il existe une matrice diagonale $\hat F$ telle que $\Z W(F)=\Z W(\hat F)$. Notant que $|F|=|\hat F|$ et utilisant (\ref{Cond4:m=N}) on d\'eduit que pour toute matrice inversible $F$, $\Z W(F)\leq c(1+|F|^p)$. Comme $\Z W$ est continue et l'ensemble des matrices inversibles est dense dans $\MM^{3\times 3}$, il suit que $\Z W$ satisfait (\ref{Cond-CroissZW}). 
\end{proof}

 Le corollaire suivant est une cons\'equence imm\'ediate des th\'eor\`emes \ref{ThExample:m=N} et \ref{GeneralThRelax}.
\begin{corollary}\label{CorollaryExample:m=N}
Si $W$ satisfait {(\ref{Hyp1Example:m=N})} et {(\ref{Hyp2Example:m=N})} alors {(\ref{FunctRelax})} a lieu avec $\overline{W}=\mathcal{Q}W=\Z W$.
\end{corollary}

\begin{proof}[D\'emonstration du th\'eor\`eme \ref{ThExample:m=N}]
On proc\`ede en quatre \'etapes.
\subsubsection*{\'Etape 1}
{\em Montrons que $\Z W$ est finie}. Il est clair que $\Z W(F)<+\infty$ pour tout $F\in\MM^{3\times 3}_*$  avec $\MM^{3\times 3}_*:=\{F\in\MM^{3\times 3}:\det F\not=0\}$. Donc, tout revient \`a  prouver que  $\Z W(F)<+\infty$ pour tout $F\in\MM^{3\times 3}\setminus\MM^{3\times 3}_*$. Fixons $F=(F_1\mid F_2\mid F_3)\in\MM^{3\times 3}\setminus\MM^{3\times 3}_*$ o\`u $F_1,F_2,F_3\in\RR^3$ sont les colonnes de la matrice $F$. Alors, $\rank(F)\in\{0,1,2\}$ (o\`u $\rank(F)$ d\'esigne le rang de la matrice $F$).

\subsubsection*{\'Etape 1.1}
{\em Montrons que si $\rank(F)=2$ alors $\Z W(F)<+\infty$}. Sans perdre de g\'en\'eralit\'e on peut supposer qu'il existe $\lambda,\mu\in\RR$ tels que $F_3=\lambda F_1+\mu F_2$.  \'Etant donn\'e $s\in\RR^*$, consid\'erons $D\subset\RR^3$ (voir la figure \ref{DessinDuDomaineD}) donn\'e par 
\begin{equation}\label{BoundedOpenSet}
D:={\rm int}(\cup_{i=1}^8\Delta_{i}^s)
\end{equation}
(o\`u ${\rm int}(E)$ d\'esigne l'int\'erieur de l'ensemble $E$) avec :
\begin{itemize}
\item[]$\Delta_{1}^s:=\{(x_1,x_2,x_3)\in \RR^3:x_1\geq 0,\;x_2\geq 0,\;x_3\geq 0 \hbox{ et }x_1+x_2+sx_3\leq 1\}$ ;
\item[]$\Delta_{2}^s:=\{(x_1,x_2,x_3)\in \RR^3:x_1\leq 0,\;x_2\geq 0,\;x_3\geq 0 \hbox{ et }-x_1+x_2+sx_3\leq 1\}$ ;
\item[]$\Delta_{3}^s:=\{(x_1,x_2,x_3)\in \RR^3:x_1\leq 0,\;x_2\leq 0,\;x_3\geq 0 \hbox{ et }-x_1-x_2+sx_3\leq 1\}$ ;
\item[]$\Delta_{4}^s:=\{(x_1,x_2,x_3)\in \RR^3:x_1\geq 0,\;x_2\leq 0,\;x_3\geq 0 \hbox{ et }x_1-x_2+sx_3\leq 1\}$ ;
\item[]$\Delta_{5}^s:=\{(x_1,x_2,x_3)\in\RR^3:x_1\geq 0,\;x_2\geq 0,\;x_3\leq 0 \hbox{ et }x_1+x_2-sx_3\leq 1\}$ ;
\item[]$\Delta_{6}^s:=\{(x_1,x_2,x_3)\in \RR^3:x_1\leq 0,\;x_2\geq 0,\;x_3\leq 0 \hbox{ et }-x_1+x_2-sx_3\leq 1\}$ ;
\item[]$\Delta_{7}^s:=\{(x_1,x_2,x_3)\in \RR^3:x_1\leq 0,\;x_2\leq 0,\;x_3\leq 0 \hbox{ et }-x_1-x_2-sx_3\leq 1\}$ ;
\item[]$\Delta_{8}^s:=\{(x_1,x_2,x_3)\in \RR^3:x_1\geq 0,\;x_2\leq 0,\;x_3\leq 0 \hbox{ et }x_1-x_2-sx_3\leq 1\}$.
 \end{itemize}
(Il est clair que $D$ est ouvert, born\'e et $|\partial D|=0$.) 

\begin{figure}[H]
\begin{center}
\scalebox{0.8} % Change this value to rescale the drawing.
{
\begin{pspicture}(0,-3.55)(11.42,3.57)
\psline[linewidth=0.04cm,arrowsize=0.05291667cm 2.0,arrowlength=1.4,arrowinset=0.4]{->}(5.98,0.07)(3.58,-2.33)
\psline[linewidth=0.04cm,arrowsize=0.05291667cm 2.0,arrowlength=1.4,arrowinset=0.4]{->}(5.98,0.07)(5.98,3.27)
\psline[linewidth=0.04cm,linestyle=dotted,dotsep=0.16cm](5.98,0.07)(5.98,-3.53)
\psline[linewidth=0.04cm,linestyle=dotted,dotsep=0.16cm](5.98,0.07)(2.38,0.07)
\psline[linewidth=0.04cm](4.38,-1.53)(8.38,0.07)
\psline[linewidth=0.04cm](5.98,2.47)(8.38,0.07)
\psline[linewidth=0.04cm](5.98,2.47)(4.38,-1.53)
\psline[linewidth=0.04cm](4.38,-1.53)(3.58,0.07)
\psline[linewidth=0.04cm](3.58,0.07)(5.98,2.47)
\psline[linewidth=0.04cm](5.98,-2.33)(8.38,0.07)
\psline[linewidth=0.04cm](4.38,-1.53)(5.98,-2.33)
\psline[linewidth=0.04cm,linestyle=dotted,dotsep=0.16cm](5.98,0.07)(8.38,2.47)
\psline[linewidth=0.04cm,linestyle=dashed,dash=0.16cm 0.16cm](3.58,0.07)(7.58,1.67)
\psline[linewidth=0.04cm,linestyle=dashed,dash=0.16cm 0.16cm](3.58,0.07)(5.98,-2.33)
\psline[linewidth=0.04cm,linestyle=dashed,dash=0.16cm 0.16cm](7.58,1.67)(5.98,-2.33)
\psline[linewidth=0.04cm](5.98,2.47)(7.58,1.67)
\psline[linewidth=0.04cm](7.58,1.67)(8.38,0.07)
\psline[linewidth=0.04cm,arrowsize=0.05291667cm 2.0,arrowlength=1.4,arrowinset=0.4]{->}(5.98,0.07)(9.18,0.07)
\psline[linewidth=0.02,arrowsize=0.05291667cm 2.0,arrowlength=1.4,arrowinset=0.4]{<-}(6.38,-0.33)(7.58,-1.53)(8.78,-1.53)(8.78,-1.53)
\psline[linewidth=0.02,arrowsize=0.05291667cm 2.0,arrowlength=1.4,arrowinset=0.4]{<-}(7.58,1.27)(8.78,1.27)(9.18,1.27)
\psline[linewidth=0.02,arrowsize=0.05291667cm 2.0,arrowlength=1.4,arrowinset=0.4]{<-}(5.18,1.27)(3.58,1.27)(3.58,1.27)
\psline[linewidth=0.02,arrowsize=0.05291667cm 2.0,arrowlength=1.4,arrowinset=0.4]{<-}(6.38,1.67)(6.38,2.87)(7.18,2.87)(7.18,2.87)
\usefont{T1}{ptm}{m}{n}
\rput(3.55,-2.625){$x_1$}
\usefont{T1}{ptm}{m}{n}
\rput(9.5,0.16){$x_2$}
\usefont{T1}{ptm}{m}{n}
\rput(6.01,3.45){$x_3$}
\psline[linewidth=0.02,arrowsize=0.05291667cm 2.0,arrowlength=1.4,arrowinset=0.4]{<-}(4.38,-0.33)(2.78,-0.73)(1.98,-0.73)(1.98,-0.73)
\psline[linewidth=0.02,arrowsize=0.05291667cm 2.0,arrowlength=1.4,arrowinset=0.4]{<-}(4.78,-1.53)(4.78,-3.13)(5.18,-3.13)
\psline[linewidth=0.02,arrowsize=0.05291667cm 2.0,arrowlength=1.4,arrowinset=0.4]{<-}(6.38,-1.53)(7.18,-2.33)(7.98,-2.33)(7.98,-2.33)
\psline[linewidth=0.02,arrowsize=0.05291667cm 2.0,arrowlength=1.4,arrowinset=0.4]{<-}(7.58,-0.33)(8.78,-0.73)(9.58,-0.73)(9.58,-0.73)
\usefont{T1}{ptm}{m}{n}
\rput(7.47,2.975){$\Delta_3^s$}
\usefont{T1}{ptm}{m}{n}
\rput(9.47,1.375){$\Delta_2^s$}
\usefont{T1}{ptm}{m}{n}
\rput(9.87,-0.625){$\Delta_6^s$}
\usefont{T1}{ptm}{m}{n}
\rput(9.07,-1.425){$\Delta_1^s$}
\usefont{T1}{ptm}{m}{n}
\rput(8.27,-2.225){$\Delta_5^s$}
\usefont{T1}{ptm}{m}{n}
\rput(5.47,-3.025){$\Delta_8^s$}
\usefont{T1}{ptm}{m}{n}
\rput(3.285,1.375){$\Delta_4^s$}
\usefont{T1}{ptm}{m}{n}
\rput(1.67,-0.625){$\Delta_7^s$}
\usefont{T1}{ptm}{m}{n}
\rput(5.79,2.7){$\frac{1}{s}$}
\usefont{T1}{ptm}{m}{n}
\rput(6.28,-2.6){$-\frac{1}{s}$}
\end{pspicture} 
}
\end{center}
\caption{Les ensembles $\Delta^s_1$, $\Delta^s_2$, $\Delta^s_3$, $\Delta^s_4$, $\Delta^s_5$, $\Delta^s_6$, $\Delta^s_7$, $\Delta^s_8$. }
\label{DessinDuDomaineD}
\end{figure}
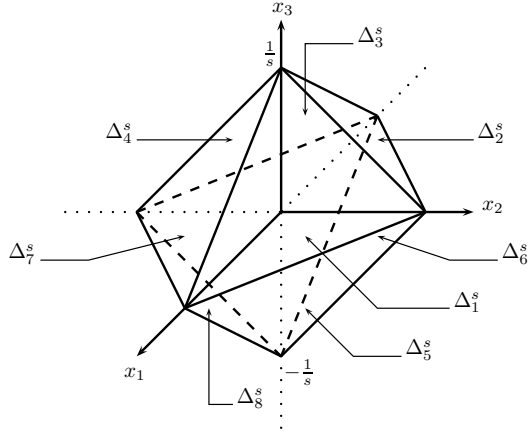

D\'efinissons $\psi_{s}\in\Aff_0(D;\RR)$ par 
\begin{equation}\label{AffineFunction}
\psi_{s}(x_1,x_2,x_3):=\left\{
\begin{array}{ll}
-(x_1-1)-x_2-sx_3&\hbox{si }(x_1,x_2,x_3)\in\Delta_{1}^s\\
x_1-(x_2-1)-sx_3&\hbox{si }(x_1,x_2,x_3)\in\Delta_{2}^s\\
(x_1+1)+x_2-sx_3&\hbox{si }(x_1,x_2,x_3)\in\Delta_{3}^s\\
-x_1+(x_2+1)-sx_3&\hbox{si }(x_1,x_2,x_3)\in\Delta_{4}^s\\
-(x_1-1)-x_2+sx_3&\hbox{si }(x_1,x_2,x_3)\in\Delta_{5}^s\\
x_1-(x_2-1)+sx_3&\hbox{si }(x_1,x_2,x_3)\in\Delta_{6}^s\\
(x_1+1)+x_2+sx_3&\hbox{si }(x_1,x_2,x_3)\in\Delta_{7}^s\\
-x_1+(x_2+1)+sx_3&\hbox{si }(x_1,x_2,x_3)\in\Delta_{8}^s.
\end{array}
\right.
\end{equation}
(Noter que $\psi_{s}$ est l'unique fonction continue, affine sur chaque $\Delta_i^s$ et nulle sur $\partial D$ telle que $\psi_{s}(0)=1$.) Fixons $s\in\RR^*\setminus\{\lambda-\mu,-(\lambda-\mu),\lambda+\mu,-(\lambda+\mu)\}$ et consid\'erons $\varphi\in\Aff_0(D;\RR^3)$ donn\'ee par 
$$
\varphi(x):=\psi_s(x)\nu\hbox{ avec }\nu:={F_1\land F_2\over| F_1\land F_2|^2}.
$$
Alors 
$$
 F+\nabla\varphi(x)=\left\{
\begin{array}{ll}
( F_1-\nu\mid F_2-\nu\mid  F_3-s\nu)&\hbox{si }x\in{\rm int}(\Delta_{1}^s)\\
( F_1+\nu\mid  F_2-\nu\mid  F_3-s\nu)&\hbox{si }x\in {\rm int}(\Delta_{2}^s)\\
( F_1+\nu\mid  F_2+\nu\mid  F_3-s\nu)&\hbox{si }x\in {\rm int}(\Delta_{3}^s)\\
( F_1-\nu\mid  F_2+\nu\mid  F_3-s\nu)&\hbox{si }x\in {\rm int}(\Delta_{4}^s)\\
( F_1-\nu\mid F_2-\nu\mid  F_3+s\nu)&\hbox{si }x\in {\rm int}(\Delta_{5}^s)\\
( F_1+\nu\mid  F_2-\nu\mid  F_3+s\nu)&\hbox{si }x\in {\rm int}(\Delta_{6}^s)\\
( F_1+\nu\mid  F_2+\nu\mid  F_3+s\nu)&\hbox{si }x\in {\rm int}(\Delta_{7}^s)\\
( F_1-\nu\mid  F_2+\nu\mid  F_3+s\nu)&\hbox{si }x\in {\rm int}(\Delta_{8}^s).
\end{array}
\right.
$$
Comme $\det F=0$, $ F_1\land F_3=\mu( F_1\land F_2)$ et $ F_2\land F_3=\lambda( F_2\land F_1)$  on a 
$$
|\det( F+\nabla\varphi(x))|=\left\{
\begin{array}{ll}
|s-(\lambda+\mu)|&\hbox{si }x\in{\rm int}(\Delta_{1}^s)\cup{\rm int}(\Delta_{7}^s)\\
|s+(\lambda-\mu)|&\hbox{si }x\in{\rm int}(\Delta_{2}^s)\cup{\rm int}(\Delta_{8}^s)\\
|s+(\lambda+\mu)|&\hbox{si }x\in{\rm int}(\Delta_{3}^s)\cup{\rm int}(\Delta_{5}^s)\\
|s-(\lambda-\mu)|&\hbox{si }x\in{\rm int}(\Delta_{4}^s)\cup{\rm int}(\Delta_{6}^s).\\
\end{array}
\right.
$$
(Pour les calculs  noter que $\det(F_1+\alpha\nu\mid F_2+\beta\nu\mid F_3+\gamma\nu)=\gamma-\alpha\lambda-\beta\mu$ avec $\alpha\in\{-1,1\}$, $\beta\in\{-1,1\}$ et $\gamma\in\{-s,s\}$.) Il suit que pour  presque tout $x\in D$,
$
|\det( F+\nabla\varphi(x)|\geq\min\{|s-(\lambda+\mu)|,|s+(\lambda-\mu)|,|s+(\lambda+\mu)|,|s-(\lambda-\mu)|\}=:\delta
$
($\delta>0$). Utilisant la proposition \ref{FonsecaProperties}(a) et {(\ref{Hyp1Example:m=N})} on d\'eduit qu'il existe $c_\delta>0$ tel que  
$$
\Z W( F)\leq{1\over|D|}\int_DW( F+\nabla\varphi(x))\leq c_\delta+{c_\delta\over|D|}\| F+\nabla \varphi\|^p_{L^p(D;\RR^3)},
$$
d'o\`u $\Z W( F)<+\infty$.

\subsubsection*{\'Etape 1.2}
{\em Montrons que si $\rank(F)=1$ alors $\Z W(F)<+\infty$}. Sans perdre de g\'en\'eralit\'e on peut supposer qu'il existe $\lambda,\mu\in\RR$ tels que $ F_2=\lambda F_1$ and $ F_3=\mu F_1$. Consid\'erons $D\subset\RR^3$ donn\'e par (\ref{BoundedOpenSet}) avec $s\in\RR^*\setminus\{-\mu,\mu\}$ et d\'efinissons $\varphi\in\Aff_0(D;\RR^3)$ par $\varphi(x):=\psi_s(x)\nu$ avec $\nu\in\RR^3\setminus\{0\}$ tel que $\langle\nu, F_1\rangle=0$, o\`u $\psi_{s}$ est d\'efinie par (\ref{AffineFunction}). Utilisant la  proposition \ref{FonsecaProperties}(d) on a 
\begin{eqnarray*}
\Z W( F)&\leq& {1\over 8}\left(\Z W( F_1-\nu\mid F_2-\nu\mid  F_3-s\nu)+\Z W( F_1+\nu\mid F_2-\nu\mid  F_3-s\nu)\right.\\
&&+\ \Z W( F_1+\nu\mid F_2+\nu\mid  F_3-s\nu)+\Z W( F_1-\nu\mid F_2+\nu\mid  F_3-s\nu)\\
&&+\ \Z W( F_1-\nu\mid F_2-\nu\mid  F_3+s\nu)+\Z W( F_1+\nu\mid F_2-\nu\mid  F_3+s\nu)\\
&&+\ \Z W( F_1+\nu\mid F_2+\nu\mid  F_3+s\nu)+\Z W( F_1-\nu\mid F_2+\nu\mid  F_3+s\nu)).
\end{eqnarray*}
Notant que $s\in\RR^*\setminus\{-\mu,\mu\}$ il est facile de voir que :
\begin{itemize}
\item[]$\rank( F_1-\nu\mid F_2-\nu\mid  F_3-s\nu)=2$ ;
\item[]$\rank( F_1+\nu\mid F_2-\nu\mid  F_3-s\nu)=2$ ;
\item[]$\rank( F_1+\nu\mid F_2+\nu\mid  F_3-s\nu)=2$ ;
\item[]$\rank( F_1-\nu\mid F_2+\nu\mid  F_3-s\nu)=2$ ;
\item[]$\rank( F_1-\nu\mid F_2-\nu\mid  F_3+s\nu)=2$ ;
\item[]$\rank( F_1+\nu\mid F_2-\nu\mid  F_3+s\nu)=2$ ;
\item[]$\rank( F_1+\nu\mid F_2+\nu\mid  F_3+s\nu)=2$ ;
\item[]$\rank( F_1-\nu\mid F_2+\nu\mid  F_3+s\nu)=2$,
\end{itemize}
et utilisant  l'\'etape 1.1 on d\'eduit que $\Z W( F)<+\infty$.

\subsubsection*{\'Etape 1.3}
{\em Montrons que $\Z W(0)<+\infty$}.  Utilisant la    proposition \ref{FonsecaProperties}(d) avec $D\subset\RR^3$ donn\'e par (\ref{BoundedOpenSet}) avec $s\in\RR^*$ et $\varphi\in\Aff_0(D;\RR^3)$ d\'efinie par $\varphi(x):=\psi_s(x)\nu$ avec $\nu\in\RR^3\setminus\{0\}$, o\`u $\psi_{s}$ est d\'efinie par (\ref{AffineFunction}), on a 
\begin{eqnarray*}
\Z W( 0)&\leq& {1\over 8}\left(\Z W(-\nu\mid -\nu\mid -s\nu)+\Z W(\nu\mid -\nu\mid  -s\nu)+\Z W(\nu\mid \nu\mid  -s\nu)\right.\\
&&+\ \Z W(-\nu\mid \nu\mid -s\nu)+\ \Z W(-\nu\mid -\nu\mid  s\nu)+\Z W(\nu\mid -\nu\mid  s\nu)\\
&&+\ \Z W(\nu\mid \nu\mid  s\nu)+\Z W(-\nu\mid \nu\mid  s\nu)).
\end{eqnarray*}
De plus, $\rank(-\nu\mid -\nu\mid -s\nu)=\rank(\nu\mid -\nu\mid  -s\nu)=\rank(\nu\mid \nu\mid  -s\nu)=\rank(-\nu\mid \nu\mid -s\nu)=\rank(-\nu\mid -\nu\mid  s\nu)=\rank(\nu\mid -\nu\mid  s\nu)=\rank(\nu\mid \nu\mid  s\nu)=\rank(-\nu\mid \nu\mid  s\nu)=1$, d'o\`u $\Z W(0)<+\infty$ par l'\'etape 1.2.

\subsubsection*{\'Etape 2}
{\em Montrons que $\Z W$ satisfait (\ref{Cond4:m=N})}. Par l'\'etape 1 et la proposition \ref{FonsecaProperties}(b) on d\'eduit que  $\Z W$ est continue, donc $\Z W$ satisfait (\ref{Cond1:m=N}). De plus, il est clair que $\Z W$ v\'erifie (\ref{Cond2:m=N}). Consid\'erons donc $ F\in\MM^{3\times 3}$ tel que $ F$ est diagonale, $|\det F|\leq 1$ et $| F|^2\geq 3$, i.e., $ F=( F_{ij})$ avec $ F_{ij}=0$ si $i\not=j$, $| F_{11} F_{22} F_{33}|\leq 1$ et $| F_{11}|^2+| F_{22}|^2+| F_{33}|^2\geq 3$. Alors, l'une des six possibilit\'es suivantes a lieu :
\begin{eqnarray}
&&| F_{11}|\leq 1,\ | F_{22}|\geq 1 \hbox{ et } | F_{33}|\geq 1\ ;\label{Possible1Det}\\
&&| F_{22}|\leq 1,\  | F_{33}|\geq 1 \hbox{ et }| F_{11}|\geq 1\ ;\label{Possible2Det}\\
&&| F_{33}|\leq 1,\  | F_{11}|\geq 1 \hbox{ et }| F_{22}|\geq 1\ ;\label{Possible3Det}\\
&&| F_{11}|\geq 1,\  | F_{22}|\leq 1 \hbox{ et } | F_{33}|\leq 1\ ;\label{Possible4Det}\\
&&| F_{22}|\geq 1,\  | F_{33}|\leq 1 \hbox{ et } | F_{11}|\leq 1\ ;\label{Possible5Det}\\
&&| F_{33}|\geq 1,\  | F_{11}|\leq 1 \hbox{ et } | F_{22}|\leq 1. \label{Possible6Det}
\end{eqnarray}

\subsubsection*{\'Etape 2.1}
{\em Montrons qu'il existe $c_2>0$ tel que si $ F$ est diagonale avec $|\det F|\leq 1$ et satisfait (\ref{Possible1Det}),  (\ref{Possible2Det}) ou  (\ref{Possible3Det}) alors $\Z W( F)\leq c_2(1+| F|^p)$.} Soient $D\subset\RR^3$ donn\'e par (\ref{BoundedOpenSet}) avec $s=1$ et $\varphi\in\Aff_0(D;\RR^3)$ d\'efinie par $\varphi(x):=\psi_1(x)\nu$, o\`u 
$$
\nu=\left\{
\begin{array}{ll}
(2,0,0)&\hbox{si on a (\ref{Possible1Det})}\\
(0,2,0)&\hbox{si on a (\ref{Possible2Det})}\\
(0,0,2)&\hbox{si on a (\ref{Possible3Det})}
\end{array}
\right.
$$
et $\psi_{1}$ est d\'efinie par (\ref{AffineFunction}) avec $s=1$. Il est facile de voir que pour presque tout $x\in D$,
$$
|\det( F+\nabla\varphi(x))|\geq\left\{
\begin{array}{ll}
|2| F_{22}|| F_{33}|-|\det F||&\hbox{si on a (\ref{Possible1Det})}\\
|2| F_{11}|| F_{33}|-|\det F||&\hbox{si on a (\ref{Possible2Det})}\\
|2| F_{11}|| F_{22}|-|\det F||&\hbox{si on a (\ref{Possible3Det}),}
\end{array}
\right.
$$
donc  $|\det( F+\nabla\varphi(x))|\geq 1$ p.p. dans $D$. Utilisant la proposition \ref{FonsecaProperties}(a) et (\ref{Hyp1Example:m=N}) et notant que  $|\nabla\varphi(x)|=2\sqrt{3}$ pour presque tout $x\in D$, on d\'eduit que 
$
\Z W( F)\leq c_2(1+| F|^p)
$
avec $c_2:=c_12^p(1+(2\sqrt{3})^p)$.

\subsubsection*{\'Etape 2.2}
{\em Montrons qu'il existe $c_3>0$ tel que si $ F$ est diagonale et satisfait (\ref{Possible4Det}), (\ref{Possible5Det}) ou (\ref{Possible6Det}), alors $\Z W( F)\leq c_3(1+| F|^p)$.} Soit $G\in\MM^{3\times 3}$ la matrice diagonale  donn\'ee par :
$$
G_{11}:=\left\{\begin{array}{ll} F_{22}+{\rm sign( F_{22})}&\hbox{si on a (\ref{Possible4Det})}\\0&\hbox{si on a (\ref{Possible5Det}) ou (\ref{Possible6Det}) ;}\end{array}\right.
$$
$$
G_{22}:=\left\{\begin{array}{ll} F_{33}+{\rm sign( F_{33})}&\hbox{si on a   (\ref{Possible5Det})}\\0&\hbox{si on a (\ref{Possible4Det}) ou (\ref{Possible6Det}) ;}\end{array}\right.
$$
$$
G_{33}:=\left\{\begin{array}{ll} F_{11}+{\rm sign( F_{11})}&\hbox{si on a (\ref{Possible6Det})}\\0&\hbox{si on a (\ref{Possible4Det}) ou (\ref{Possible5Det}),}\end{array}\right.
$$
o\`u ${\rm sign}(r)=1$ si $r\geq 0$ et ${\rm sign}(r)=-1$ si $r<0$. (Il est clair que $\rank(G)=1$.) Alors, $ F^+:= F+G$ et $ F^{-}:= F-G$ sont des matrices diagonales telles que : 
\begin{eqnarray}
&&\left\{\begin{array}{l}
| F^+_{11}|\geq 1, | F^+_{22}|\geq 1, | F^+_{33}|\leq 1\\| F^-_{11}|\geq 1, | F^-_{22}|\geq 1, | F^-_{33}|\leq 1 
\end{array}\right.
\hbox{ si on a (\ref{Possible4Det}) ;}\label{DiaG1}\\
&&
\left\{\begin{array}{l}
| F^+_{11}|\leq 1, | F^+_{22}|\geq 1, | F^+_{33}|\geq 1\\
| F^-_{11}|\leq 1, | F^-_{22}|\geq 1, | F^-_{33}|\geq 1 
\end{array}\right.
\hbox{ si on a (\ref{Possible5Det}) ;}\label{DiaG2}\\
&&
\left\{\begin{array}{l}
| F^+_{11}|\geq 1, | F^+_{22}|\leq 1, | F^+_{33}|\geq 1\\
| F^-_{11}|\geq 1, | F^-_{22}|\leq 1, | F^-_{33}|\geq 1
\end{array}\right.
\hbox{ si on a (\ref{Possible6Det}).}\label{DiaG3}
\end{eqnarray}
Comme $\Z W$ est finie (voir l'\'etape 1), de la proposition \ref{FonsecaProperties}(b) on d\'eduit que $\Z W$ est rang-1 convexe, d'o\`u 
\begin{equation}\label{Rank-OneConvexity}
\Z W( F)\leq {1\over 2}\left(\Z W( F^+)+\Z W( F^-)\right).
\end{equation}
Prenant en compte (\ref{DiaG1}), (\ref{DiaG2}) et (\ref{DiaG3}) et utilisant l'\'etape 2.1 on voit que $\Z W( F^+)\leq c_2(1+| F^+|^p)$ si $|\det  F^+|\leq 1$ (resp. $\Z W( F^-)\leq c_2(1+| F^+|^p)$ si $|\det  F^-|\leq 1$). D'autre part,  par (\ref{Hyp1Example:m=N}) on a  $\Z W( F^+)\leq c_1(1+| F^+|^p)$ si $|\det  F^+|\geq 1$ (resp. $\Z W( F^-)\leq c_1(1+| F^+|^p)$ si $|\det  F^-|\geq 1$). Notant que $| F^+|^p\leq 2^{2p}(1+| F|^p)$ (resp. $| F^-|^p\leq 2^{2p}(1+| F|^p)$) et utilisant (\ref{Rank-OneConvexity}) on d\'eduit que $\Z W( F)\leq c_3(1+| F|^p)$ avec $c_3:=2^{2p}\max\{c_1,c_2\}$.

\subsubsection*{\'Etape 2.3}
Des \'etapes 2.1 et 2.2, il suit que pour chaque $ F\in\MM^{3\times 3}$, si $ F$ est diagonale avec $| F|^2\geq 3$ et $|\det F|\leq 1$ alors $\Z W( F)\leq c_4(1+| F|^p)$ avec $c_4:=\max\{c_2,c_3\}$. D'o\`u (\ref{Cond4:m=N}) a lieu avec $c:=\max\{c_0,c_4\}$.

\subsubsection*{\'Etape 3}
{\em Montrons que $\Z W(P F Q)=\Z W( F)$ pour tout $ F\in\MM^{3\times 3}$ et tout $P, Q\in\SO(3)$.} Il suffit de prouver que 
\begin{eqnarray}
&&\Z W(P F Q)\leq \Z W( F)\hbox{ pour tout } F\in\MM^{3\times 3}\hbox{ et tout }P,Q\in\SO(3).\label{StEp3DeT}
\end{eqnarray}
En effet, \'etant donn\'es $ F\in\MM^{3\times 3}$ et $P,Q\in\SO(3)$, on a  $ F=P^{\rm T}(P F Q)Q^{\rm T}$ (avec $M^{\rm T}$ d\'esignant la matrice transpos\'ee de $M$) et utilisant (\ref{StEp3DeT}) on obtient $\Z W( F)\leq\Z W(P F Q)$. De plus, (\ref{StEp3DeT}) est \'equivalent aux deux assertions suivantes : 
\begin{eqnarray}
&&\Z W(P F)\leq \Z W( F)\hbox{ pour tout } F\in\MM^{3\times 3}\hbox{ et tout }P\in\SO(3)\ ;\label{StEp3DeT1}\\
&&\Z W( F Q)\leq \Z W( F)\hbox{ pour tout } F\in\MM^{3\times 3}\hbox{ et tout }Q\in\SO(3).\label{StEp3DeT2}
\end{eqnarray} 
En effet, (\ref{StEp3DeT1}) (resp. (\ref{StEp3DeT2})) suit de  (\ref{StEp3DeT}) en prenant $Q=I_3$ (resp. $P=I_3$)  (o\`u $I_3$ est la matrice identit\'e dans $\MM^{3\times 3}$). D'autre part, \'etant donn\'es $ F\in\MM^{3\times 3}$ et  $P,Q\in\SO(3)$,  par (\ref{StEp3DeT1}) (resp. (\ref{StEp3DeT2})) on a  $\Z W(P( F Q))\leq\Z W( F Q)$ (resp. $\Z W( F Q)\leq\Z W( F)$), d'o\`u $\Z W(P F Q)\leq\Z W( F)$. Tout revient donc \`a prouver (\ref{StEp3DeT1})  et (\ref{StEp3DeT2}). 

\subsubsection*{\'Etape 3.1}{\em Montrons (\ref{StEp3DeT1})}.
Soient $F\in\MM^{3\times 3}$ et $P\in\SO(3)$. Consid\'erons $\varphi\in \Aff_0(Y;\RR^3)$ et posons $\phi:=P\varphi$. Alors $\phi\in \Aff_0(Y;\RR^3)$ et $\nabla\phi=P\nabla\varphi$, donc 
\begin{eqnarray*}
\Z W(P F)\leq\int_Y W(P F+\nabla\phi(x))dx&=&\int_Y W\big(P( F+P^{\rm T}\nabla\phi(x))\big)dx\\
&=&\int_YW\big(P( F+\nabla\varphi(x))\big)dx.
\end{eqnarray*}
De (\ref{Hyp2Example:m=N}) on d\'eduit que 
$
\Z W(P F)\leq\int_YW( F+\nabla\varphi(x))dx
$
pour tout $\varphi\in\Aff_0(Y;\RR^3)$, d'o\`u
$
\Z W(P F)\le \Z W( F).
$

\subsubsection*{\'Etape 3.2}{\em Montrons (\ref{StEp3DeT2})}.
Soient $F\in\MM^{3\times 3}$ et $Q\in\SO(3)$. Par le th\'eor\`eme de recouvrement de Vitali, il existe une famille au plus d\'enombrable $\{a_i+\eps_i Q^{\rm T}Y\}_{i\in I}$ de sous-ensembles disjoints de $Y$, o\`u $a_i\in\RR^3$ et $0<\eps_i<1$, telle que $|Y\setminus\cup_{i\in I}(a_i+\eps_i Q^{\rm T}Y)|=0$ (donc $\sum_{i\in I}\eps_i^3=1$). Consid\'erons $\varphi\in \Aff_0(Y;\RR^3)$ et d\'efinissons $\phi\in\Aff_0(Y;\RR^3)$ par 
$$
\phi(x)=\eps_i\varphi\left(Q{{x-a_i}\over{\eps_i}}\right)\hbox{ si } x\in a_i+\eps_iQ^{\rm T}Y.
$$
Alors 
\begin{eqnarray*}
\Z W( F Q)\leq\int_Y W( F Q+\nabla\phi(x))dx&=&\sum_{i\in I}\eps_i^3\int_{Y} W\left( F Q+\nabla\varphi\left(x\right)Q\right)dx\\
&=&\int_YW\big(( F+\nabla\varphi(x))Q\big)dx.
\end{eqnarray*}
De (\ref{Hyp2Example:m=N}) on d\'eduit que 
$
\Z W( F Q)\leq\int_YW( F+\nabla\varphi(x))dx
$
pour tout $\varphi\in\Aff_0(Y;\RR^3)$, d'o\`u $\Z W( F Q)\le \Z W( F)$.

\subsubsection*{\'Etape 4}{\em Montrons que $\Z W$ satisfait (\ref{Cond-CroissZW})}.
Soient $ F\in\MM^{3\times 3}_*$ (avec $\MM^{3\times 3}_*:=\{ F\in\MM^{3\times 3}:\det F\not=0\}$) et $P\in\SO(3)$ donn\'e par $P:= F M^{-1}$ avec 
$$
M:=\left\{
\begin{array}{ll}
\sqrt{ F^{\rm T} F}&\hbox{si }\det F>0\\
-\sqrt{ F^{\rm T} F}&\hbox{si }\det F<0.
\end{array}
\right.
$$
Comme $M$ est sym\'etrique, il existe $Q\in\SO(3)$ et  $G\in\MM^{3\times 3}$ tels que $G$ est diagonale et $M=Q^{\rm T}G Q$, donc
$
 F=PQ^{\rm T}G Q.
$
D'o\`u $\Z W( F)=\Z W(G)$ par l'\'etape 3. Notant que $|G|=| F|$ et utilisant (\ref{Cond4:m=N}) (voir l'\'etape 2) on d\'eduit que $\Z W( F)\leq c(1+| F|^p)$ pour tout $ F\in\MM^{3\times 3}_*$. Comme $\Z W$ est continue (par l'\'etape 1 et la proposition \ref{FonsecaProperties}(c)) et $\MM^{3\times 3}_*$ est dense dans $\MM^{3\times 3}$, il suit que $\Z W$ v\'erifie (\ref{Cond-CroissZW}). 
\end{proof}

On peut g\'en\'eraliser le corollaire \ref{CorollaryExample:m=N} comme suit. (En fait, on peut supprimer la condition (\ref{Hyp2Example:m=N}) dans le corollaire \ref{CorollaryExample:m=N}.)

\begin{corollary}\label{CorollaryExample:m=NBis}
Si $W:\MM^{N\times N}\to[0,+\infty]$ satisfait (\ref{Hyp1Example:m=N}) alors {(\ref{FunctRelax})} a lieu avec $\overline{W}=\mathcal{Q}W=\Z W$.  
\end{corollary}

\begin{proof}
De l'\'etape 1 de la d\'emonstration du th\'eor\`eme \ref{ThExample:m=N} on voit $\Z W$ est finie, donc $\Z W$ est rang-1 convexe par la proposition \ref{FonsecaProperties}(b), d'o\`u $\Z W\leq \R W$. D'autre part, du th\'eor\`eme \ref{RWaCCbyBB} on d\'eduit que $\R W$ satisfait (\ref{RWGrowth-Cond}). En effet, comme $W$ satisfait (\ref{Hyp1Example:m=N}) et $|\det F|=v(F)$ pour tout $F\in\MM^{N\times N}$, o\`u $v(F)$ est le produit des valeurs singuli\`eres de $F$, alors $W$ v\'erifie (\ref{RWaCCbyBBCondition}). Par cons\'equent, $\Z W$ satisfait (\ref{Cond-CroissZW}), et le r\'esultat suit par le th\'eor\`eme \ref{GeneralThRelax}. 
\end{proof}

\subsection{Contrainte d\'eterminant strictement positif}

\'Etant donn\'e $n\geq 1$ quelconque, on d\'efinit $W_n:\MM^{N\times N}\to[0,+\infty]$ par 
$$
W_n(F):=\left\{
\begin{array}{ll}
W(F)&\hbox{si }\det F>0\\
h_n(F)&\hbox{si }\det F\leq 0,
\end{array}
\right.
$$
o\`u $h_n:\MM^{N\times N}\to[0,+\infty]$ est Borel mesurable, v\'erifie $h_n(F)\leq c_n(1+|F|^p)$ pour tout $F\in\MM^{N\times N}$ avec $c_n>0$ et, pour chaque $F\in\MM^{N\times N}$ tel que ${\rm det}F\leq 0$, $\{h_n(F)\}_{n\geq 1}$ est croissante et $\sup_{n\geq 1}h_n(F)=+\infty$. La suite $\{W_n\}_{n\geq 1}$ est donc croissante et,  lorsque $W$ satisfait (\ref{Cst1}), $\sup_{n\geq 1}W_n=W$. Pour plus de d\'etails sur la question suivante voir Ball \cite[\S 2.2 p. 7]{ball02}.
\begin{question}\label{CoNjEcTuRe}
Peut-on choisir $\{h_n\}_{n\geq 1}$ de mani\`ere \`a avoir le r\'esultat qui suit. Si $W$ satisfait {(\ref{Cst1})} et la condition suivante 
\begin{eqnarray}\label{Hyp3Example:m=N}
&&\hbox{pour tout }\delta>0\hbox{ il existe }c_\delta>0 \hbox{ tel que pour tout }F\in\MM^{N\times N}\hbox{,}\\
&&\hbox{si }\det F\geq\delta\hbox{ alors }W(F)\leq c_\delta(1+|F|^p)\hbox{,}\nonumber
\end{eqnarray}
alors {(\ref{FunctRelax})} a lieu avec $\overline{W}=\sup_{n\geq 1}\mathcal{Q}W_n=\sup_{n\geq 1}\Z W_n$ ?
\end{question}
\begin{remark}
Si $W$ satisfait {\rm(\ref{Cst1})} et  (\ref{Hyp3Example:m=N}) on a  
\begin{eqnarray}\label{EquCOnJEcTUrE}
&&\displaystyle\int_\Omega\sup_{n\geq 1}\mathcal{Q}W_n(\nabla\phi(x))dx=\int_\Omega\sup_{n\geq 1}\Z W_n(\nabla\phi(x))dx\leq\overline{I}(\phi)\\
&& \hbox{pour tout }\phi\in W^{1,p}(\Omega;\RR^N).\nonumber
\end{eqnarray}
En effet, d\'efinissons $I_n:W^{1,p}(\Omega;\RR^N)\to[0,+\infty]$ par 
$$
I_n(\phi):=\int_\Omega W_n(\nabla\phi(x))dx
$$ 
et notons $\overline{I}_n$ la r\'egularis\'ee sci de $I_n$ par rapport \`a la topologie faible de $W^{1,p}(\Omega;\RR^N)$. Comme $W$ v\'erifie {\rm(\ref{Cst1})}, il est clair que $\sup_{n\geq 1}\overline{I}_n\leq\overline{I}$. De plus, chaque $W_n$ v\'erifie {\rm(\ref{Hyp1Example:m=N})}  puisque $W$ satisfait {\rm(\ref{Hyp3Example:m=N})}, donc, par le corollaire {\rm\ref{CorollaryExample:m=NBis}}, $\overline{I}_n(\phi)=\int_\Omega \Z W_n(\nabla \phi(x))dx=\int_\Omega\mathcal{Q}W_n(\nabla\phi(x))dx$ pour tout $n\geq 1$ et tout $\phi\in W^{1,p}(\Omega;\RR^N)$, et (\ref{EquCOnJEcTUrE}) suit. 
\end{remark}

\section{Passage 3D-2D avec contraintes de type d\'eterminant}

\subsection{G\'en\'eralit\'es}

Dans \cite{oah-jpm06,oah-jpm08b} nous avons montr\'e les deux th\'eor\`emes suivants (voir aussi le th\'eor\`eme \ref{3D-2DTheoremBis}). 
\begin{theorem}\label{ThMemb1}
Si $W$ est continue et satisfait {(\ref{Cst1prime})}, {(\ref{Hyp1Example:m=N})} et 
\begin{eqnarray}\label{AdditionalCond}
W(\xi\mid\zeta)=W(\xi\mid-\zeta)\hbox{ pour tout }\xi\in\MM^{3\times 2}\hbox{ et tout }\zeta\in\RR^3\hbox{,}
\end{eqnarray}
alors $I_{\rm mem}=\Gamma(\pi)\hbox{\rm -}\lim_{\eps\to 0}I_\eps$ avec $W_{\rm mem}=\mathcal{Q} W_0=\Z W_0$.
\end{theorem}
 (Les fonctionnelles $I_\eps:W^{1,p}(\Sigma_\eps;\RR^3)\to[0,+\infty]$ et $I_{\rm mem}:W^{1,p}(\Sigma;\RR^3)\to[0,+\infty]$, avec $\Sigma_\eps:=\Sigma\times]-{\eps\over 2},{\eps\over 2}[\subset\RR^3$ o\`u  $\Sigma\subset\RR^2$ est un ouvert born\'e  et $\eps>0$, sont respectivement d\'efinies en (\ref{TriDimEnerMem}) et (\ref{BiDimEnerMem}). La fonction $W_0:\MM^{3\times 2}\to[0,+\infty]$ est donn\'ee par (\ref{W_0}).)
\begin{theorem}\label{ThMemb2}
Si $W$ est continue et v\'erifie les conditions {(\ref{Cst1})} et {(\ref{Hyp3Example:m=N})} alors $I_{\rm mem}=\Gamma(\pi)\hbox{\rm -}\lim_{\eps\to 0}I_\eps$ avec $W_{\rm mem}=\mathcal{Q} W_0=\Z W_0$.
\end{theorem}
 Ces th\'eor\`emes \'etendent le th\'eor\`eme \ref{LeDretRaoult}. Ils sont plus r\'ealistes du point de vue de l'hyper\'elasticit\'e : le th\'eor\`eme \ref{ThMemb1} (resp. \ref{ThMemb2}) est compatible avec (\ref{Cst1prime}) (resp. (\ref{Cst1})) et (\ref{Cst2}).

Dans ce qui suit nous donnons les preuves des th\'eor\`emes \ref{ThMemb1} et \ref{ThMemb2} (voir \S 3.1.1-3.1.3 pour la structure g\'en\'erale et \S 3.2-3.3 pour les d\'etails). Les d\'emonstrations des th\'eor\`emes {\rm\ref{ThMemb1}} et {\rm\ref{ThMemb2}} ont la m\^eme structure que l'on peut d\'ecomposer en trois points (voir \S 3.1.1-3.1.3).

\subsubsection{Pr\'eliminaires}

Pour chaque $\eps>0$, on consid\`ere  $\mathcal{I}_\eps:W^{1,p}(\Sigma;\RR^3)\to[0,+\infty]$ d\'efinie par 
$$
\mathcal{I}_\eps(\psi):=\inf\big\{I_\eps(\phi):\pi_\eps(\phi)=\psi\big\}.
$$
\begin{definition}\label{GammaCvgDef}
On dit que $\mathcal{I}_\eps$ $\Gamma$-converge vers $I_{\rm mem}$ et on \'ecrit
$$
I_{\rm mem}=\Gamma\hbox{-}\lim_{\eps\to 0}\mathcal{I}_\eps
$$
si 
$$\Gamma\hbox{-}\liminf_{\eps\to 0}\mathcal{I}_\eps=\Gamma\hbox{-}\limsup_{\eps\to 0}\mathcal{I}_\eps=I_{\rm mem}
$$  
avec $\displaystyle\Gamma\hbox{-}\liminf_{\eps\to 0}\mathcal{I}_\eps,\Gamma\hbox{-}\limsup_{\eps\to 0}\mathcal{I}_\eps:W^{1,p}(\Sigma;\RR^3)\to[0,+\infty]
$ d\'efinies par :
\begin{itemize}
\item[]  $\displaystyle\left(\Gamma\hbox{-}\liminf_{\eps\to 0}\mathcal{I}_\eps\right)(\psi):=\inf\left\{\liminf_{\eps\to 0}\mathcal{I}_\eps(\psi_\eps):W^{1,p}(\Sigma;\RR^3)\ni\psi_\eps\wto\psi\right\};$
\item[] $\displaystyle\left(\Gamma\hbox{-}\limsup_{\eps\to 0}\mathcal{I}_\eps\right)(\psi):=\inf\left\{\limsup_{\eps\to 0}\mathcal{I}_\eps(\psi_\eps):W^{1,p}(\Sigma;\RR^3)\ni\psi_\eps\wto\psi\right\}$.
\end{itemize}
\end{definition}
 La d\'efinition \ref{GammaCvgDef} est \'equivalente \`a la d\'efinition \ref{DefGammaPiConvergence} o\`u ``$\psi_\eps\wto\psi$" est substitu\'e par ``$\pi_\eps(\phi_\eps)\wto\psi$". Il est alors facile de voir que 
\begin{eqnarray}\label{EquivGammaPi-Gamma}
{I}_{\rm mem}=\Gamma(\pi)\hbox{\rm -}\lim_{\eps\to 0}{I}_\eps\hbox{ si et seulement si }{I}_{\rm mem}=\Gamma\hbox{\rm -}\lim_{\eps\to 0}\mathcal{I}_\eps.
\end{eqnarray}
Tout revient donc \`a ``calculer" la $\Gamma$-limite de $\mathcal{I}_\eps$ lorsque $\eps\to 0$. Pour cela, on \'etudie d'abord la $\Gamma$-convergence de $\mathcal{I}_\eps$ lorsque $\eps\to 0$ (voir \S 3.1.2) et on \'etablit ensuite une repr\'esentation int\'egrale de la $\Gamma$-limite (voir \S 3.1.3). (Cette proc\'edure g\'en\'erale a \'et\'e initi\'ee dans \cite{oah05}, voir aussi \cite{anzellotti-baldo-percivale94}). 

Le lemme suivant sera utilis\'e dans la suite (rappelons que $W$ est suppos\'ee coercive). 

\begin{lemma}\label{Lemme1Det>0}
Si $W$ est continue et v\'erifie {(\ref{Cst1prime})} et {(\ref{Hyp1Example:m=N})} ou {(\ref{Cst1})} et {(\ref{Hyp3Example:m=N})} alors $W_0$ satisfait les propri\'et\'es suivantes {:}
\begin{eqnarray}
&&W_0 \hbox{ est continue {(}et coercive{)}}\hbox{ ;}\nonumber\\
&&W_0(\xi_1\mid\xi_2)=+\infty\hbox{ si et seulement si }\xi_1\land\xi_2=0\hbox{ ;}\label{Adj-Infty1}\\
&&\hbox{pour tout }\delta>0\hbox{ il existe }c_\delta>0 \hbox{ tel que pour tout }\xi=(\xi_1\mid\xi_2)\in\MM^{3\times 2}\hbox{,}\label{Adj-Infty2}\\
&&\hbox{si }|\xi_1\land\xi_2|\geq\delta\hbox{ alors }W_0(\xi)\leq c_\delta(1+|\xi|^p).\nonumber
\end{eqnarray}
Plus pr\'ecis\'ement, $W_0$ est continue {(}et coercive{)} d\`es que $W$ est continue {(}et coercive{)}, {(\ref{Adj-Infty1})} a lieu d\`es que $W$ v\'erifie {(\ref{Cst1prime})} ou {(\ref{Cst1})} et {(\ref{Adj-Infty2})} a lieu d\`es que $W$ v\'erifie {(\ref{Hyp1Example:m=N})} ou {(\ref{Hyp3Example:m=N})}. 
\end{lemma}
\begin{proof}
Comme $W$ est coercive, on a 
\begin{equation}\label{UsEfulLeMMaEq1}
W(\xi\mid\zeta)\geq C(|\xi|^p+|\zeta|^p)\hbox{ pour tout }(\xi,\zeta)\in\MM^{3\times 2}\times\RR^3
\end{equation}
avec $C>0$ (qui ne d\'epend que de $p$). Donc $W_0(\xi)\geq C|\xi|^p$ pour tout $\xi\in\MM^{3\times 2}$, i.e., $W_0$ est coercive. Puisque $W$ est continue, $W(\cdot\mid\zeta)$ est continue pour chaque $\zeta\in\RR^3$, d'o\`u $W_0$ est semi-continue sup\'erieurement. Donc, pour montrer que $W_0$ est continue, il suffit de prouver que $W_0$ est sci. Pour cela, consid\'erons $\xi\in\MM^{3\times 2}$ et $\{\xi_n\}_{n\geq 1}\subset\MM^{3\times 2}$ tels que $\xi_n\to\xi$, $\sup_{n\geq 1}W_0(\xi_n)<+\infty$ et $\lim_{n\to+\infty}W_0(\xi_n)=\liminf_{n\to+\infty}W_0(\xi_n)$. Pour chaque $n\geq 1$, il existe $\zeta_n\in\RR^3$ tel que 
\begin{equation}\label{UsEfulLeMMaEq2}
W_0(\xi_n)<W(\xi_n\mid\zeta_n)\leq W_0(\xi_n)+{1\over n}.
\end{equation}
Alors $\sup_{n\geq 1}W(\xi_n\mid\zeta_n)<+\infty$ (puisque $\sup_{n\geq 1}W_0(\xi_n)<+\infty$). Or, par (\ref{UsEfulLeMMaEq1}), $|\zeta_n|^p\leq{1\over C} W(\xi_n\mid\zeta_n)$ pour tout $n\geq 1$, donc $\sup_{n\geq 1}|\zeta_n|^p<+\infty$. Il suit qu'il existe $\zeta\in\RR^3 $ tel que   $\zeta_n\to\zeta$ (\`a une sous-suite pr\`es). D'o\`u $\lim_{n\to+\infty}W(\xi_n\mid\zeta_n)=W(\xi\mid\zeta)$ (puisque $W$ est continue). De \eqref{UsEfulLeMMaEq2} on d\'eduit que $\lim_{n\to+\infty}W_0(\xi_n)=W(\xi\mid\zeta)\geq W_0(\xi)$.

Montrons (\ref{Adj-Infty1}). \'Etant donn\'e $\xi=(\xi_1\mid\xi_2)$, si $W_0(\xi_1\mid\xi_2)<+\infty$ (resp. $W_0(\xi_1\mid\xi_2)=+\infty$) alors $W(\xi\mid\zeta)<+\infty$ (resp. $W(\xi\mid\zeta)=+\infty$) avec $\zeta\in\RR^3$ (resp. pour tout $\zeta\in\RR^3$), donc $\xi_1\land\xi_2\not=0$ (resp. $\xi_1\land\xi_2=0$) par {(\ref{Cst1prime})} ou  {(\ref{Cst1})}.

Montrons (\ref{Adj-Infty2}). Soient $\delta>0$ et $\xi=(\xi_1\mid\xi_2)$ tels que $|\xi_1\land\xi_2|\geq\delta$. Posant $\zeta:={\xi_1\land\xi_2\over|\xi_1\land\xi_2|}$ on a $\det (\xi\mid\zeta)\geq\delta$, et utilisant  {(\ref{Hyp1Example:m=N})} ou  {(\ref{Hyp3Example:m=N})} on d\'eduit qu'il existe $c_\delta>0$ (qui ne d\'epend pas de $\xi$) tel que $W_0(\xi)\leq c_\delta(1+|\xi|^p)$.
\end{proof}
\begin{remark}\label{QRW_0=QW_0}
Pour Ben Belgacem $W_{\rm mem}=\mathcal{Q}\R W_0$ (voir \cite{benbelgacem96,benbelgacem97,benbelgacem00}). En fait, lorsque $W$ satisfait (\ref{Hyp3Example:m=N}), on a $\mathcal{Q}\R W_0=\mathcal{Q}W_0=\Z W_0$. En effet,  du lemme \ref{Lemme1Det>0} on voit que $W_0$ satisfait {\rm(\ref{Adj-Infty2})} donc {\rm(\ref{HypExample:m=N+1})}. Du  th\'eor\`eme \ref{ThExample:m=N+1} on d\'eduit que $\Z W_0$ v\'erifie (\ref{Cond-CroissZW}) (donc $\Z W_0$ est finie), et le r\'esultat suit de la remarque \ref{QW-QRW-LiNk}.
\end{remark}

On d\'efinit $\mathcal{I}:W^{1,p}(\Sigma;\RR^3)\to[0,+\infty]$ par 
$$
\mathcal{I}(\psi):=\int_\Sigma W_0(\nabla\psi(x))dx,
$$
et on consid\`ere $\overline{\mathcal{I}},\overline{\mathcal{I}}_{\rm aff},\overline{\mathcal{I}}_{\rm diff_*}:W^{1,p}(\Sigma;\RR^3)\to[0,+\infty]$ donn\'ees par :
\begin{itemize}
\item[$\diamond$] $\displaystyle\overline{\mathcal{I}}(\psi):=\inf\left\{\liminf_{n\to+\infty}\mathcal{I}(\psi_n):W^{1,p}(\Sigma;\RR^3)\ni\psi_n\wto\psi\right\};$
\end{itemize}
\begin{itemize}
\item[$\diamond$] $\displaystyle\overline{\mathcal{I}}_{\rm aff}(\psi):=\inf\left\{\liminf_{n\to+\infty}\mathcal{I}(\psi_n):\Aff(\Sigma;\RR^3)\ni\psi_n\wto\psi\right\};$
\item[$\diamond$] $\displaystyle\overline{\mathcal{I}}_{\rm diff_*}(\psi):=\inf\left\{\liminf_{n\to+\infty}\mathcal{I}(\psi_n):C^1_*(\overline{\Sigma};\RR^3)\ni\psi_n\wto\psi\right\}$
\end{itemize}
avec $C^1_*(\overline{\Sigma};\RR^3):=\big\{\psi\in C^1(\overline{\Sigma};\RR^3):\partial_1\psi(x)\land\partial_2\psi(x)\not=0\hbox{ pour tout }x\in\overline{\Sigma}\big\}$ o\`u $\partial_1\psi(x)$ (resp. $\partial_2\psi(x)$) d\'esigne la d\'eriv\'ee partielle de $\psi$ en $x=(x_1,x_2)$ par rapport \`a $x_1$ (resp. $x_2$). (En fait, $C^1_*(\overline{\Sigma};\RR^3)$ est l'ensemble des $C^1$-immersions de $\overline{\Sigma}$ dans $\RR^3$.)
\begin{remark}\label{Aff=Aff_*LoRsQuEW_0veRifieAdj}
Lorsque $W_0$ v\'erifie (\ref{Adj-Infty1}) on a 
\begin{itemize}
\item[$\diamond$] $\displaystyle\overline{\mathcal{I}}_{\rm aff}(\psi):=\inf\left\{\liminf_{n\to+\infty}\mathcal{I}(\psi_n):\Aff_*(\Sigma;\RR^3)\ni\psi_n\wto\psi\right\}$
\end{itemize}
avec $\Aff_*(\Sigma;\RR^3):=\big\{\psi\in \Aff(\Sigma;\RR^3):\partial_1\psi(x)\land\partial_2\psi(x)\not=0\hbox{ p.p. dans  }\overline{\Sigma}\big\}$. 
\end{remark}

\subsubsection{Existence de la $\Gamma$-limite de $\mathcal{I}_\eps$.}

On d\'emontre les trois propositions suivantes.

\begin{proposition}\label{det>0Prop1}
$\Gamma$\hbox{-}$\liminf\limits_{\eps\to 0}\mathcal{I}_\eps\geq\overline{\mathcal{I}}$.
\end{proposition}

\begin{proposition}\label{det>0Prop2}
Si $W$ est continue et satisfait {(\ref{Cst1prime})}, {(\ref{Hyp1Example:m=N})} et {(\ref{AdditionalCond})} alors 
$$
\Gamma\hbox{-}\limsup\limits_{\eps\to 0}\mathcal{I}_\eps\leq\overline{\mathcal{I}}_{\rm aff}.
$$
\end{proposition}

\begin{proposition}\label{det>0Prop3}
Si $W$ est continue et satisfait {(\ref{Cst1})} et {(\ref{Hyp3Example:m=N})} alors 
$$
\Gamma\hbox{-}\limsup\limits_{\eps\to 0}\mathcal{I}_\eps\leq\overline{\mathcal{I}}_{\rm diff_*}.
$$
\end{proposition}

La proposition \ref{det>0Prop1} se d\'emontre assez facilement (voir ci-dessous). Les propositions \ref{det>0Prop2} et \ref{det>0Prop3} sont plus difficiles \`a prouver (voir \S 3.2 et \S 3.3).

\begin{proof}[D\'emonstration de la proposition \ref{det>0Prop1}]
Consid\'erons $\psi\in W^{1,p}(\Sigma;\RR^3)$ et $\{\psi_\eps\}_\eps\subset W^{1,p}(\Sigma;\RR^3)$ tels que $\psi_\eps\wto \psi$ dans $W^{1,p}(\Sigma;\RR^3)$. On doit prouver que 
\begin{equation}\label{main_ineq1}
\liminf_{\eps\to 0}\mathcal{I}_\eps(\psi_\eps)\geq \overline{\mathcal{I}}(v).
\end{equation}
Sans perdre de g\'en\'eralit\'e on peut supposer que $\sup_{\eps>0}\mathcal{I}_\eps(\psi_\eps)<+\infty$ and 
$\psi_{\eps}\to \psi$ dans $L^p(\Sigma;\RR^3)$. Pour chaque $\eps>0$, il existe $\phi_\eps\in\pi_{\eps}^{-1}(\psi_\eps)$ tel que 
\begin{equation}\label{main_ineq2}
\mathcal{I}_\eps(\psi_\eps)\geq I_\eps(\phi_\eps)-\eps.
\end{equation}
D\'efinissons $\hat \phi_\eps:\Sigma_1\to\RR^3$ par $\hat \phi_\eps(x,x_3):=\phi_\eps(x,\eps x_3)$ (avec $\Sigma_1=\Sigma\times]-{1\over 2},{1\over 2}[$). On a alors 
\begin{equation}\label{main_ineq3}
I_\eps(\phi_\eps)=\int_{\Sigma_1}W\left(\partial_{1} \hat \phi_\eps(x,x_3)\mid\partial_{2} \hat \phi_\eps(x,x_3)\mid{1\over\eps}\partial_3 \hat \phi_\eps(x,x_3)\right)dxdx_3.
\end{equation}
Utilisant la coercivit\'e de  $W$  on d\'eduit qu'il existe $c>0$ tel que $\|{\partial_3 \hat \phi_\eps}\|_{L^p(\Sigma_1;\RR^3)}\leq c\eps^{p}$ pour tout $\eps>0$, donc
$
\|\hat \phi_\eps-\psi_\eps\|_{L^p(\Sigma_1;\RR^3)}\leq c^\prime\eps^p
$
par l'in\'egalit\'e de  Poincar\'e-Wirtinger, o\`u $c^\prime>0$ est une constante ind\'ependante de $\eps$. Il suit que $\hat \phi_\eps\to \psi$ dans $L^p(\Sigma_1;\RR^3)$. Soit $\varphi_\eps^{x_3}\in W^{1,p}(\Sigma;\RR^3)$ d\'efinie par $\varphi_\eps^{x_3}(x):=\hat \phi_\eps(x,x_3)$ avec $x_3\in]-{1\over 2},{1\over 2}[$. Alors (\`a une sous-suite pr\`es) $\varphi_\eps^{x_3}\to \psi$ dans $L^p(\Sigma;\RR^3)$ pour presque tout $x_3\in ]-{1\over 2},{1\over 2}[$. Prenant en compte  (\ref{main_ineq2}) et (\ref{main_ineq3}) et utilisant le lemme de Fatou, on obtient 
$$
\liminf_{\eps\to 0}{\mathcal{I}}_\eps(\psi_\eps)\geq\int_{-{1\over 2}}^{1\over 2}\left(\liminf_{\eps\to 0}\int_\Sigma W_0\big(\nabla \varphi_\eps^{x_3}(x)\big)dx\right)dx_3,
$$
donc $\liminf_{\eps\to 0}{\mathcal{I}}_\eps(v_\eps)\geq\overline{\mathcal{J}}(\psi)$ avec $\overline{\mathcal{J}}:W^{1,p}(\Sigma;\RR^3)\to[0,+\infty]$ donn\'ee par 
$$
\overline{\mathcal{J}}(\varphi):=\inf\left\{\liminf_{n\to+\infty}\int_\Sigma W_0\big(\nabla \varphi_n(x)\big)dx:W^{1,p}(\Sigma;\RR^3)\ni \varphi_n\to \varphi\hbox{ dans }L^p(\Sigma;\RR^3)\right\}.
$$
Comme $W_0$ est coercive (voir le lemme \ref{Lemme1Det>0}) on a  $\overline{\mathcal{J}}=\overline{\mathcal{I}}$, et (\ref{main_ineq1}) suit.
\end{proof}

La proposition suivante est une cons\'equence imm\'ediate des propositions \ref{det>0Prop1}, \ref{det>0Prop2} et \ref{det>0Prop3}.
\begin{proposition}\label{det>0Prop4}
$\Gamma\hbox{-}\lim_{\eps\to 0}\mathcal{I}_\eps=\overline{\mathcal{I}}$ dans les deux cas suivants {:}
\begin{enumerate}
\item[(a)] $W$ est continue (coercive), satisfait {(\ref{Cst1prime})}, {(\ref{Hyp1Example:m=N})} et {(\ref{AdditionalCond})} et $\overline{\mathcal{I}}=\overline{\mathcal{I}}_{\rm aff}\hbox{ ;}$
\item[(b)] $W$ est continue (coercive), satisfait {(\ref{Cst1})} et {(\ref{Hyp3Example:m=N})} et $\overline{\mathcal{I}}=\overline{\mathcal{I}}_{\rm diff_*}.$
\end{enumerate}
\end{proposition}

\subsubsection{Repr\'esentation int\'egrale de la $\Gamma$-limite de $\mathcal{I}_\eps$}

D'apr\`es les th\'eor\`emes \ref{GeneralThRelax} et \ref{ThExample:m=N+1}, si $W_0$ satisfait (\ref{Adj-Infty2}) (donc (\ref{HypExample:m=N+1})) alors $\overline{\mathcal{I}}=I_{\rm mem}$ avec $W_{\rm mem}=\mathcal{Q}W_0=\Z W_0$. 

Donc, prenant en compte (\ref{EquivGammaPi-Gamma}) et la proposition \ref{det>0Prop4}(a), le th\'eor\`eme \ref{ThMemb1} est d\'emon-tr\'e si on prouve que $\overline{\mathcal{I}}=\overline{\mathcal{I}}_{\rm aff}$ ce qui est vraie gr\^ace au th\'eor\`eme \ref{GeneralThRelax}-bis.

De m\^eme, prenant en compte (\ref{EquivGammaPi-Gamma}) et la proposition \ref{det>0Prop4}(b), le th\'eor\`eme \ref{ThMemb2} est d\'emontr\'e si on prouve que 
\begin{eqnarray}\label{BenBelgacemBennequinGromovEliashbergEgalitŽ}
\overline{\mathcal{I}}=\overline{\mathcal{I}}_{\rm diff_*}.
\end{eqnarray}
Cette derni\`ere \'egalit\'e  est plus difficile \`a d\'emontrer (voir le th\'eor\`eme \ref{TheoEgalitŽ(37)}). Sa preuve utilise deux th\'eor\`emes d'approximation par Ben Belgacem-Bennequin (voir le th\'eor\`eme \ref{BBBapproxTheo}) et Gromov-Eliashberg (voir le th\'eor\`eme \ref{GEdensityTheo}) ainsi que deux lemmes de Ben Belgacem (voir les lemmes \ref{BenBelgacemLemma} et \ref{BenBelgacemLemma2}). (Pour la d\'emonstration de (\ref{BenBelgacemBennequinGromovEliashbergEgalitŽ}) voir \S 3.3.3.)

\subsection{Contrainte d\'eterminant non nul}

Dans ce paragraphe on d\'emontre la proposition \ref{det>0Prop2}. 

\subsubsection{Pr\'eliminaires}

\'Etant donn\'e $\psi\in\Aff_*(\Sigma;\RR^3)$,  il existe une famille finie $\{V_i\}_{i\in I}$ de sous-ensembles ouverts et disjoints de $\Sigma$ telle que $|\Sigma\setminus\cup_{i\in I} V_i|=0$, $|\partial V_i|=0$ et $\nabla \psi(x)=\xi_i$ dans $V_i$ pour tout $i\in I$ avec $\xi_i=(\xi_{i,1}\mid\xi_{i,2})\in\MM^{3\times 2}$ et $\xi_{i,1}\land\xi_{i,2}\not=0$. Pour chaque $i\in I$ et chaque $j\geq 1$, on d\'efinit $U^-_{i,j},U^+_{i,j}\subset\RR^3$ par 
$$
U^-_{i,j}:=\left\{\zeta\in\RR^3:\det(\xi_i\mid\zeta)\leq -{1\over j}\right\}\hbox{ et }U^+_{i,j}:=\left\{\zeta\in\RR^3:\det(\xi_i\mid\zeta)\geq {1\over j}\right\}.
$$
Il est facile de voir que :
\begin{eqnarray}\label{Detnot=0Prop1}
&&U^-_{i,j}\hbox{ et }U^+_{i,j}\hbox{ sont non vides et convexes}\ ;\\
&&U^-_{i,j}\cup U^+_{i,j}=\left\{\zeta\in\RR^3:|\det(\xi\mid\zeta)|\geq{1\over j}\right\};\label{Detnot=0Prop2}\\
&&U^-_{i,1}\subset U^-_{i,2}\subset U^-_{i,3}\subset\cdots\subset \cupp\limits_{j\geq 1}U^-_{i,j}=\big\{\zeta\in\RR^3:\det(\xi\mid\zeta)<0\big\}\hskip0.3mm;\label{Detnot=0Prop3}\\
&&U^+_{i,1}\subset U^+_{i,2}\subset U^+_{i,3}\subset\cdots\subset \cupp\limits_{j\geq 1}U^+_{i,j}=\big\{\zeta\in\RR^3:\det(\xi\mid\zeta)>0\big\}.\label{Detnot=0Prop4}
\end{eqnarray}
De plus, on a 
\begin{eqnarray}\label{Detnot=0Prop5}
&&\hbox{il existe }j_\psi\geq 1\hbox{ et une partition }(I^-,I^+)\hbox{ de } I\hbox{ telle que }\\
 &&(\cap_{i\in I^-}U^-_{i,j})\cap(\cap_{i\in I^+}U^+_{i,j})\not=\emptyset\hbox{ pour tout }j\geq j_\psi.\nonumber
\end{eqnarray}
\begin{proof}[Preuve de (\ref{Detnot=0Prop5})]
Pour chaque $i\in I$, consid\'erons l'hyperplan $H_i\subset\RR^3$ d\'efini par
$
H_i:=\{\zeta\in\RR^3:\det(\xi_i\mid \zeta)=0\}.
$
Il est clair que $\cup_{i\in I}H_i\not=\RR^3$ et qu'il existe $\zeta\in\RR^3$ tel que
$
\det(\xi_i\mid \zeta)\not=0
$
pour tout $i\in I$. Utilisant (\ref{Detnot=0Prop2}) on d\'eduit qu'il existe $j_\psi\geq 1$ tel que $\zeta\in\cap_{i\in I}(U_{i,j_\psi}^-\cup U_{i,j_\psi}^+)$. On peut donc trouver une partition $(I^-, I^+)$ de $I$ telle que
$
(\cap_{i\in I^-}U_{i,j_\psi}^-)\cap(\cap_{i\in I^+}U_{i,j_\psi}^+)\not=\emptyset,
$
et (\ref{Detnot=0Prop5}) suit en utilisant  (\ref{Detnot=0Prop3}) et (\ref{Detnot=0Prop4}).
\end{proof}

On pose $V=\cup_{i\in I}V_i$ et, pour chaque $j\geq j_\psi$, on d\'efinit $\Gamma_\psi^j:\overline{\Sigma}\dto\RR^3$ par 
$$
\Gamma_\psi^j(x):=\left\{
\begin{array}{ll}
U_{i,j}^-&\hbox{ si }x\in V_i\hbox{ avec }i\in I^-\\
U_{i,j}^+&\hbox{ si }x\in V_i\hbox{ avec }i\in I^+\\
\Big(\capp\limits_{i\in I^-}U_{i,j}^-\Big)\cap\Big(\capp\limits_{i\in I^+}U_{i,j}^+\Big)&\hbox{ si }x\in \overline{\Sigma}\setminus V.
\end{array}
\right.
$$
Le lemme suivant est une cons\'equence du th\'eor\`eme \ref{InfIntTH2}. (Dans ce qui suit, \'etant donn\'e $\Gamma:\overline{\Sigma}\dto\RR^3$, on pose $C(\overline{\Sigma};\Gamma):=\big\{\varphi\in C(\overline{\Sigma};\RR^3):\varphi(x)\in\Gamma(x)\hbox{ p.p. dans }\overline{\Sigma}\big\}$ avec $C(\overline{\Sigma};\RR^3)$ d\'esignant l'espace des fonctions continues de $\overline{\Sigma}$ dans $\RR^3$.)
\begin{lemma}\label{LemmeDeTnot=0}
Soient $\psi\in\Aff_*(\Sigma;\RR^3)$ et $j\geq j_\psi$. Si $W$ est continue et satisfait {(\ref{Hyp1Example:m=N})} alors 
\begin{eqnarray}\label{EquALitYLemmeDeTnot=0}
\inf_{\varphi\in C(\overline{\Sigma};\Gamma_\psi^j)}\int_\Sigma W(\nabla\psi(x)\mid\varphi(x))dx=\int_\Sigma \inf_{\zeta\in\Gamma_\psi^j(x)}W(\nabla\psi(x)\mid\zeta)dx.
\end{eqnarray}
\end{lemma}
\begin{proof}
(On utilise la notation $\Det(\xi\mid\zeta)=|\det(\xi\mid\zeta)|$.) Puisque $W$ est continue, (\ref{ContInterHyp1}) a lieu avec $f(x,\zeta)=W(\nabla\psi(x)\mid\zeta)$. Prenant en compte (\ref{Detnot=0Prop1}) et (\ref{Detnot=0Prop5}), on voit que $\Gamma^j_\psi$ est une multifonction sci \`a valeurs convexes ferm\'ees non vides et par cons\'equent (\ref{ContInterHyp2}) a lieu avec $\Gamma=\Gamma^j_\psi$. \'Etant donn\'es $\varphi,\hat\varphi\in C(\overline{\Sigma};\Gamma^j_\psi)$, il est clair que $\Det(\nabla\psi(x)\mid\alpha\varphi(x)+(1-\alpha)\hat\varphi(x))\geq{1\over j}$ pour tout $\alpha\in[0,1]$ et presque tout $x\in\Sigma$. De (\ref{Hyp1Example:m=N}) on d\'eduit qu'il existe $c>0$ (d\'ependant seulement de $j,\psi,\varphi$ et $\hat\varphi$) tel que $W(\nabla\psi(x)\mid\alpha\varphi(x)+(1-\alpha)\hat\varphi(x))\leq c$  pour tout $\alpha\in[0,1]$ et presque tout $x\in\Sigma$. Ainsi (\ref{ContInterHyp3}) a lieu avec $f(x,\zeta)=W(\nabla\psi(x)\mid\zeta)$ et $\Gamma=\Gamma^j_\psi$. On applique le th\'eor\`eme \ref{InfIntTH2} et on obtient (\ref{EquALitYLemmeDeTnot=0}).
\end{proof}
\begin{proof}[Une autre d\'emonstration du lemme \ref{LemmeDeTnot=0}]
Il est clair que 
$$
\inf_{\varphi\in C(\overline{\Sigma};\Gamma_\psi^j)}\int_{\Sigma}W(\nabla \psi(x)\mid\varphi(x))dx\geq\int_\Sigma \inf_{\zeta\in\Gamma_\psi^j(x)}W(\nabla \psi(x)\mid\zeta)dx.
$$
Prouvons donc l'autre in\'egalit\'e. Par (\ref{Detnot=0Prop5}) on a $(\cap_{i\in I^-}U_{i,j}^-)\cap(\cap_{i\in I^+}U_{i,j}^+)\not=\emptyset$, donc il existe $\bar\zeta\in(\cap_{i\in I^-}U_{i,j}^-)\cap(\cap_{i\in I^+}U_{i,j}^+)$. Comme chaque $U^-_{i,j}$ (resp. $U^+_{i,j}$) est ferm\'e, chaque $W(\xi_i\mid\cdot)$ est continue et $W$ est coercive, pour tout $i\in I^-$ (resp. $i\in I^+$) il existe $\zeta_i\in U^-_{i,j}$ (resp. $\zeta_i\in U^+_{i,j}$) tel que 
\begin{eqnarray}\label{AddING}
&&W(\xi_i\mid\zeta_i)=\inf_{\zeta\in U^-_{i,j}}W(\xi\mid\zeta)\  (\hbox{resp. }W(\xi_i\mid\zeta_i)=\inf_{\zeta\in U^+_{i,j}}W(\xi\mid\zeta)).
\end{eqnarray}
Soit $n\geq 1$. Consid\'erons $\alpha_n:\overline{\Sigma}\to\RR$ donn\'ee par $\alpha_n(x):=h(n{\rm dist}(x,\overline{\Sigma}\setminus V))$, o\`u ${\rm dist}(x,\overline{\Sigma}\setminus V):=\inf\{|x-y|:y\in \overline{\Sigma}\setminus V\}$ et $h:[0,+\infty[\to[0,1]$ est une fonction continue telle que $h(0)=0$ et $h(t)=1$ pour tout $t\geq 1$. D\'efinissons $\varphi_n:\overline{\Sigma}\to\RR$ par 
$$
\varphi_n(x):=(1-\alpha_n(x))\bar\zeta+\alpha_n(x)\zeta_i.
$$
Il est clair que $\varphi_n$ est continue et puisque $\Gamma^j_\psi(x)$ est convexe, $\varphi_n(x)\in\Gamma^j_\psi(x)$ pour tout $x\in\overline{\Sigma}$, d'o\`u $\varphi_n\in C(\overline{\Sigma};\Gamma^j_\psi)$. Utilisant (\ref{Hyp1Example:m=N}) on d\'eduit que $\sup_{n\geq 1}W(\nabla \psi(\cdot)\mid\varphi_n(\cdot))\in L^1(\Sigma)$. Rappelant que  $W$ est continue et prenant en compte (\ref{AddING}), on voit que $\lim_{n\to+\infty}W(\nabla \psi(x)\mid\varphi_n(x))=\inf_{\zeta\in\Gamma^j_\psi(x)}W(\nabla \psi(x)\mid\zeta)$ pour presque tout $x\in\Sigma$. D'o\`u 
\begin{eqnarray*}
\inf_{\varphi\in C(\overline{\Sigma};\Gamma_\psi^j)}\int_{\Sigma}W(\nabla \psi(x)\mid\varphi(x))dx&\leq&\lim_{n\to+\infty}\int_\Sigma W(\nabla \psi(x)\mid\varphi_n(x))dx\\
&=&\int_\Sigma\inf_{\zeta\in\Gamma^j_\psi(x)}W(\nabla \psi(x)\mid\zeta)dx
\end{eqnarray*}
par le th\'eor\`eme de la convergence domin\'ee de Lebesgue.
\end{proof}

Pour chaque $j\geq j_\psi$, on d\'efinit $\hat\Gamma^j_\psi:\overline{\Sigma}\dto\RR^3$ (la multifonction ``sym\'etris\'ee" de $\Gamma^j_\psi$) par 
$$
\hat\Gamma^j_\psi(x):=\left\{
\begin{array}{ll}
U^-_{i,j}\cup U^+_{i,j}&\hbox{si }x\in V_i\\
\Gamma^j_\psi(x)&\hbox{si }x\in\overline{\Sigma}\setminus V.
\end{array}
\right.
$$
Noter que $\hat\Gamma^j_\psi$ n'est pas \`a valeurs convexes. C'est pour cela que l'on n'utilise pas cette multifonction dans le lemme \ref{LemmeDeTnot=0} et que l'on aura besoin de l'hypoth\`ese de ``sym\'etrie"(\ref{AdditionalCond}). Le th\'eor\`eme suivant donne une repr\'esentation ``non int\'egrale" de $\mathcal{I}$ sur $\Aff_*(\Sigma;\RR^3)$. 
\begin{theorem}\label{IntermediateTheorDetNot}
Si $W$ est continue et v\'erifie {(\ref{Cst1prime})}, {(\ref{Hyp1Example:m=N})} et {(\ref{AdditionalCond})} et si $\psi\in\Aff_*(\Sigma;\RR^3)$ alors
$$
\mathcal{I}(\psi)=\inf_{j\geq j_\psi}\inf_{\varphi\in C(\overline{\Sigma};\hat\Gamma^j_\psi)}\int_\Sigma W(\nabla\psi(x)\mid\varphi(x))dx.
$$
\end{theorem}
\begin{proof}
 Il suffit de prouver que 
\begin{eqnarray}\label{ineQuality}
\mathcal{I}(\psi)\geq\inf_{j\geq j_\psi}\inf_{\varphi\in C(\overline{\Sigma};\hat\Gamma^j_\psi)}\int_\Sigma W(\nabla\psi(x)\mid\varphi(x))dx.
\end{eqnarray}
Par (\ref{AdditionalCond}) on a 
\begin{eqnarray}\label{equality1}
\inf_{\zeta\in\Gamma^j_\psi(x)}W(\nabla \psi(x)\mid\zeta)=\inf_{\zeta\in\hat\Gamma^j_\psi(x)}W(\nabla \psi(x)\mid\zeta)
\end{eqnarray}
pour tout $j\geq j_\psi$ and tout $x\in{\Sigma}$.  Utilisant (\ref{equality1}), le lemme \ref{LemmeDeTnot=0} et le fait que $\Gamma^j_\psi\subset\hat\Gamma^j_\psi$, on obtient 
\begin{equation}\label{inequality1}
\inf_{j\geq j_\psi}\int_{\Sigma}\inf_{\zeta\in\hat\Gamma^j_\psi(x)}W(\nabla \psi(x)\mid\zeta)dx\geq\inf_{j\geq j_\psi}\inf_{\varphi\in C(\overline{\Sigma};\hat\Gamma^j_\psi)}\int_\Sigma W(\nabla\psi(x)\mid\varphi(x))dx.
\end{equation}
D'autre part, on a :
\begin{enumerate}
\item[$\diamond$] $\displaystyle\inf_{\zeta\in\hat\Gamma^{j_\psi}_\psi(\cdot)}W(\nabla\psi(\cdot)\mid\zeta)\in L^1(\Sigma)$ par (\ref{Hyp1Example:m=N})\ ;\\
\item[$\diamond$] $\displaystyle\Bigg\{\inf_{\zeta\in\hat\Gamma^j_\psi(\cdot)}W(\nabla \psi(\cdot)\mid\zeta)\Bigg\}_{j\geq j_\psi}$ est une suite d\'ecroissante par (\ref{Detnot=0Prop3}) et (\ref{Detnot=0Prop4}).
\end{enumerate}
De plus, on a 
\begin{equation}\label{equality}
\inf_{j\geq j_\psi}\inf_{\zeta\in\hat\Gamma^j_\psi(\cdot)}W(\nabla \psi(\cdot)\mid\zeta)=\inf_{\zeta\in\cupp\limits_{j\geq j_\psi}\hat\Gamma^j_\psi(\cdot)}W(\nabla \psi(\cdot)\mid\zeta)=W_0(\nabla \psi(\cdot))
\end{equation}
gr\^ace \`a (\ref{Cst1prime})  et au fait que $\cup_{j\geq j_\psi}\hat\Gamma^j_\psi(\cdot)=\big\{\zeta\in\RR^3:\det(\nabla\psi(\cdot)\mid\zeta)\not=0\big\}$. D'o\`u (\ref{ineQuality}) suit de (\ref{inequality1}) et ({\ref{equality}}) en utilisant le th\'eor\`eme de la convergence monotone.
\end{proof}

\subsubsection{D\'emonstration de la proposition {\rm\ref{det>0Prop2}}}

(On utilise les notations $\Det(\xi\mid\zeta)=|\det(\xi\mid\zeta)|$ et $\mathcal{F}=\Aff_*(\Sigma;\RR^3)$. On peut travailler avec $\Aff_*(\Sigma;\RR^3)$ au lieu de $\Aff(\Sigma;\RR^3)$ gr\^ace \`a la remarque \ref{Aff=Aff_*LoRsQuEW_0veRifieAdj}.) Il suffit de prouver que 
\begin{equation}\label{limsupequality}
\limsup_{\eps\to 0}\mathcal{I}_\eps(\psi)\leq\mathcal{I}(\psi)
\end{equation}  
pour tout $\psi\in\mathcal{F}$ (puisque $\Gamma\hbox{-}\limsup_{\eps\to 0}\mathcal{I}_\eps$ est sfsci dans $W^{1,p}(\Sigma;\RR^3)$, voir \cite{dalmaso93}). \'Etant donn\'e $\psi\in\mathcal{F}$, consid\'erons  $j\geq j_\psi$ (avec $j_\psi$ donn\'e par (\ref{Detnot=0Prop5})) et $n\geq 1$. Utilisant le th\'eor\`eme \ref{IntermediateTheorDetNot} on obtient l'existence de $\varphi\in C(\overline{\Sigma};\hat\Gamma^j_\psi)$ tel que 
\begin{equation}\label{mediainequality}
\int_\Sigma W(\nabla \psi(x)\mid\varphi(x))dx\leq\mathcal{I}(\psi)+{1\over n}.
\end{equation}
Soit $\{\varphi_k\}_{k\geq 1}\subset C^\infty(\overline{\Sigma};\RR^3)$ tel que 
\begin{equation}\label{uniformityconvergence}
\varphi_k\to\varphi\hbox{ uniform\'ement.}
\end{equation}
On affirme que :
\begin{eqnarray}\label{CoNdItioN(c1)}
&&\Det(\nabla \psi(x)\mid\varphi_k(x))\geq{1\over 2j}\hbox{ pour tout }x\in V\hbox{ et tout }k\geq k_\psi\hbox{ avec }k_\psi\geq 1\ ;\\
&&\lim\limits_{k\to+\infty}\int_\Sigma W(\nabla\psi(x)\mid\varphi_k(x))dx=\int_\Sigma W(\nabla\psi(x)\mid\varphi(x))dx.\label{CoNdItioN(c2)}
\end{eqnarray}
En effet, posant $\mu_\psi:=\sup_{x\in V}|\partial_1\psi(x)\land\partial_2\psi(x)|=\max_{i\in I}|\xi_{i,1}\land\xi_{i,2}|$ ($\mu_\psi>0$) et utilisant (\ref{uniformityconvergence}), on d\'eduit qu'il existe $k_\psi\geq 1$ tel que 
\begin{equation}\label{supequality}
\sup_{x\in\overline{\Sigma}}|\varphi_k(x)-\varphi(x)|<{1\over 2j\mu_\psi}
\end{equation}
pour tout $k\geq k_\psi$. Soient $x\in V$ et $k\geq k_\psi$. Comme $\varphi\in C(\overline{\Sigma};\hat\Gamma^j_\psi)$ on a :
\begin{equation}\label{supequality1}
\Det(\nabla\psi(x)\mid\varphi_k(x))\geq{1\over j}-\Det(\nabla\psi(x)\mid\varphi_k(x)-\varphi(x)).
\end{equation}
Remarquant que $\Det(\nabla\psi(x)\mid\varphi_k(x)-\varphi(x))\leq|\partial_1\psi(x)\land\partial_2\psi(x)||\varphi_k(x)-\varphi(x)|$, de (\ref{supequality}) et (\ref{supequality1}) on d\'eduit que
$
\Det(\nabla\psi(x)\mid\varphi_k(x))\geq{1\over 2j}
$
et (\ref{CoNdItioN(c1)}) est prouv\'ee. Combinant (\ref{CoNdItioN(c1)}) et (\ref{Hyp1Example:m=N}) on voit que
$
\sup_{k\geq k_\psi}W(\nabla \psi(\cdot)\mid\varphi_k(\cdot))\in L^1(\Sigma).
$ 
Comme $W$ est continue on a
$
\lim_{k\to+\infty}W(\nabla \psi(x)\mid\varphi_k(x))=W(\nabla \psi(x)\mid\varphi(x))
$
pour tout $x\in V$ et (\ref{CoNdItioN(c2)}) suit par le th\'eor\`eme de la convergence domin\'ee de Lebesgue. 

Consid\'erons $k\geq k_\psi$ et d\'efinissons la fonction continue  $\theta:]-{1\over 2},{1\over 2}[\to\RR$ par 
$
\theta(x_3):=\min_{i\in I}\inf_{x\in \overline{V}_i}\Det(\xi_i+x_3\nabla\varphi_k(x)\mid\varphi_k(x)).
$
Par (\ref{CoNdItioN(c1)}) on a $\theta(0)\geq{1\over 2j}$ et par cons\'equent il existe $\eta_\psi\in]0,{1\over 2}[$ tel que $\theta(x_3)\geq{1\over 4j}$ pour tout $x_3\in]-\eta_\psi,\eta_\psi[$. Soit $\phi_k:\Sigma_1\to\RR$ donn\'ee par
$
\phi_k(x,x_3):=\psi(x)+x_3\varphi_k(x).
$
Il suit que 
\begin{eqnarray}\label{CoNdItioN(c3)}
&& \Det(\nabla \phi_k(x,\eps x_3))\geq{1\over 4j}\hbox{ pour tout }\eps\in]0,\eta_\psi[\hbox{ et tout }(x,x_3)\in V\times]-{1\over 2},{1\over 2}[.
\end{eqnarray}
De la m\^eme fa\c con que dans la preuve de (\ref{CoNdItioN(c2)}), combinant (\ref{CoNdItioN(c3)}) et (\ref{Hyp1Example:m=N}) et utilisant la continuit\'e de $W$, on obtient 
\begin{equation}\label{finalequality}
\lim_{\eps\to 0}I_\eps(\phi_k)=\lim_{\eps\to 0}\int_{\Sigma_1} W(\nabla \phi_k(x,\eps x_3))dxdx_3=\int_\Sigma W(\nabla \psi(x)\mid\varphi_k(x))dx.
\end{equation}

 Puisque $\pi_\eps(\phi_k)=\psi$ on a $\mathcal{I}_\eps(\psi)\leq I_\eps(\phi_k)$ pour tout $\eps>0$ et tout $k\geq k_\psi$. Utilisant (\ref{finalequality}), (\ref{CoNdItioN(c2)}) et (\ref{mediainequality}), on d\'eduit que
$
\limsup_{\eps\to 0}\mathcal{I}_\eps(\psi)\leq\mathcal{I}(\psi)+{1\over n},
$
et (\ref{limsupequality}) suit en faisant $n\to+\infty$.\hfill$\square$

\begin{remark}\label{ample-integrand}\'Etant donn\'e $\eps>0$, on consid\`ere $W_\eps:\MM^{m\times N}\to[0,+\infty]$ Borel mesurable  et on introduit la d\'efinition suivante.
\begin{definition}\label{DimensionReductionAmpleDefinition}
Soit $p>1$. On dit que $W_\eps:\MM^{m\times N}\to[0,+\infty]$ est $p$-ample si $\Z W_\eps(F)\leq c_\eps(1+|F|^p)$ pour tout $F\in\MM^{m\times N}$ avec $c_\eps>0$.\end{definition}
(Pour tout $F\in\MM^{m\times N}$, $\Z W_\eps(F):=\inf\{\int_Y W_\eps(F+\nabla\varphi(y))dy:\varphi\in\Aff_0(Y;\RR^m)\}$.) On d\'efinit $I_\eps, \Z I_\eps:W^{1,p}(\Sigma_\eps;\RR^m)\to[0,+\infty]$ par :
\begin{itemize}
\item[$\diamond$] $\displaystyle I_\eps(\phi):=\int_{\Sigma_\eps} W_\eps(\nabla\phi(y))dy$ ;
\item[$\diamond$] $\displaystyle \Z I_\eps(\phi):=\int_{\Sigma_\eps}\Z W_\eps(\nabla\phi_\eps(y))dy$,
\end{itemize}
o\`u $\Sigma_\eps\subset\RR^N$ est un ouvert born\'e, et on consid\`ere  $\pi=\{\pi_\eps\}_\eps$ une famille d'applica-tions de $W^{1,p}(\Sigma_\eps;\RR^m)$ dans $W^{1,p}(\Sigma;\RR^m)$ avec $\Sigma\subset\RR^k$ un ouvert born\'e. Le th\'eor\`eme suivant justifie la d\'efinition \ref{DimensionReductionAmpleDefinition}. La famille $\{W_\eps\}_\eps$ est suppos\'ee uniform\'ement coercive, i.e., $W_\eps(F)\geq C|F|^p$ pour tout $F\in\MM^{m\times N}$ et tout $\eps>0$ avec $C>0$.
\begin{theorem}\label{DimensionReductionAmpleTheorem}
Supposons que :
\begin{itemize}
\item[$\diamond$] pour chaque $\eps>0$, $W_\eps$ est $p$-ample ;
\item[$\diamond$] il existe $I_0:W^{1,p}(\Sigma;\RR^m)\to[0,+\infty]$ telle que $I_0=\Gamma(\pi)\hbox{-}\lim\limits_{\eps\to 0}\Z I_\eps$.
\end{itemize}
Alors $I_0=\Gamma(\pi)$-$\lim\limits_{\eps\to 0} I_\eps$.
\end{theorem}
\begin{proof}
On a $\Gamma(\pi)$-$\lim_{\eps\to0}I_\eps=\Gamma(\pi)$-$\lim_{\eps\to 0}\overline{I}_\eps$, o\`u $\overline{I}_\eps:W^{1,p}(\Sigma_\eps;\RR^m)\to[0,+\infty]$ est la r\'egularis\'ee sci par rapport \`a la convergence faible de $W^{1,p}(\Sigma_\eps;\RR^m)$. Or, pour chaque $\eps>0$, $W_\eps$ est $p$-ample, donc, par le th\'eor\`eme \ref{GeneralThRelax},  $\overline{I}_\eps=\Z I_\eps$ pour tout $\eps>0$. D'o\`u $\Gamma(\pi)$-$\lim_{\eps\to0}I_\eps=\Gamma(\pi)$-$\lim_{\eps\to 0}\Z I_\eps$, et le r\'esultat suit.
\end{proof}
Le r\'esultat suivant est une cons\'equence du th\'eor\`eme \ref{DimensionReductionAmpleTheorem} (rappelons que la fonction $W$ est suppos\'ee coercive).
\begin{theorem}\label{3D-2DTheoremBis}
Si $W_\eps={1\over\eps} W$ pour tout $\eps>0$ et si $W:\MM^{3\times 3}\to[0,+\infty]$ satisfait {(\ref{Hyp1Example:m=N})} alors $I_{\rm mem}=\Gamma(\pi)\hbox{\rm -}\lim_{\eps\to 0}I_\eps$ avec $W_{\rm mem}=\mathcal{Q} W_0=\Z W_0$.
\end{theorem}
\begin{proof}
Puisque $W$ est coercive, la famille $\{W_\eps\}_\eps$ est uniform\'ement coercive.  Ici $m=N=3$, $k=2$, $\Sigma_\eps=\Sigma\times]-{\eps\over 2},{\eps\over 2}[$ avec $\Sigma\subset\RR^2$ un ouvert born\'e et $\pi=\{\pi_\eps\}_\eps$ avec $\pi_\eps:W^{1,p}(\Sigma_\eps;\RR^3)\to W^{1,p}(\Sigma;\RR^3)$ d\'efinie par (\ref{PourAJOUTRmK}). Comme $W$ satisfait (\ref{Hyp1Example:m=N}), par le corollaire \ref{CorollaryExample:m=NBis}, on a $\Z W(F)\leq c(1+|F|^p)$ pour tout $F\in\MM^{3\times 3}$ avec $c>0$ (donc $\Z W$ est finie). Or pour chaque $\eps>0$, $\Z W_\eps={1\over \eps}\Z W$, donc $\Z W_\eps(F)\leq {c\over\eps}(1+|F|^p)$ pour tout $F\in\MM^{3\times 3}$ et tout $\eps>0$, i.e., pour chaque $\eps>0$, $W_\eps$ est $p$-ample. D'autre part, $\Z W$ est continue d'apr\`es la proposition \ref{FonsecaProperties}(c), donc, par le th\'eor\`eme \ref{LeDretRaoult}, on a $I_0=\Gamma(\pi)$-$\lim_{\eps\to 0} \Z I_\eps$ avec $I_0:W^{1,p}(\Sigma;\RR^3)\to[0,+\infty]$ d\'efinie par $I_0(\psi):=\int_\Sigma\Z[\Z W]_0(\nabla\psi(x))dx$, o\`u $[\Z W]_0:\MM^{3\times 2}\to[0,+\infty]$ est donn\'ee par $[\Z W]_0(\xi):=\inf\{\Z W(\xi\mid\zeta):\zeta\in\RR^3\}$, et le th\'eor\`eme suit du  th\'eor\`eme \ref{DimensionReductionAmpleTheorem} en montrant que 
\begin{equation}\label{DimenSIONalReDuCTIonEqProof}
\Z[\Z W]_0=\Z W_0.
\end{equation}
(Puisque  $\Z W$ est finie, $\Z W_0$ l'est aussi par (\ref{DimenSIONalReDuCTIonEqProof}), donc $\mathcal{Q} W_0=\Z W_0$ par le th\'eor\`eme \ref{QuasiconvexificationFormulaTheorem}-bis.) 

Prouvons (\ref{DimenSIONalReDuCTIonEqProof}). Pour chaque $\xi\in\MM^{3\times 2}$, on a $\Z[\Z W]_0(\xi)\leq[\Z W]_0(\xi)\leq\Z W(\xi\mid\zeta)\leq W(\xi\mid\zeta)$ pour tout $\zeta\in\RR^3$, donc $\Z[\Z W]_0(\xi)\leq W_0(\xi)$ pour tout $\xi\in\MM^{3\times 2}$, i.e., $\Z[\Z W]_0\leq\Z W_0$. Il suit que $\Z[\Z W]_0\leq\Z W_0$. D'autre part, \'etant donn\'es $\eps>0$ et $\xi\in\MM^{3\times 2}$, il existe $\zeta\in\RR^3$ et $\varphi\in \Aff_0(Y;\RR^3)$ (avec $Y:=]0,1[^3$) tels que 
\[
\left[\Z W\right]_0(\xi)+\eps\ge \int_YW\big(\xi+\nabla\varphi_{x_3}(x)\mid\zeta+\partial_3\varphi(x,x_3)\big)dxdx_3
\]
avec $\varphi_{x_3}\in \Aff_0(]0,1[^2;\RR^3)$ d\'efinie par $\varphi_{x_3}(x):=\varphi(x,x_3)$. Or 
\begin{eqnarray*}
\int_YW\big(\xi+\nabla\varphi_{x_3}(x)\mid\zeta+\partial_3\varphi(x,x_3)\big)dxdx_3&\ge& \int_{0}^{1}\int_{]0,1[^2} W_0(\xi+\nabla\varphi_{x_3}(x))dxdx_3\\
&\ge&\int_{0}^{1}\Z W_0(\xi)dx_3=\Z W_0(\xi),
\end{eqnarray*}
donc $[\Z W]_0(\xi)+\eps\ge\Z W_0(\xi)$, d'o\`u  $[\Z W]_0(\xi)\ge \Z W_0(\xi)$ en faisant $\eps\to 0$. Il suit que $[\Z W]_0\ge \Z W_0$ et par cons\'equent $\Z[\Z W]_0\ge \Z W_0$.
\end{proof}

Le th\'eor\`eme \ref{3D-2DTheoremBis} est ``meilleur" que le th\'eor\`eme \ref{ThMemb1} car on n'a besoin ni de la continuit\'e de $W$, ni de l'hypoth\`ese de ``sym\'etrie" \eqref{AdditionalCond}, ni m\^eme de la condition (\ref{Cst1prime}). Cependant, on ne peut pas appliquer le th\'eor\`eme \ref{DimensionReductionAmpleTheorem} pour traiter  la contrainte d\'eterminant strictement positif puisque lorsque $W$ satisfait {(\ref{Cst1})} et {(\ref{Hyp3Example:m=N})}, les $W_\eps={1\over\eps}W$ ne sont pas $p$-amples. 
\end{remark}
\subsection{Contrainte d\'eterminant strictement positif}

Dans ce paragraphe on d\'e-montre la proposition \ref{det>0Prop3}  et  l'\'egalit\'e {(\ref{BenBelgacemBennequinGromovEliashbergEgalitŽ})}.

\subsubsection{Pr\'eliminaires}

\'Etant donn\'es $\psi\in C^1_*(\overline{\Sigma};\RR^3)$ et $j\geq 1$, on d\'efinit $\Lambda_\psi^j:\overline{\Sigma}\dto\RR^3$ par 
$$
\Lambda^j_\psi(x):=\left\{\zeta\in\RR^3:\det(\nabla\psi(x)\mid\zeta)\geq{1\over j}\right\}.
$$
On peut voir que :
 \begin{eqnarray}\label{Det>0Prop1}
&&\Lambda^1_\psi(x)\subset\Lambda^2_\psi(x)\subset\cdots\subset\cupp\limits_{j\geq 1}\Lambda^j_\psi(x)=\big\{\zeta\in\RR^3:\det(\nabla\psi(x)\mid\zeta)>0\big\}\ ;\\
&&\Lambda^j_\psi\hbox{ est une multifonction  sci \`a valeurs convexes ferm\'ees non vides.}\label{Det>0Prop1BIs}
 \end{eqnarray}
(Ici, on ne peut pas prendre $\psi$ dans $\Aff_*(\Sigma;\RR^3)$. En effet, on serait amen\'e \`a consid\'erer seulement $U^+_{i,j}$  ce qui ne permettrait pas de reproduire la d\'emarche d\'evelopp\'ee en \S 3.2.1.) 

\begin{proof}[Preuve de (\ref{Det>0Prop1BIs})]
 Il est facile de voir que pour chaque  $x\in\overline{\Sigma}$, $\Lambda_\psi^j(x)$ est \`a valeurs convexes fermŽes non vides. Montrons que $\Lambda^j_\psi$ est sci (voir la d\'efinition \ref{DefMultifonctionSCI}). Soient $F$ un ferm\'e de $\RR^3$, $x\in\overline{\Sigma}$ et $\{x_n\}_{n\geq 1}\subset\overline{\Sigma}$ tels que $x_n\to x$ et $\Lambda^j_{\psi}(x_n)\subset F$ pour tout $n\geq 1$. Soient $\zeta\in \Lambda_\psi^j(x)$ et $\{\zeta_m\}_{m\geq 1}\subset\RR^3$ donn\'es par $\zeta_m:=\zeta+{1\over m}\zeta$. Alors 
\begin{equation}\label{DetEq}
\det\big(\nabla \psi(x)\mid\zeta_m\big)=\det\big(\nabla \psi(x)\mid\zeta\big)+{1\over m}\det\big(\nabla \psi(x)\mid\zeta\big)\geq{1\over j}+{1\over mj}
\end{equation}
pour tout $m\geq 1$. \'Etant donn\'e $m\geq 1$, puisque $\det(\nabla \psi(x_n)\mid\zeta_m)\to\det(\nabla \psi(x)\mid\zeta_m)$, utilisant (\ref{DetEq}) on voit qu'il existe $n_0\geq 1$ tel que $\det(\nabla \psi(x_{n_0})\mid\zeta_m)>{1\over j}$, donc $\zeta_m\in\Lambda_\psi^j(x_{n_0})$. D'o\`u $\zeta_m\in F$ pour tout $m\geq 1$. Comme $F$ est ferm\'e on a $\zeta=\lim_{m\to +\infty}\zeta_m\in F$.
\end{proof}

Le lemme suivant est une cons\'equence du th\'eor\`eme \ref{InfIntTH2}. (Il se prouve de la m\^eme fa\c con que le lemme \ref{LemmeDeTnot=0}.)
\begin{lemma}\label{LemmeDeT>0}
Soient $\psi\in C^1_*(\overline{\Sigma};\RR^3)$ et $j\geq 1$. Si $W$ est continue et satisfait {\rm(\ref{Hyp3Example:m=N})} alors
\begin{eqnarray}\label{EquALitYLemmeDeT>0}
\inf_{\varphi\in C(\overline{\Sigma};\Lambda_\psi^j)}\int_\Sigma W(\nabla\psi(x)\mid\varphi(x))dx=\int_\Sigma \inf_{\zeta\in\Lambda_\psi^j(x)}W(\nabla\psi(x)\mid\zeta)dx.
\end{eqnarray}
\end{lemma}
\begin{proof}
On r\'ecrit la preuve du lemme \ref{LemmeDeT>0} en prenant  $\Gamma^j_\psi=\Lambda^j_\psi$ ($\Lambda^j_\psi$ est une multifonction  sci \`a valeurs convexes ferm\'ees non vides), $\Det(\xi\mid\zeta)=\det(\xi\mid\zeta)$ et en utilisant (\ref{Hyp3Example:m=N}) \`a la place de {\rm(\ref{Hyp1Example:m=N})}. On peut ainsi appliquer le th\'eor\`eme \ref{InfIntTH2} avec $f(x,\zeta)=W(\nabla\psi(x)\mid\zeta)$ et $\Gamma=\Lambda^j_\psi$ et on obtient (\ref{EquALitYLemmeDeT>0}).
\end{proof}

 Le th\'eor\`eme suivant donne une repr\'esentation ``non int\'egrale" de $\mathcal{I}$ sur $C_*^1(\overline{\Sigma};\RR^3)$ (voir \cite[Theorem 3.4]{oah-jpm08b}). (Il se d\'emontre de la m\^eme fa\c con que le th\'eor\`eme \ref{IntermediateTheorDetNot}.)
\begin{theorem}\label{IntermediateTheorDet>}
Si $W$ est continue et v\'erifie {\rm(\ref{Cst1})} et {\rm(\ref{Hyp3Example:m=N})} et si $\psi\in C^1_*(\overline{\Sigma};\RR^3)$ alors
$$
\mathcal{I}(\psi)=\inf_{j\geq 1}\inf_{\varphi\in C(\overline{\Sigma};\Lambda^j_\psi)}\int_\Sigma W(\nabla\psi(x)\mid\varphi(x))dx.
$$
\end{theorem}
\begin{proof}
On r\'ecrit la preuve du th\'eor\`eme \ref{IntermediateTheorDetNot} en prenant $j_\psi=1$, $\Gamma^j_\psi=\hat\Gamma_\psi^j=\Lambda^j_\psi$ (donc $\cup_{j\geq j_\psi}\hat\Gamma^j_\psi(\cdot)=\cup_{j\geq j_\psi}\Lambda^j_\psi(\cdot)=\{\zeta\in\RR^3:\det(\nabla\psi(\cdot)\mid\zeta)>0\}$ dans (\ref{equality})) et en utilisant le lemme \ref{LemmeDeT>0} \`a la place du lemme \ref{LemmeDeTnot=0}, (\ref{Hyp3Example:m=N}) \`a la place de (\ref{Hyp1Example:m=N}), (\ref{Det>0Prop1}) \`a la place de (\ref{Detnot=0Prop3})-(\ref{Detnot=0Prop4}) et (\ref{Cst1}) \`a la place de (\ref{Cst1prime}).
\end{proof}

\subsubsection{D\'emonstration de la proposition {\rm\ref{det>0Prop3}}}

On r\'ecrit la preuve de la proposition \ref{det>0Prop2} en prenant $\Det(\xi\mid\zeta)=\det(\xi\mid\zeta)$, $\mathcal{F}=C^1_*(\overline{\Sigma};\RR^3)$, $j_\psi=1$, $\hat\Gamma^j_\psi=\Lambda^j_\psi$, $V=V_i=\overline{\Sigma}$ et en utilisant (\ref{Hyp3Example:m=N}) \`a la place de (\ref{Hyp1Example:m=N}).\hfill$\square$

\subsubsection{Preuve de {(\ref{BenBelgacemBennequinGromovEliashbergEgalitŽ})}}
On termine la d\'emonstration du th\'eor\`eme \ref{ThMemb2} (voir \S 3.1.3) en utilisant le th\'eor\`eme suivant.

\begin{theorem}\label{TheoEgalitŽ(37)}
Si $W_0$ est continue {(}voir le lemme {\ref{Lemme1Det>0}}) et satisfait {(\ref{Adj-Infty2})} alors {(\ref{BenBelgacemBennequinGromovEliashbergEgalitŽ})} a lieu, i.e., $\overline{\mathcal{I}}=\overline{\mathcal{I}}_{\rm diff_*}
$.
\end{theorem}
\begin{proof}
Soient $\overline{\mathcal{I}}_{\AffETli}, \overline{\mathcal{R I}}_{\AffETli}, \overline{\mathcal{R I}}:W^{1,p}(\Sigma;\RR^3)\to[0,+\infty]$ d\'efinies par :
\begin{itemize}
\item[$\diamond$] $\displaystyle\overline{\mathcal{I}}_{\AffETli}(\psi):=\inf\left\{\liminf_{n\to+\infty}\int_\Sigma W_0(\nabla\psi_n(x))dx:\AffETli(\Sigma;\RR^3)\ni\psi_n\wto\psi\right\};$
\item[$\diamond$] $\displaystyle\overline{\mathcal{R I}}_{\AffETli}(\psi):=\inf\left\{\liminf_{n\to+\infty}\int_\Sigma \mathcal{R}W_0(\nabla\psi_n(x))dx:\AffETli(\Sigma;\RR^3)\ni\psi_n\wto\psi\right\};$
\item[$\diamond$] $\displaystyle\overline{\mathcal{R I}}(\psi):=\inf\left\{\liminf_{n\to+\infty}\int_\Sigma \mathcal{R}W_0(\nabla\psi_n(x))dx:W^{1,p}(\Sigma;\RR^3)\ni\psi_n\wto\psi\right\}$
\end{itemize}
avec $\AffETli(\Sigma;\RR^3)$ d\'efini en \S 4.1.4 et $\mathcal R W_0$  d\'esignant la rang-1 convexifi\'e de $W_0$ (la plus grande fonction rang-1 convexe qui est inf\'erieure \`a $W_0$). Il est clair que $\overline{\mathcal{R I}}_{\AffETli}\leq \overline{\mathcal{I}}_{\AffETli}$. De plus, d'apr\`es Ben Belgacem (voir \cite{benbelgacem96}) on a 
\begin{lemma}\label{BenBelgacemLemma}
$\displaystyle\overline{\mathcal{I}}_{\AffETli}(\psi)\leq\int_\Sigma\mathcal{R}W_0(\nabla\psi(x))dx$ pour tout $\psi\in\AffETli(\Sigma;\RR^3)$.
\end{lemma}
(Pour une preuve du lemme \ref{BenBelgacemLemma} voir \S 3.3.4-3.3.5.) Donc $\overline{\mathcal{I}}_{\AffETli}\leq \overline{\mathcal{R I}}_{\AffETli}$, d'o\`u $\overline{\mathcal{I}}_{\AffETli}=\overline{\mathcal{R I}}_{\AffETli}$. D'autre part, $\overline{\mathcal{I}}\leq\overline{\mathcal{I}}_{\rm diff_*}$ et $\overline{\mathcal{RI}}\leq \overline{\mathcal{I}}$. Donc, pour avoir (\ref{BenBelgacemBennequinGromovEliashbergEgalitŽ}) il suffit  de prouver les deux in\'egalit\'es suivantes :
\begin{eqnarray}\label{GEBBeQuAliTY1}
&&\overline{\mathcal{I}}_{\rm diff_*}\leq\overline{\mathcal{I}}_{\AffETli}\ ;\\
&&\overline{\mathcal{RI}}_{\AffETli}\leq\overline{\mathcal{RI}}.\label{GEBBeQuAliTY2}
\end{eqnarray}

\begin{proof}[Preuve de (\ref{GEBBeQuAliTY1})]
Il suffit de prouver que 
\begin{equation}\label{AddTheorEq2}
\overline{\mathcal{I}}_{\rm diff_*}(\psi)\leq\int_{\Sigma}W_0(\nabla \psi(x))dx
\end{equation}
pour tout $\psi\in \AffETli(\Sigma;\RR^3)$. Soit $\psi\in \AffETli(\Sigma;\RR^3)$. Par le th\'eor\`eme \ref{BBBapproxTheo} il existe $\{\psi_n\}_{n\geq 1}\subset C^1_*(\overline{\Sigma};\RR^3)$ tel que  (\ref{BB_1}) et (\ref{BB_2}) sont satisfaites et $\nabla \psi_n(x)\to\nabla \psi(x)$ p.p. dans $\Sigma$. Comme $W_0$ est continue  on a 
$$
\lim_{n\to +\infty}W_0\big(\nabla \psi_n(x)\big)=W_0\big(\nabla \psi(x)\big)\;\hbox{ p.p. dans }\Sigma.
$$
Utilisant (\ref{Adj-Infty2}) et (\ref{BB_2}) on d\'eduit qu'il existe $c>0$  tel que pour chaque $n\geq 1$ et chaque ensemble mesurable $A\subset\Sigma$,
$$
\int_A W_0\big(\nabla \psi_n(x)\big)dx\leq c\Big(|A|+\int_A|\nabla \psi_n(x)-\nabla \psi(x)|^pdx+\int_A|\nabla \psi(x)|^pdx\Big).
$$
Or $\nabla v_n\to\nabla v$ dans $L^p(\Sigma;\MM^{3\times 2})$ par (\ref{BB_1}), donc $\{W_0(\nabla \psi_n(\cdot))\}_{n\geq 1}$ est uniform\'ement absolument int\'egrable. Utilisant le th\'eor\`eme de Vitali on obtient 
$$
\lim_{n\to+\infty}\int_{\Sigma}W_0(\nabla \psi_n(x))dx=\int_{\Sigma}W_0(\nabla \psi(x))dx,
$$ 
et (\ref{AddTheorEq2}) suit.
\end{proof}

\begin{proof}[Preuve de (\ref{GEBBeQuAliTY2})]
Il suffit de prouver que 
\begin{equation}\label{AddTheorEq4}
\overline{\mathcal{RI}}_{\AffETli}(\psi)\leq\int_{\Sigma}\mathcal{R}W_0(\nabla \psi(x))dx
\end{equation} 
pour tout $\psi\in W^{1,p}(\Sigma;\RR^3)$. Soit $\psi\in W^{1,p}(\Sigma;\RR^3)$. Par le th\'eor\`eme \ref{GEdensityTheo} il existe $\{\psi_n\}_{n\geq 1}\subset\AffETli(\Sigma;\RR^3)$ tel que $\nabla \psi_n\to\nabla \psi$ dans $L^p(\Sigma;\RR^3)$ et $\nabla \psi_n(x)\to\nabla \psi(x)$ p.p. dans $\Sigma$. Or, d'apr\`es Ben Belgacem (voir \cite{benbelgacem96,benbelgacem00}) on a  
\begin{lemma}\label{BenBelgacemLemma2}
Si $W_0$ satisfait {(\ref{Adj-Infty2})} alors {:}
\begin{itemize}
\item[$\diamond$]  ${\mathcal R}W_0(\xi)\leq c(1+|\xi|^p)$ pour tout $\xi\in\MM^{3\times 2}$ avec $c>0$ ;
\item[$\diamond$] ${\mathcal R}W_0$ est continue.
\end{itemize}
\end{lemma}
(Pour une preuve du lemme \ref{BenBelgacemLemma2} voir \S 3.3.7). D'o\`u, utilisant le th\'eor\`eme de Vitali, on d\'eduit que 
$$
\lim_{n\to+\infty}\int_{\Sigma}\mathcal{R}W_0(\nabla \psi_n(x))dx=\int_{\Sigma}\mathcal{R}W_0(\nabla \psi(x))dx,
$$
et (\ref{AddTheorEq4}) suit.
\end{proof}
 Ce qui termine la d\'emonstration du th\'eor\`eme \ref{TheoEgalitŽ(37)}. \end{proof}

\subsubsection{Pr\'eliminaires pour la d\'emonstration du lemme \ref{BenBelgacemLemma}} 
 (Pour la d\'emonstration du lemme \ref{BenBelgacemLemma} voir \S4.3.5.) Soit la suite $\{\R_i W_0\}_{i\geq 0}$ par $\R_0 W_0=W_0$ et pour tout $i\geq 1$, $\R_{i+1}W_0$ est d\'efinie par (\ref{Kohn-StrangFormula}) (avec $W=W_0$, $m=3$ et $N=2$). Rappelons que $W_0$ est continue et coercive (voir le lemme \ref{Lemme1Det>0}) et d'apr\`es Kohn et Strang (voir \cite{kohn-strang86}) on a 
\begin{eqnarray}\label{KohnStrangBBLemma1}
\R_{i+1} W_0\leq \R_iW_0\hbox{ pour tout }i\geq 0\hbox{ et }\R W_0=\inf_{i\geq 0} \R_i W_0.
\end{eqnarray}
Soient $i\geq 0$ et $\psi\in\AffETli(\Sigma;\RR^3)$ (pour la d\'efinition de $\AffETli(\Sigma;\RR^3)$ voir \S 4.1.4).  Alors, il existe une famille finie de sous-ensembles ouverts et disjoints de $\Sigma$ telle que $|\Sigma\setminus\cup_{j\in J}V_j|=0$  et pour chaque $j\in J$, $|\partial V_j|=0$ et $\nabla \psi(x)=\xi_j$ dans $V_j$ avec $\xi_j\in\MM^{3\times 2}$. (Comme $\psi$ est localement injective on a ${\rm rang}(\xi_j)=2$ pour tout $j\in J$.) Soit $j\in J$. On peut montrer que :
\begin{eqnarray}
&&\R_i W_0\hbox{ est continue ;} \label{EquAtiOn1BBLemma1Continue}\\
&&\R_{i+1}W_0(\xi_j)=(1-t)\R_i W_0(\xi_j-t a\otimes b)+t\R_i W_0(\xi_j+(1-t)a\otimes b) \label{EquAtiOn1BBLemma1}\\
&&\hbox{avec }a\in\RR^2,\ b\in\RR^3\hbox{ et }t\in[0,1]\nonumber
\end{eqnarray}
avec $a\otimes b\in\RR^2\otimes\RR^3\subset\MM^{3\times 2}$ donn\'e par $(a\otimes b)x:=\langle a,x\rangle b$ pour tout $x\in\RR^2$, o\`u $\langle\cdot,\cdot\rangle$ d\'esigne le produit scalaire dans $\RR^2$. (Pour une preuve de \eqref{EquAtiOn1BBLemma1Continue} et \eqref{EquAtiOn1BBLemma1} voir \S 3.3.6.)

Sans perdre de g\'en\'eralit\'e on peut supposer que $a=(1,0)$. Pour chaque $n\geq 1$ et chaque $k\in\{0,\cdots,n-1\}$, consid\'erons $A^-_{k,n}$, $A^+_{k,n}$, $B_{k,n}$, $B^-_{k,n}$, $B^+_{k,n}$, $C_{k,n}$, $C^-_{k,n}$, $C^+_{k,n}\subset Y$ (voir la figure \ref{FiGUREbbbLeMMa}) donn\'es par :
\begin{itemize}
\item[]$A^-_{k,n}:=\big\{(x_1,x_2)\in Y:{k\over n}\leq x_1\leq{k\over n}+{1-t\over n}\hbox{ et }{1\over n}\leq x_2\leq 1-{1\over n}\big\}$ ;
\item[]$A^+_{k,n}:=\big\{(x_1,x_2)\in Y:{k\over n}+{1-t\over n}\leq x_1\leq {k+1\over n}\hbox{ et }{1\over n}\leq x_2\leq 1-{1\over n}\big\}$ ;
\item[]$B_{k,n}:=\big\{(x_1,x_2)\in Y:{k\over n}\leq x_1\leq{k+1\over n}\hbox{ et }0\leq x_2\leq -x_1+{k+1\over n}\big\}$ ;
\item[]$B^-_{k,n}:=\big\{(x_1,x_2)\in Y:-x_2+{k+1\over n}\leq x_1\leq-tx_2+{k+1\over n}\hbox{ et }0\leq x_2\leq {1\over n}\big\}$ ;
\item[]$B^+_{k,n}:=\big\{(x_1,x_2)\in Y:-tx_2+{k+1\over n}\leq x_1\leq {k+1\over n}\hbox{ et }0\leq x_2\leq {1\over n}\big\}$ ;
\item[]$C_{k,n}:=\big\{(x_1,x_2)\in Y:{k\over n}\leq x_1\leq{k+1\over n}\hbox{ et }x_1+1-{k+1\over n}\leq x_2\leq 1\big\}$ ;
\item[]$C^-_{k,n}:=\big\{(x_1,x_2)\in Y:x_2-1+{k+1\over n}\leq x_1\leq t(x_2-1)+{k+1\over n}\hbox{ et }{n-1\over n}\leq x_2\leq 1\big\}\ ;$ 
\item[]$C^+_{k,n}:=\big\{(x_1,x_2)\in Y:t(x_2-1)+{k+1\over n}\leq x_1\leq {k+1\over n}\hbox{ et }{n-1\over n}\leq x_2\leq 1\big\}$
\end{itemize}
et d\'efinissons  $\{\sigma_{n}\}_{n\geq 1}\subset\AffET_0(Y;\RR)$ par 
$$
\sigma_{n}(x_1,x_2):=\left\{
\begin{array}{ll}
-t(x_1-{k\over n})&\hbox{si }(x_1,x_2)\in A^-_{k,n}\\
(1-t)(x_1-{k+1\over n})&\hbox{si }(x_1,x_2)\in A^+_{k,n}\cup B^+_{k,n}\cup C^+_{k,n}\\
-t(x_1+x_2-{k+1\over n})&\hbox{si }(x_1,x_2)\in B^-_{k,n}\\
-t(x_1-x_2+1-{k+1\over n})&\hbox{si }(x_1,x_2)\in C^-_{k,n}\\
0&\hbox{si }(x_1,x_2)\in B_{k,n}\cup C_{k,n}.
\end{array}
\right.
$$

\begin{figure}[H]
\begin{center}
\begin{picture}(300,255)
\put(-5,40){\vector(1,0){210}}
\put(195,34){\SMALL$x_1$}
\put(-10,240){\SMALL$x_2$}
\put(0,39.9){\line(1,0){180}}
\put(0,40.1){\line(1,0){180}}
\put(0,40){\circle*{3}}
\put(-5,34){\SMALL$0$}
\put(20,40){\circle*{3}}
\put(12,31){\SMALL$1-t\over n$}
\put(30,40){\circle*{3}}
\put(26.5,31){\SMALL$1\over n$}
\put(50,40){\circle*{3}}
\put(42,31){\SMALL${2-t\over n}$}
\put(60,40){\circle*{3}}
\put(56.5,31){\SMALL$2\over n$}
\put(71,32){\SMALL$\cdots$}
\put(90,40){\circle*{3}}
\put(86.5,31){\SMALL$k\over n$}
\put(120,40){\circle*{3}}
\put(115,31){\SMALL$k+1\over n$}
\put(93,31){\Tiny$k+1-t\over n$}
\put(110,40){\circle*{3}}
\put(131.5,32){\SMALL$\cdots$}
\put(150,40){\circle*{3}}
\put(144,31){\SMALL$n-1\over n$}
\put(170,40){\circle*{3}}
\put(180,40){\circle*{3}}
\put(162,31){\SMALL$n-t\over n$}
\put(178,32){\SMALL$1$}
\put(0,35){\vector(0,1){210}}
\put(0,220){\circle*{3}}
\put(-5,218){\SMALL$1$}
\put(0,190){\circle*{3}}
\put(-17,188){\SMALL$n-1\over n$}
\put(0,70){\circle*{3}}
\put(-8,68){\SMALL$1\over n$}
\put(0,220){\line(1,0){180}}
\put(0,220.1){\line(1,0){180}}
\put(0,219.9){\line(1,0){180}}
\put(0,190){\line(1,0){20}}
\put(30,190){\line(1,0){20}}
\put(60,190){\line(1,0){20}}
\put(90,190){\line(1,0){20}}
\put(120,190){\line(1,0){20}}
\put(150,190){\line(1,0){20}}
\put(150,70){\line(1,0){20}}
\put(120,70){\line(1,0){20}}
\put(90,70){\line(1,0){20}}
\put(60,70){\line(1,0){20}}
\put(30,70){\line(1,0){20}}
\put(0,70){\line(1,0){20}}
\put(-0.1,40){\line(0,1){180}}
\put(0.1,40){\line(0,1){180}}
\put(180,40){\line(0,1){180}}
\put(180.1,40){\line(0,1){180}}
\put(179.9,40){\line(0,1){180}}
\put(150,40){\line(0,1){180}}
\put(150.1,40){\line(0,1){180}}
\put(149.9,40){\line(0,1){180}}
\put(120,40){\line(0,1){180}}
\put(120.1,40){\line(0,1){180}}
\put(119.9,40){\line(0,1){180}}
\put(90,40){\line(0,1){180}}
\put(90.1,40){\line(0,1){180}}
\put(89.9,40){\line(0,1){180}}
\put(60,40){\line(0,1){180}}
\put(60.1,40){\line(0,1){180}}
\put(59.9,40){\line(0,1){180}}
\put(30,40){\line(0,1){180}}
\put(30.1,40){\line(0,1){180}}
\put(29.9,40){\line(0,1){180}}
\put(170,70){\line(0,1){120}}
\put(140,70){\line(0,1){120}}
\put(110,70){\line(0,1){120}}
\put(80,70){\line(0,1){120}}
\put(50,70){\line(0,1){120}}
\put(20,70){\line(0,1){120}}
\put(0,70){\line(1,-1){30}}
\put(30,70){\line(1,-1){30}}
\put(60,70){\line(1,-1){30}}
\put(90,70){\line(1,-1){30}}
\put(120,70){\line(1,-1){30}}
\put(150,70){\line(1,-1){30}}
\put(0,190){\line(1,1){30}}
\put(30,190){\line(1,1){30}}
\put(60,190){\line(1,1){30}}
\put(90,190){\line(1,1){30}}
\put(120,190){\line(1,1){30}}
\put(150,190){\line(1,1){30}}
\put(20,70){\line(1,-3){10}}
\put(50,70){\line(1,-3){10}}
\put(80,70){\line(1,-3){10}}
\put(110,70){\line(1,-3){10}}
\put(140,70){\line(1,-3){10}}
\put(170,70){\line(1,-3){10}}
\put(20,190){\line(1,3){10}}
\put(50,190){\line(1,3){10}}
\put(80,190){\line(1,3){10}}
\put(110,190){\line(1,3){10}}
\put(140,190){\line(1,3){10}}
\put(170,190){\line(1,3){10}}
\put(260,40){\line(1,0){30}}
\put(260,40.1){\line(1,0){30}}
\put(260,39.9){\line(1,0){30}}
\put(235,52){\Tiny$B_{k,n}$}\put(252,52){\vector(1,0){17}}
\put(235,208){\Tiny$C_{k,n}$}\put(252,208){\vector(1,0){17}}
\put(235,128){\Tiny$A^-_{k,n}$}\put(252,128){\vector(1,0){20}}
\put(307,128){\Tiny$A^+_{k,n}$}\put(304,128){\vector(-1,0){20}}
\put(307.5,198){\Tiny$C^+_{k,n}$}\put(304.5,198){\vector(-1,0){20}}
\put(307.5,62){\Tiny$B^+_{k,n}$}\put(304.5,62){\vector(-1,0){20}}
\put(235,62){\Tiny$B^-_{k,n}$}\put(252,62){\vector(1,0){25}}
\put(235,198){\Tiny$C^-_{k,n}$}\put(252,198){\vector(1,0){25}}
\put(260,220){\line(1,0){30}}
\put(260,220.1){\line(1,0){30}}
\put(260,219.9){\line(1,0){30}}
\put(260,40){\circle*{3}}
\put(256.5,31){\SMALL$k\over n$}
\put(290,40){\circle*{3}}
\put(285,31){\SMALL$k+1\over n$}
\put(263,31){\Tiny$k+1-t\over n$}
\put(280,40){\circle*{3}}
\put(290,40){\line(0,1){180}}
\put(290.1,40){\line(0,1){180}}
\put(289.9,40){\line(0,1){180}}
\put(260,40){\line(0,1){180}}
\put(260.1,40){\line(0,1){180}}
\put(259.9,40){\line(0,1){180}}
\put(260,70){\line(1,-1){30}}
\put(280,70){\line(1,-3){10}}
\put(280,190){\line(1,3){10}}
\put(260,190){\line(1,1){30}}
\put(280,70){\line(0,1){120}}
\put(260,70){\line(1,0){30}}
\put(260,190){\line(1,0){30}}
\end{picture}
\end{center}
\caption{Les ensembles $\Delta^s_1$, $\Delta^s_2$, $\Delta^s_3$, $\Delta^s_4$, $\Delta^s_5$, $\Delta^s_6$, $\Delta^s_7$, $\Delta^s_8$.}
\label{FiGUREbbbLeMMa}
\end{figure}
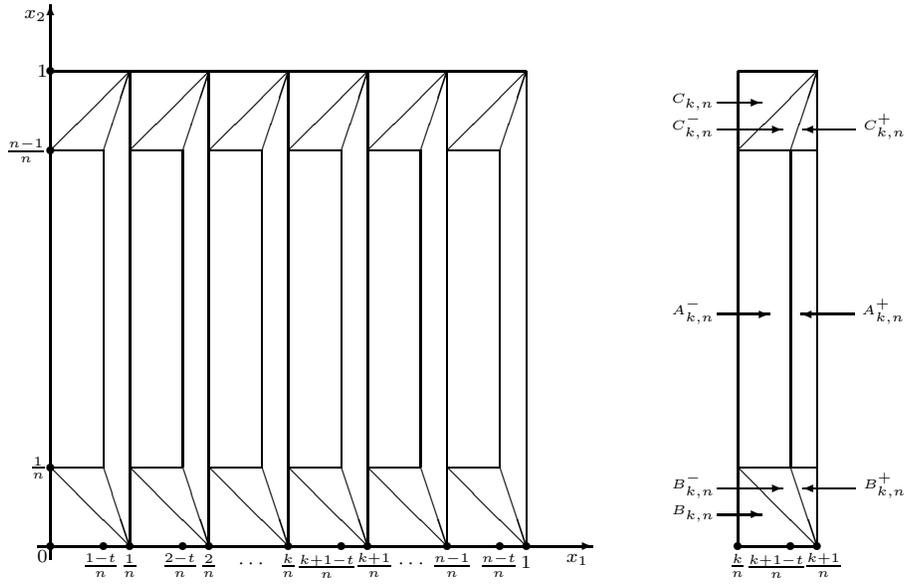

Posons 
$$
b_{\ell}:=\left\{
\begin{array}{ll}
b&\hbox{si }b\not\in{\rm Im}\xi_j\\
b+{1\over\ell}\nu&\hbox{si }b\in{\rm Im}\xi_j
\end{array}
\right.
$$
(avec ${\rm Im}\xi_j:=\{\xi_j\cdot x:x\in\RR^2\}\subset\RR^3$) o\`u $\ell\geq 1$ et $\nu\in\RR^3$ est orthogonal \`a  ${\rm Im}\xi_j$ et d\'efinissons $\{\theta_{n,\ell}\}_{n,\ell\geq 1}\subset\AffET_0(Y;\RR^3)$ d\'efinie par 
$$
\theta_{n,\ell}(x):=\sigma_{n}(x)b_{\ell}.
$$
Alors 
\begin{eqnarray}
&&\displaystyle\lim_{\ell\to+\infty}\lim_{n\to+\infty}\int_Y\R_iW_0(\xi_j+\nabla \theta_{n,\ell}(x))dx=\R_{i+1}W_0(\xi_j).\label{EquBBLeMMa2}
\end{eqnarray}

\begin{proof}[preuve de (\ref{EquBBLeMMa2})]
Il est facile de voir que 
$$
\xi_j+\nabla\theta_{n,\ell}(x):=\left\{
\begin{array}{ll}
\xi_j-ta\otimes b_\ell&\hbox{si }x\in {\rm int}(A^-_{k,n})\\
\xi_j+(1-t)a\otimes b_\ell&\hbox{si }x\in {\rm int}(A^+_{k,n}\cup B^+_{k,n}\cup C^+_{k,n})\\
\xi_j-t(a+a^\perp)\otimes b_\ell&\hbox{si }x\in {\rm int}(B^-_{k,n})\\
\xi_j-t(a-a^\perp)\otimes b_\ell&\hbox{si }x\in {\rm int}(C^-_{k,n})\\
\xi_j&\hbox{si }x\in {\rm int}(B_{k,n})\cup{\rm int}(C_{k,n})
\end{array}
\right.
$$
avec $a^\perp=(0,1)$ (et ${\rm int}(E)$ d\'esigne l'int\'erieur de l'ensemble $E$). De plus, on a :
\begin{itemize}
\item[]${\displaystyle\int_{\cup_{k=0}^{n-1} A^-_{k,n}}\R_iW_0(\xi_j-ta\otimes b_\ell)dx}=(1-t)(1-{2\over n})\R_iW_0(\xi_j-ta\otimes b_\ell)$\ ;
\item[]${\displaystyle\int_{\cup_{k=0}^{n-1} A^+_{k,n}}\R_i W_0(\xi_j+(1-t)a\otimes b_\ell)dx}=t(1-{2\over n})\R_i W_0(\xi_j+(1-t)a\otimes b_\ell)$\ ;
\item[]${\displaystyle\int_{\cup_{k=0}^{n-1} (B^+_{k,n}\cup C^+_{k,n})}\R_i W_0(\xi_j+(1-t)a\otimes b_\ell)dx}={t\over n}\R_i W_0(\xi_j+(1-t)a\otimes b_\ell)$\ ;
\item[]${\displaystyle\int_{\cup_{k=0}^{n-1} B^-_{k,n}}\R_i W_0(\xi_j-t(a+a^\perp)\otimes b_\ell)dx}={1-t\over 2n}\R_i W_0(\xi_j-t(a+a^\perp)\otimes b_\ell)$\ ;
\item[]${\displaystyle\int_{\cup_{k=0}^{n-1} C^-_{k,n}}\R_i W_0(\xi_j-t(a-a^\perp)\otimes b_\ell)dx}={1-t\over 2n}\R_i W_0(\xi_j-t(a-a^\perp)\otimes b_\ell)$\ ;
\item[]${\displaystyle\int_{\cup_{k=0}^{n-1} (B_{k,n}\cup C_{k,n})}\R_i W_0(\xi_j)dx}={1\over n}\R_i W_0(\xi_j)$.
\end{itemize}
Donc 
\begin{eqnarray*}
\int_Y\R_i W_0(\xi_j+\nabla\theta_{n,\ell}(x))dx&=&\left(1-{2\over n}\right)\Big[(1-t)\R_iW_0(\xi_j-ta\otimes b_\ell)+t\R_i W_0(\xi_j\\
&&+\ (1-t)a\otimes b_\ell)\Big]+{1\over n}\Big[t\R_i W_0(\xi_j+(1-t)a\otimes b_\ell)\\
&&+\ {1-t\over 2}\big(\R_i W_0(\xi_j-t(a+a^\perp)\otimes b_\ell)+\R_i W_0(\xi_j-\\
&&t(a-a^\perp)\otimes b_\ell)\big)+\R_i W_0(\xi_j)\Big]
\end{eqnarray*}
pour tout $n,\ell\geq 1$. Il suit que pour chaque $\ell\geq 1$ on a 
\begin{eqnarray*}
\lim_{n\to+\infty}\int_Y\R_i W_0(\xi_j+\nabla\theta_{n,\ell}(x))dx&=&(1-t)\R_iW_0(\xi_j-ta\otimes b_\ell)\\
&&+\ t\R_i W_0(\xi_j+(1-t)a\otimes b_\ell).
\end{eqnarray*}
Rappelant que  $\R_i W_0$ est continue (voir (\ref{EquAtiOn1BBLemma1Continue})) et notant que $b_\ell\to b$ on obtient 
\begin{eqnarray*}
\lim_{\ell\to+\infty}\lim_{n\to+\infty}\int_Y\R_i W_0(\xi_j+\nabla\theta_{n,\ell}(x))dx&=&(1-t)\R_iW_0(\xi_j-ta\otimes b)\\
&&+\ t\R_i W_0(\xi_j+(1-t)a\otimes b),
\end{eqnarray*}
et (\ref{EquBBLeMMa2}) suit de (\ref{EquAtiOn1BBLemma1}).
 \end{proof}
Consid\'erons $V^j_{q}\subset V_j$ donn\'e par $V^j_{q}:=\{x\in V_j:{\rm dist}(x,\partial V_j)>{1\over q}\}$ avec $q\geq 1$ assez grand. Alors, il existe une famille finie  $\{r_{m}+\rho_{m}Y\}_{m\in M}$ de sous-ensembles disjoints de $V^j_{q}$, o\`u $r_{m}\in\RR^2$ et $\rho_{m}\in]0,1[$, telle que
$
|V^j_{q}\setminus\cup_{m\in M}(r_{m}+\rho_{m}Y)|\leq{1\over q}.
$ 

D\'efinissons $\{\phi_{n,\ell,q}\}_{n,\ell,q\geq 1}\subset\AffET_0(V_j;\RR^3)$ par 
$$
\phi_{n,\ell,q}(x):=\left\{
\begin{array}{ll}
\displaystyle\rho_{m}\theta_{n,\ell}\left({x-r_{m}\over \rho_{m}}\right)&\hbox{si }x\in r_{m}+\rho_{m}Y\subset V^j_{q}\\
0&\hbox{si }x\in V_j\setminus V^j_{q}
\end{array}
\right.
$$
et posons 
\begin{eqnarray}\label{BBFunct2}
\Phi^j_{n,\ell,q}(x):=\psi(x)+\phi_{n,\ell,q}(x).
\end{eqnarray}
Alors $\{\Phi^j_{n,\ell,q}\}_{n,\ell,q\geq 1}\subset\AffET(V_j;\RR^3)$ et  :
\begin{eqnarray}
&&\hbox{pour chaque }n,\ell,q\geq 1,\  \Phi^j_{n,\ell,q}\hbox{ est localement injective ;}\label{EQUaT(1)BBLemmA1}\\
&&\hbox{pour chaque } \ell,q\geq 1,\ \Phi^j_{n,\ell,q}\wto \psi\hbox{ dans }W^{1,p}(V_j;\RR^3)\ ;\label{EQUaT(2)BBLemmA1}\\
&&\displaystyle\lim_{q\to+\infty}\lim_{\ell\to+\infty}\lim_{n\to+\infty}\int_{V_j}\R_iW_0(\nabla \Phi^j_{n,\ell,q}(x))dx=|V_j|\R_{i+1}W_0(\xi_j).\label{EQUaT(3)BBLemmA1}
\end{eqnarray}

\begin{proof}[Preuve de (\ref{EQUaT(1)BBLemmA1})]
Soient $x\in V_j$ et $W\subset V_j$ une composante connexe de $V_j$ contenant $x$ (Comme $V_j$ est ouvert, $W$ l'est aussi). Puisque $\nabla \psi=\xi_j$ dans $W$, il existe $c\in\RR^3$ tel que $\psi(x^\prime)=\xi_j\cdot x^\prime+c$ pour tout $x^\prime\in W$.   On affirme que ${\Phi^j_{n,\ell,q}}{\lfloor_{W}}$ est injective. En effet, soit $x^\prime\in W$ tel que $\Phi^j_{n,\ell,q}(x)=\Phi^j_{n,\ell,q}(x^\prime)$. Alors, l'une des trois possibilit\'es suivantes a lieu :
\begin{eqnarray}
&&\left\{\begin{array}{l}\Phi^j_{n,\ell,q}(x)=\xi_j\cdot x+c+\rho_m\sigma_{n}\left({x-r_m\over\rho_m}\right)b_\ell\\
\Phi^j_{n,\ell,q}(x^\prime)=\xi_j\cdot x^\prime+c+\rho_{m^\prime}\sigma_{n}\left({x^\prime-r_{m^\prime}\over\rho_{m^\prime}}\right)b_\ell\ ;\end{array}\right.\label{EQUaT(4)BBLemmA1}\\
&&\left\{\begin{array}{l}\Phi^j_{n,\ell,q}(x)=\xi_j\cdot x+c+\rho_m\sigma_{n}\left({x-r_m\over\rho_m}\right)b_\ell\\
\Phi^j_{n,\ell,q}(x^\prime)=\xi_j\cdot x^\prime+c\ ;
\end{array}\right.\label{EQUaT(5)BBLemmA1}\\
&&\left\{\begin{array}{l}\Phi^j_{n,\ell,q}(x)=\xi_j\cdot x+c\\
\Phi^j_{n,\ell,q}(x^\prime)=\xi_j\cdot x^\prime+c.
\end{array}\right.\label{EQUaT(6)BBLemmA1}
\end{eqnarray}
Posons $\alpha:=\rho_m\sigma_n({x-r_m\over\rho_m})-\rho_{m^\prime}\sigma_n({x^\prime-r_{m^\prime}\over \rho_{m^\prime}})$ et $\beta:=\rho_m\sigma_n({x-r_m\over\rho_m})$. On a alors : 
\begin{itemize}
\item[$\diamond$]$\left\{\begin{array}{ll}
\xi_j(x^\prime-x)=0&\hbox{si }\alpha=0\\
b_\ell={1\over \alpha}\xi_j(x^\prime-x)&\hbox{si }\alpha\not=0
\end{array}
\right.
\hbox{ lorsque (\ref{EQUaT(4)BBLemmA1}) a lieu ;}$
\item[$\diamond$]$\left\{
\begin{array}{ll}
\xi_j(x^\prime-x)=0&\hbox{si }\beta=0\\
b_\ell={1\over \beta}\xi_j(x^\prime-x)&\hbox{si }\beta\not=0
\end{array}
\right.
\hbox{ lorsque (\ref{EQUaT(5)BBLemmA1}) a lieu ;}$
\item[$\diamond$]\ $\xi_j(x^\prime-x)=0\hbox{ lorsque (\ref{EQUaT(6)BBLemmA1}) a lieu.}$
\end{itemize}
Il suit que si $x\not=x^\prime$ alors ${\rm rank}(\xi_j)<2$ ou $b_\ell\in{\rm Im}\xi_j$ ce qui est  est impossible. Donc $x=x^\prime$. 
\end{proof}

\begin{proof}[Preuve de (\ref{EQUaT(2)BBLemmA1})]
\'Etant donn\'es $\ell,q\geq 1$, on a $\|\phi_{n,\ell,q}\|_{L^\infty(V_j;\RR^3)}\leq \|\theta_{n,\ell}\|_{L^\infty(Y;\RR^3)}=|b_l|\|\sigma_n\|_{L^\infty(Y;\RR)}$.  D'autre part, pour chaque $k\in\{0,\cdots,n-1\}$, il est clair que $|\sigma_n(x)|\leq {t(1-t)\over n}$ pour tout $x\in]{k\over n},{k+1\over n}[\times]0,1[$, donc $\sigma_n\to 0$ dans $L^\infty(Y;\RR)$, et le r\'esultat suit. donc $\phi_{n,\ell,q}\to 0$ dans $L^\infty(V_j;\RR^3)$, et le r\'esultat suit.
\end{proof}

\begin{proof}[Preuve de (\ref{EQUaT(3)BBLemmA1})]
Rappelant que $\phi_{n,\ell,q}=0$ dans $V_j\setminus\hat{V}^j_q$ et $\sum_{m\in M}\rho_m^2=|\hat{V}^j_q|$ on voit que 
\begin{eqnarray*}
\int_{V_j}\R_iW_0(\nabla \Phi^j_{n,\ell,q}(x))dx\hskip-2mm&=&\hskip-2mm\int_{V_j}\R_iW_0(\xi_j+\nabla \phi_{n,\ell,q}(x))dx\\
&=&\hskip-2mm\int_{\hat{V}^j_q}\R_iW_0(\xi_j+\nabla \phi_{n,\ell,q}(x))dx+|V_j\setminus \hat{V}^j_q|\R_iW_0(\xi_j)\\
&=&\hskip-2mm|\hat{V}^j_q|\int_Y\R_iW_0(\xi_j+\nabla\theta_{n,\ell}(x))dx+|V_j\setminus \hat{V}^j_q|\R_iW_0(\xi_j).
\end{eqnarray*}
Utilisant (\ref{EquBBLeMMa2}) on d\'eduit que 
$$
\lim_{\ell\to+\infty}\lim_{n\to+\infty}\int_{V_j}\R_iW_0(\nabla \Phi^j_{n,\ell,q}(x))dx=|\hat{V}^j_q|\R_{i+1}W_0(\xi_j)+|V_j\setminus \hat{V}^j_q|\R_iW_0(\xi_j)
$$
pour tout $q\geq 1$, et (\ref{EQUaT(3)BBLemmA1}) suit puisque $|\hat{V}^j_q|=|V^j_q|-|V^j_q\setminus \hat{V}^j_q|\to|V_j|$ (car $|V^j_q|\to|V_j|$ et ${1\over q}\geq|V^j_q\setminus \hat{V}^j_q|\to 0$) et $|V_j\setminus \hat{V}^j_q|=|V_j\setminus V^j_q|+|V^j_q\setminus\hat{V}^j_q|\to 0$ (car $|V_j\setminus V^j_q|\to 0$).
\end{proof}

\subsubsection{D\'emonstration du lemme \ref{BenBelgacemLemma}} Prenant en compte (\ref{KohnStrangBBLemma1}) on voit qu'il est suffisant de prouver que 
$$
\overline{\mathcal{I}}_{\AffETli}(\psi)\leq \int_\Sigma\R_iW_0(\nabla \psi(x))dx\hbox{ pour tout }\psi\in\AffETli(\Sigma;\RR^3)\leqno (P_i)
$$
pour tout $i\geq 0$. Pour cela, on raisonne par r\'ecurrence sur $i$. Comme $R_0 W_0=W_0$ on a $(P_0)$.  Supposons $(P_i)$ et montrons $(P_{i+1})$. Soit $\psi\in\AffETli(\Sigma;\RR^3)$. Alors, il existe une famille finie $\{V_j\}_{j\in J}$ de sous-ensemble ouverts et disjoints de $\Sigma$ telle que $|\Sigma\setminus\cup_{j\in J}V_j|=0$ et pour chaque $j\in J$, $|\partial V_j|=0$ et $\nabla \psi(x)=\xi_j$ dans $V_j$ avec $\xi_j\in\MM^{3\times 2}$. D\'efinissons $\{\Psi_{n,\ell,q}\}_{n,\ell,q\geq 1}\subset\AffET(\Sigma;\RR^3)$ par 
$$
\Psi_{n,\ell,q}(x):=\Phi^j_{n,\ell,q}(x)\hbox{ si }x\in V_j
$$
avec $\Phi^j_{n,\ell,q}$ donn\'ee par (\ref{BBFunct2}). Utilisant (\ref{EQUaT(1)BBLemmA1}) (et rappelant que $\psi$ est localement injective) il est facile de voir que $\Psi_{n,\ell,q}$ est localement injective.  Par $(P_i)$ on a 
$$
\overline{\mathcal{I}}_{\AffETli}(\Psi_{n,\ell,q})\leq\int_{\Sigma}\R_iW_0(\nabla\Psi_{n,\ell,q}(x))dx
$$
pour tout $n,\ell,q\geq 1$. Utilisant (\ref{EQUaT(2)BBLemmA1}) on voit que  pour chaque $\ell,q\geq 1$, $\Psi_{n,l,q}\wto \psi$ dans $W^{1,p}(\Sigma;\RR^3)$. Il suit que 
$$
\overline{\mathcal{I}}_{\AffETli}(\psi)\leq \lim_{n\to+\infty}\overline{\mathcal{I}}_{\AffETli}(\Psi_{n,\ell,q})\leq\lim_{n\to+\infty}\int_{\Sigma}\R_iW_0(\nabla\Psi_{n,\ell,q}(x))dx
$$
pour tout $\ell,q\geq 1$. De plus, par (\ref{EQUaT(3)BBLemmA1}) on a 
$$
\lim_{q\to+\infty}\lim_{\ell\to+\infty}\lim_{n\to+\infty}\int_{\Sigma}\R_iW_0(\nabla\Psi_{n,\ell,q}(x))dx=\int_\Sigma\R_{i+1}W_0(\nabla \psi(x))dx,
$$
d'o\`u  
$$
\overline{\mathcal{I}}_{\AffETli}(\psi)\leq \int_\Sigma\R_{i+1}W_0(\nabla \psi(x))dx,
$$
et ($P_{i+1}$) suit. \hfill$\square$

\subsubsection{Preuve de (\ref{EquAtiOn1BBLemma1Continue}) et (\ref{EquAtiOn1BBLemma1})} On commence par prouver le lemme suivant.

\begin{lemma}\label{ProduitTensorielEstFermŽDanslesMatrices}
$\RR^2\otimes\RR^3$ est un ferm\'e de $\MM^{3\times 2}$. 
\end{lemma}
\begin{proof}
Soient $\{a_n\otimes b_n\}_{n\geq 1}\subset\RR^2\otimes\RR^3$ et $\xi\in\MM^{3\times 2}$ tels que $a_n\otimes b_n\to \xi$. Pour chaque $n\geq 1$, $a_n\otimes b_n=u_n\otimes v_n$ avec $u_n={a_n\over|a_n|}\in\SS^1$ et $v_n=|a_n| b_n$, o\`u $\SS^1$ est la sph\`ere unit\'e dans $\RR^2$. Comme $\SS^1$ est compact, il existe $u\in\SS^1$ tel que $u_n\to u$ (\`a une sous-suite pr\`es). Soit $u_0\in\RR^2$ tel que $\langle u,u_0\rangle\not=0$, alors $\langle u_n, u_0\rangle\not=0$ pour $n\geq n_0$ avec $n_0\geq 1$ assez grand. Pour chaque $n\geq n_0$, $v_n={1\over \langle u_n,u_0\rangle}(u_n\otimes v_n)u_0$, donc $v_n\to{1\over \langle u,u_0\rangle}\xi u_0=:v\in\RR^3$. Il suit que $a_n\otimes b_n\to u\otimes v$, d'o\`u $\xi=u\otimes v$.
\end{proof}

Soit $\mathcal{H}:\MM^{3\times 2}\to[0,+\infty]$ d\'efinie par 
$$
\mathcal{H}(\xi):=\inf\Big\{H(\xi,t,a\otimes b):(t,a\otimes b)\in[0,1]\times\RR^2\otimes\RR^3\Big\},
$$
o\`u $H:\MM^{3\times 2}\times[0,1]\times\RR^2\otimes\RR^3\to[0,+\infty]$ est donn\'ee par 
$$
H(\xi,t,a\otimes b):=(1-t)h(\xi-ta\otimes b)+th(\xi+(1-t)a\otimes b)
$$
avec $h:\MM^{3\times 2}\to[0,+\infty]$ continue et coercive (la fonction $H$ est donc continue). (Noter que si $h=\R_q W_0$ avec $q\geq 0$ alors $\mathcal{H}=\R_{q+1}W_0$.) Pour prouver (\ref{EquAtiOn1BBLemma1Continue}) et (\ref{EquAtiOn1BBLemma1}) on aura besoin des deux lemmes suivants (voir les lemmes \ref{KhonStrangFormulaProofContinuitŽetInfAtteint} et \ref{KhonStrangFormulaProofContinuitŽetInfAtteint2}).
\begin{lemma}\label{KhonStrangFormulaProofContinuitŽetInfAtteint}
Soit $\xi\in\MM^{3\times 2}$ tel que $\mathcal{H}(\xi)<+\infty$. Alors, il existe $(t,a\otimes b)\in[0,1]\times\RR^2\otimes\RR^3$ tel que $\mathcal{H}(\xi)=H(\xi,t,a\otimes b)$.
\end{lemma}
\begin{proof}
Soit $\{(t_n,a_n\otimes b_n)\}_{n\geq 1}\subset[0,1]\times\RR^2\otimes\RR^3$ une suite minimisante pour $\mathcal{H}(\xi)$ telle que $t_n\to t\in[0,1]$. Posons $F_n:=\xi-t_na_n\otimes b_n$ et $G_n:=\xi+(1-t_n)a_n\otimes b_n$ (alors $(1-t_n)F_n+t_nG_n=\xi$ et $G_n-F_n=a_n\otimes b_n$ pour tout $n\geq 1$). Par la coercivit\'e de $h$ on a 
\begin{eqnarray}\label{KhonStrangFormulaProofContinuitŽetInfAtteintCond1}
(1-t_n)|F_n|^p+t_n|G_n|^p\leq c\hbox{ pour tout }n\geq 1\hbox{ avec } c>0.
\end{eqnarray}
L'une des deux possibilit\'es suivantes a lieu :
\begin{itemize}
\item[$\diamond$] $t\in]0,1[$ ;
\item[$\diamond$] $t=0$ ou $t=1$.
\end{itemize}
\subsubsection*{Cas $t\in]0,1[$} On a $1-t_n\geq\alpha_1>0$ et $t_n\geq\alpha_2>0$ pour tout $n\geq 1$. Utilisant (\ref{KhonStrangFormulaProofContinuitŽetInfAtteintCond1}) on d\'eduit qu'il existe  $F,G\in\MM^{3\times 2}$ tels que $F_n\to F$ et $G_n\to G$ (\`a une sous-suite pr\`es). Par cons\'equent,  $G_n-F_n=a_n\otimes b_n\to G-F$. Or $\RR^2\otimes\RR^3$ est un ferm\'e de $\MM^{3\times 2}$ d'apr\`es le lemme  \ref{ProduitTensorielEstFermŽDanslesMatrices}, donc $G-F=a\otimes b$ avec $a\in\RR^2$ et $b\in\RR^3$. Comme $H(\xi,\cdot,\cdot)$ est continue, il suit que $\mathcal{H}(\xi)=\lim_{n\to+\infty}H(\xi,t_n,a_n\otimes b_n)=H(\xi,t,a\otimes b)$.
\subsubsection*{Cas $t=0$ ou $t=1$} Supposons que $t=0$ (le cas $t=1$ se traite de la m\^eme fa\c con). On a alors $1-t_n\geq\alpha>0$ pour tout $n\geq 1$. Utilisant (\ref{KhonStrangFormulaProofContinuitŽetInfAtteintCond1}) on d\'eduit que $F_n\to F$ avec $F\in\MM^{3\times 2}$ et $t_n G_n\to 0$ (car $p>1$ et  $t_n\to 0$). Comme $(1-t_n)F_n+t_nG_n=\xi$ pour tout $n\geq 1$, il suit que $F=\xi$, d'o\`u $\lim_{n\to+\infty}(1-t_n)h(F_n)=h(\xi)$ puisque $h$ est continue. Or $t_n h(G_n)=H(\xi,t_n,a_n\otimes b_n)-(1-t_n)h(F_n)$ pour tout $n\geq 1$, donc $\lim_{n\to+\infty}t_n h(G_n)=\mathcal{H}(\xi)-h(\xi)\leq 0$ (car $\mathcal{H}(\xi)\leq h(\xi)$). D'autre part, utilisant la coercivit\'e de $h$, on voit que $t_nh(G_n)\geq Ct_n|G_n|^p$ pour tout $n\geq 1$ avec $C>0$, donc $\lim_{n\to+\infty}t_nh(G_n)\geq C\lim_{n\to+\infty}t_n|G_n|^p=0$. Ainsi, $\lim_{n\to+\infty}t_nh(G_n)=0$ et par cons\'equent $\mathcal{H}(\xi)=h(\xi)=H(\xi,0,a\otimes b)$ o\`u $a\otimes b$ est un \'el\'ement quelconque de $\RR^2\otimes\RR^3$. 
\end{proof}
\begin{lemma}\label{KhonStrangFormulaProofContinuitŽetInfAtteint2}
La fonction $\mathcal{H}$ est continue et coercive. 
\end{lemma}
\begin{proof}
Puisque $H(\cdot,t,a\otimes b)$ est continue pour tout $(t,a\otimes b)\in[0,1]\times \RR^2\otimes\RR^3$, $\mathcal{H}$ est semi-continue sup\'erieurement. Donc, pour montrer que $\mathcal{H}$ est continue, il suffit de prouver que $\mathcal{H}$ est sci. Pour cela, consid\'erons $\xi\in\MM^{3\times 2}$ et $\{\xi_n\}_{n\geq 1}\subset\MM^{3\times 2}$ tels que $\xi_n\to\xi$,  $\sup_{n\geq 1}\mathcal{H}(\xi_n)<+\infty$ et $\lim_{n\to+\infty}\mathcal{H}(\xi_n)=\liminf_{n\to+\infty}\mathcal{H}(\xi_n)$ et montrons que $\mathcal{H}(\xi)\leq\lim_{n\to+\infty}\mathcal{H}(\xi_n)$. Par le lemme \ref{KhonStrangFormulaProofContinuitŽetInfAtteint}, pour chaque $n\geq 1$, il existe $(t_n,a_n\otimes b_n)\in[0,1]\times\RR^2\otimes\RR^3$ tel que $\mathcal{H}(\xi_n)=H(\xi_n,t_n,a_n\otimes b_n)$. Sans perdre de g\'en\'eralit\'e on peut supposer que $t_n\to t\in[0,1]$. De la coercivit\'e de $h$ on d\'eduit que (\ref{KhonStrangFormulaProofContinuitŽetInfAtteintCond1}) a lieu avec $F_n:=\xi_n-t_na_n\otimes b_n$ et $G_n:=\xi_n+(1-t_n)a_n\otimes b_n$. Comme dans la preuve du lemme \ref{KhonStrangFormulaProofContinuitŽetInfAtteint}, on distingue deux cas.
\subsubsection*{Cas $t\in]0,1[$} En utilisant les m\^emes arguments que dans la preuve du lemme \ref{KhonStrangFormulaProofContinuitŽetInfAtteint}, on obtient $G_n-F_n=a_n\otimes b_n\to a\otimes b$ avec $a\in\RR^2$ et $b\in\RR^3$, d'o\`u, comme $H$ est continue, $\lim_{n\to+\infty}\mathcal{H}(\xi_n)=\lim_{n\to+\infty}H(\xi_n,t_n,a_n\otimes b_n)=H(\xi,t,a\otimes b)\geq \mathcal{H}(\xi)$.

\subsubsection*{Cas $t=0$ ou $t=1$} Supposons que $t=1$ (le cas $t=0$ se traite de la m\^eme fa\c con). On a alors $t_n\geq\beta>0$ pour tout $n\geq 1$. Par (\ref{KhonStrangFormulaProofContinuitŽetInfAtteintCond1}) on a $G_n\to G$ avec $G\in\MM^{3\times 2}$ et $(1-t_n)F_n\to 0$ (car $p>1$ et $t_n\to 1$). Comme $(1-t_n)F_n+t_nG_n=\xi_n$ pour tout $n\geq 1$, il suit que $G=\lim_{n\to+\infty}(1-t_n)F_n+t_nG_n=\lim_{n\to+\infty}\xi_n=\xi$, d'o\`u $\lim_{n\to+\infty}t_nh(G_n)=h(\xi)$ puisque $h$ est continue. Or $(1-t_n)h(F_n)=H(\xi_n,t_n,a_n\otimes b_n)-t_nh(G_n)$ pour tout $n\geq 1$, donc $\lim_{n\to+\infty}(1-t_n)h(F_n)=\lim_{n\to+\infty}\mathcal{H}(\xi_n)-h(\xi)\leq 0$ (car $\lim_{n\to+\infty}\mathcal{H}(\xi_n)\leq h(\xi)$ puisque $\mathcal{H}(\xi_n)\leq h(\xi_n)$ pour tout $n\geq 1$). D'autre part, utilisant la coercivit\'e de $h$, on voit que $(1-t_n)h(F_n)\geq C(1-t_n)|F_n|^p$ pour tout $n\geq 1$ avec $C>0$, donc $\lim_{n\to+\infty}(1-t_n)h(F_n)\geq C\lim_{n\to+\infty}(1-t_n)|F_n|^p=0$. Ainsi, $\lim_{n\to+\infty}(1-t_n)h(F_n)=0$ et par cons\'equent $\lim_{n\to+\infty}\mathcal{H}(\xi_n)=h(\xi)\geq \mathcal{H}(\xi)$. 

\medskip

Montrons que $\mathcal{H}$ est coercive. Par la coercivit\'e de $h$ on a  $\mathcal{H}(\xi)\geq C\inf\{(1-t)|\xi-ta\otimes b|^p+t|\xi+(1-t)a\otimes b|^p:(t,a\otimes b)\in[0,1]\times\RR^2\otimes\RR^3\}$ pour tout $\xi\in\MM^{3\times 2}$ avec $C>0$. Or $(1-t)|\xi-ta\otimes b|^p+t|\xi+(1-t)a\otimes b|^p\geq|(1-t)(\xi-ta\otimes b)+t(\xi+(1-t)a\otimes b)|^p=|\xi|^p$, d'o\`u $\mathcal{H}(\xi)\geq C|\xi|^p$ pour tout $\xi\in\MM^{3\times 2}$.
\end{proof}
Par le lemme \ref{Lemme1Det>0}, $W_0=\mathcal{R}_0W_0$ est continue et coercive, et si $\R_q W_0$ est continue et coercive alors, par le lemme \ref{KhonStrangFormulaProofContinuitŽetInfAtteint2} avec $h=\R_{q}W_0$, $\R_{q+1}W_0$ l'est aussi. D'o\`u $\R_q W_0$ est continue et coercive pour tout $q\geq 0$ et, en particulier, (\ref{EquAtiOn1BBLemma1Continue}) a lieu.

Comme  $\rank(\xi_j)=2$, par (\ref{Adj-Infty1}) on a $W_0(\xi_j)<+\infty$, d'o\`u $\R_{i+1}W_0(\xi_j)<+\infty$ puisque $\R_{i+1}W_0\leq W_0$, et (\ref{EquAtiOn1BBLemma1}) suit en utilisant le lemme \ref{KhonStrangFormulaProofContinuitŽetInfAtteint} avec $h=\R_iW_0$.
\hfill{$\square$}

\subsubsection{D\'emonstration du lemme \ref{BenBelgacemLemma2}} Puisqu'une fonction rang-1 convexe finie est continue, il suffit de prouver le premier point du lemme \ref{BenBelgacemLemma2}. Pour cela, on va d\'emontrer le th\'eor\`eme plus g\'en\'eral suivant d\^u \`a Ben Belgacem (voir \cite{benbelgacem96}). (Dans ce qui suit $N\leq m$ et, \'etant donn\'e $F\in\MM^{m\times N}$, on note $0\leq v_1(F)\leq\cdots\leq v_N(F)$ les valeurs singuli\`eres de $F$, i.e., les valeurs propres de $\sqrt{F^{\rm T}F}\in\MM^{N\times N}$, et on pose $v(F):=\prod_{i=1}^Nv_i(F)$.)

\begin{theorem}\label{RWaCCbyBB}
Si $W:\MM^{m\times N}\to[0,+\infty]$ satisfait la condition suivante 
\begin{eqnarray}\label{RWaCCbyBBCondition}
&&\hbox{il existe }\alpha,\beta>0\hbox{ tels que pour tout }F\in\MM^{m\times N}\hbox{,}\\
&&\hbox{si }v(F)\geq \alpha\hbox{ alors }W(F)\leq\beta(1+|F|^p)\hbox{,}\nonumber
\end{eqnarray}
alors $\R W$ satisfait (\ref{RWGrowth-Cond}).
\end{theorem}
\begin{proof}
Sans perdre de g\'en\'eralit\'e on peut supposer que $\alpha\geq 1$. Il est clair que $\R W(F)\leq\beta(1+|F|^p)$ pour tout $F\in\MM^{m\times N}$ tel que $v(F)\geq \alpha$. Consid\'erons donc $F\in\MM^{m\times N}$ tel que $v(F)<\alpha$.  Soient $P\in\OO(m)$ tel que $F=PJU$, o\`u $U:=\sqrt{F^{\rm T}F}$ et $J=(J_{ij})\in\MM^{m\times N}$ avec $J_{ij}=0$ si $i\not=j$ et $J_{ii}=1$, et $Q\in\SO(N)$ tel que $U=Q^{\rm T}\diag(v_1(F),\cdots,v_N(F)) Q$ (voir par exemple \cite[\S 7.3 p. 411]{horn-johnson90} pour plus de d\'etails sur une telle d\'ecomposition de $F$). Alors :
\begin{itemize}
\item[$\diamond$] $F=PJQ^{\rm T}\diag(v_1(F),\cdots,v_N(F)) Q$ ;
\item[$\diamond$] $|F|^2=\sum_{i=1}^Nv_i^2(F)$.
\end{itemize}
Comme $v(F)<\alpha$, il existe $1\leq i_1\leq\cdots\leq i_k\leq N$ avec $k\in\{1,\cdots,N\}$ tels que $v_{i_1}(F)<\alpha,\cdots,v_{i_k}(F)<\alpha$ (et $v_i(F)\geq\alpha$ pour tout $i\not\in\{i_1,\cdots,i_k\}$). Pour chaque $j\in\{1,\cdots,k\}$, consid\'erons $t_j\in]0,1[$ tel que 
$
v_{i_j}(F)=(1-t_j)(-\alpha)+t_j\alpha.
$
Alors 
\begin{eqnarray*}
\diag(v_1(F),\cdots,v_{i_1}(F),\cdots,v_N(F))&=&(1-t_1)\diag(v_1(F),\cdots,-\alpha,\cdots,v_N(F))\\ &&+\ t_1\diag(v_1(F),\cdots,\alpha,\cdots,v_N(F)),
\end{eqnarray*}
donc $F=(1-t_1)F_{1}^-+t_1F^+_{1}$ avec :
\begin{itemize}
\item[$\diamond$] $F^-_1:=PJQ^{\rm T}\diag(v_1(F),\cdots,-\alpha,\cdots,v_N(F)) Q$ ;
\item[$\diamond$] $F^+_1:=PJQ^{\rm T}\diag(v_1(F),\cdots,\alpha,\cdots,v_N(F)) Q$ ;
\item[$\diamond$] $\rank(F^-_1-F^+_1)=1$.
\end{itemize}
On a aussi 
\begin{eqnarray*}
&&\diag(v_1(F),\cdots,-\alpha,\cdots,v_{i_2}(F),\cdots,v_N(F))=
(1-t_2)\diag(v_1(F),\cdots,-\alpha,\\ &&\cdots,-\alpha,\cdots,v_N(F))
+t_2\diag(v_1(F),\cdots,-\alpha,\cdots,\alpha,\cdots,v_N(F)),
\end{eqnarray*}
donc $F^-_1=(1-t_2)F^{-,-}_2+t_2F^{-,+}_2$ avec :
\begin{itemize}
\item[$\diamond$] $F^{-,-}_2:=PJQ^{\rm T}\diag(v_1(F),\cdots,-\alpha,\cdots,-\alpha,\cdots,v_N(F)) Q$ ;
\item[$\diamond$] $F^{-,+}_2:=PJQ^{\rm T}\diag(v_1(F),\cdots,-\alpha,\cdots,\alpha,\cdots,v_N(F)) Q$ ;
\item[$\diamond$] $\rank(F^{-,-}_2-F^{-,+}_2)=1$.
\end{itemize}
De la m\^eme fa\c con on a $F^+_1=(1-t_2)F^{+,-}_2+t_2F^{+,+}_2$ avec :
\begin{itemize}
\item[$\diamond$] $F^{+,-}_2:=PJQ^{\rm T}\diag(v_1(F),\cdots,\alpha,\cdots,-\alpha,\cdots,v_N(F)) Q$ ;
\item[$\diamond$] $F^{+,+}_2:=PJQ^{\rm T}\diag(v_1(F),\cdots,\alpha,\cdots,\alpha,\cdots,v_N(F)) Q$ ;
\item[$\diamond$] $\rank(F^{+,-}_2-F^{+,+}_2)=1$.
\end{itemize}
Continuant ainsi on obtient une suite finie $\{F_j^{\sigma}\}_{j\in\{1,\cdots,k\}}^{\sigma\in\mathfrak{S}_j}\subset\MM^{m\times N}$, o\`u $\mathfrak{S}_j$ d\'esigne la classe de toutes les applications  $\sigma:\{1,\cdots,j\}\to\{-,+\}$, telle que :
\begin{itemize}
\item[$\diamond$] $F_j^\sigma=PJQ^{\rm T}\diag(v_1(F),\cdots,\sigma(1)\alpha,\cdots,\sigma(j)\alpha,\cdots,v_N(F))$ pour tout $j\in\{1,\cdots,k\}$ et tout $\sigma\in\mathfrak{S}_j$ ;
\item[$\diamond$] pour chaque $j\in\{1,\cdots,k\}$ et chaque $\sigma,\sigma^\prime\in\mathfrak{S}_j$, si $\sigma(j)\not=\sigma^\prime(j)$ et $\sigma(\ell)=\sigma^\prime(\ell)$ pour tout $\ell\in\{1,\cdots,j-1\}$ alors $\rank(F_j^\sigma-F_j^{\sigma^\prime})=1$ ;
\item[$\diamond$] $F=(1-t_1)F_1^\sigma+t_1F^{\sigma^\prime}_1$ pour tout $\sigma,\sigma^\prime\in\mathfrak{S}_1$ tels que $\sigma(1)\not=\sigma^\prime(1)$ ;
\item[$\diamond$] pour chaque $j\in\{1,\cdots,k\}$ et chaque $\sigma\in\mathfrak{S}_j$, $F_j^\sigma=(1-t_{j+1})F_{j+1}^{\sigma^\prime}+t_{j+1}F_{j+1}^{\sigma^{\prime\prime}}$ pour tout $\sigma^\prime,\sigma^{\prime\prime}\in\mathfrak{S}_{j+1}$ tels que $\sigma^\prime(l)=\sigma^{\prime\prime}(l)=\sigma(l)$ pour tout $l\in\{1,\cdots,j\}$ et $\sigma^\prime(j+1)\not=\sigma^{\prime\prime}(j+1)$.
\end{itemize}
Il suit que :
\begin{itemize}
\item[$\diamond$] $\R W(F)\leq\R W(F^\sigma_1)+\R W(F^{\sigma^\prime}_1)$ pour tout $\sigma,\sigma^\prime\in\mathfrak{S}_1$ tels que $\sigma(1)\not=\sigma^\prime(1)\ ;$
\item[$\diamond$] pour chaque $j\in\{1,\cdots,k\}$ et chaque $\sigma\in\mathfrak{S}_j$, $\R W(F^\sigma_j)\leq\R W(F_{j+1}^{\sigma^\prime})+\R W(F_{j+1}^{\sigma^{\prime\prime}})$ pour tout $\sigma^\prime,\sigma^{\prime\prime}\in\mathfrak{S}_{j+1}$ tels que $\sigma^\prime(l)=\sigma^{\prime\prime}(l)=\sigma(l)$ pour tout $l\in\{1,\cdots,j\}$ et $\sigma^\prime(j+1)\not=\sigma^{\prime\prime}(j+1)$.
\end{itemize}
D'o\`u 
$$
\R W(F)\leq\sum_{\sigma\in\mathfrak{S}_k}\R W(F^\sigma_k).
$$
De plus, on a 
$
v(F^\sigma_k)=|\det(\diag(v_1(F),\cdots,\sigma(1)\alpha,\cdots,\sigma(k)\alpha,\cdots,v_N(F)))|=\alpha^k\prod_{i\not\in\{i_1,\cdots i_k\}}v_i(F)
$
pour tout $\sigma\in\mathfrak{S}_k$. Donc, pour chaque $\sigma\in\mathfrak{S}_k$, $v(F^\sigma_k)\geq \alpha^N\geq\alpha$. Utilisant (\ref{RWaCCbyBBCondition}) on d\'eduit que 
$$
\R W(F)\leq\sum_{\sigma\in\mathfrak{S}_k}\beta(1+|F^\sigma_k|^p).
$$
Or, pour tout $\sigma\in\mathfrak{S}_k$,
$|F^\sigma_k|^2=|\diag(v_1(F),\cdots,\sigma(1)\alpha,\cdots,\sigma(k)\alpha,\cdots,v_N(F))|^2=k\alpha^2+\sum_{i\not\in\{i_1,\cdots,i_k\}}v_i^2(F)\leq N\alpha^2+|F|^2$,  d'o\`u 
$$
\R W(F)\leq \sum_{\sigma\in\mathfrak{S}_k}\beta\big(1+2^{p\over 2}(N^{p\over 2}\alpha^p+|F|^p)\big)\leq c(1+|F|^p)
$$
avec $c=2^N\beta(1+2^{p\over 2}N^{p\over 2}\alpha^p)$, ce qui donne le r\'esultat.
\end{proof}
Utilisant le th\'eor\`eme \ref{RWaCCbyBB} avec $W=W_0$, $m=3$ et $N=2$ et remarquant que $v(\xi)=|\xi_1\land \xi_2|$ pour tout $\xi=(\xi_1\mid \xi_2)\in\MM^{3\times 2}$, on obtient le premier point du lemme \ref{BenBelgacemLemma2}. \hfill$\square$

\begin{remark}
On peut g\'en\'eraliser le corollaire \ref{CorollaryExample:m=NBis} comme suit.
\begin{corollary}
Si $W:\MM^{m\times N}\to[0,+\infty]$ satisfait la condition suivante 
\begin{eqnarray}\label{SingUlarValuECondition}
&&\hbox{pour tout }\delta>0\hbox{ il existe }c_\delta>0\hbox{ tel que pour tout }F\in\MM^{m\times N}\hbox{,}\\
&&\hbox{si }v(F)\geq \delta\hbox{ alors }W(F)\leq c_\delta(1+|F|^p)\hbox{,}\nonumber
\end{eqnarray}
alors {(\ref{FunctRelax})} a lieu avec $\overline{W}=\mathcal{Q}W=\Z W$. 
\end{corollary}
\begin{proof}
La seule difficult\'e est de prouver que si $W$ satisfait (\ref{SingUlarValuECondition}) alors $\Z W$ est finie. (Pour cela, on pourra s'inspirer de l'\'etape 1 de la d\'emonstration du th\'eor\`eme \ref{ThExample:m=N}.) On fait ensuite le m\^eme raisonnement que dans la preuve du corollaire \ref{CorollaryExample:m=NBis} en remarquant que (\ref{SingUlarValuECondition}) implique (\ref{RWaCCbyBBCondition}). 
\end{proof}
\end{remark}

\section{Deux th\'eor\`emes d'approximation}

\subsection{Th\'eor\`eme de Gromov-Eliashberg}

 En 1971, Gromov et Eliashberg  ont d\'emontr\'e le th\'eor\`eme suivant (voir \cite[Theorem 1.3.4B]{gromov-eliashberg71}, voir aussi \cite[Theorem B$^\prime_1$ p. 20]{gromov86}).
\begin{theorem}\label{GEapproxTheo}
Soient $m>N\geq 1$ et $M$ une vari\'et\'e compacte de dimension $N$ qui peut \^etre immerg\'ee dans $\RR^m$. Alors, pour chaque fonction $C^1$-diff\'erentiable $\psi$ de $M$ dans $\RR^m$ il existe une suite  $\{\psi_n\}_n$ de $C^1$-immersions de $M$ dans $\RR^m$ telle que $\psi_n\to \psi$ dans $W^{1,p}(M;\RR^m)$.
\end{theorem}

Dans notre contexte ($m=3, N=2$ et $M=\overline{\Sigma}$) on a 

\newtheorem*{GEapproxTheo}{\bf Th\'eor\`eme \ref{GEapproxTheo}-bis}

\begin{GEapproxTheo}
Pour chaque $\psi\in C^1(\overline{\Sigma};\RR^3)$ il existe $\{\psi_n\}_{n\geq 1}\subset C^1_*(\overline{\Sigma};\RR^3)$ tel que $\psi_n\to \psi$ dans $W^{1,p}(\Sigma;\RR^3)$.
\end{GEapproxTheo}

(Pour une preuve du th\'eor\`eme \ref{GEapproxTheo}-bis voir \S 4.1.1-4.1.3.) Le r\'esultat suivant a \'et\'e \'enonc\'e par Ben Belgacem (voir \cite[p. 137-138]{benbelgacem96}, voir aussi \cite[Proposition 3.1.7 p. 100]{trabelsi04}). Pour la d\'efinition de  $\Aff^{\rm reg}_{\rm c}(\RR^2;\RR^3)$, $\Aff^{\rm reg}(\Sigma;\RR^3)$ et $\Aff_{\rm li}^{\rm reg}(\Sigma;\RR^3)$ voir \S 4.1.4.
\begin{proposition}\label{ElementsFinisProp}
Pour chaque $\psi\in C^1_*(\overline{\Sigma};\RR^3)$ il existe $\{\psi_n\}_{n\geq 1}\subset\Aff_{\rm li}^{\rm reg}(\Sigma;\RR^3)$ tel que  $\psi_n\to \psi$ dans $W^{1,\infty}(\Sigma;\RR^3)$.
\end{proposition}
\begin{proof} 
Soit $\psi\in C^1_*(\overline{\Sigma};\RR^3)$ avec $\psi=\hat\psi\lfloor_{\overline{\Sigma}}$ o\`u $\hat\psi\in C^1_*(\RR^2;\RR^3)$. Par le th\'eor\`eme \ref{Ekeland-TemamDensityTheorem} il existe $\{\psi_n\}_{n\ge 1}\subset \Aff^{\rm reg}(\Sigma;\RR^3)$ tel que $\psi_n\to \psi\mbox{ dans }W^{1,\infty}(\Sigma;\RR^3)$. Par d\'efinition, pour chaque $n\geq 1$, $\psi_n=\hat\psi_n\lfloor_{\Sigma}$ avec $\hat\psi_n\in\Aff^{\rm reg}_{\rm c}(\RR^2;\RR^3)$. 

Fixons $x_0\in\overline{\Sigma}$ et montrons que $\hat\psi_n$ est localement injective en $x_0$ \`a partir d'un certain rang. D'abord, puisque $\hat\psi\in C^\infty_*(\RR^2;\RR^3)$ il existe $c_0>0$ tel que 
\begin{equation}\label{inegalite derive}
\big\vert \nabla\hat\psi(x_0)v\big\vert \ge c_0\vert v\vert\hbox{ pour tout }v\in\RR^2.
\end{equation}
Consid\'erons deux entiers $k>\frac{2}{c_0}$ et $n\geq n_k$ avec $n_k\in\NN$ tel que 
\begin{equation}\label{inegalite approximation}
\big\| \nabla\hat\psi_n-\nabla\hat\psi\big\|_{L^\infty(\Sigma;\MM^{3\times 2})}\le \frac{1}{2k}.
\end{equation}
 Comme $\nabla\hat\psi$ est continue en $x_0$ il existe $\delta>0$ tel que 
\begin{equation}\label{continuite de la derivee}
\big\vert \nabla\hat\psi(z)-\nabla\hat\psi(x_0)\big\vert\le \frac{1}{2k}\hbox{ pour tout }z\in B_\delta(x_0).
\end{equation}
\'Etant donn\'es $x,y\in B_\delta(x_0)$, du th\'eor\`eme de la moyenne  on d\'eduit que 
\[
\big\vert \hat\psi(x)-\hat\psi(y)-\nabla\hat\psi(x_0)(x-y)\big\vert\le \vert x-y\vert\sup_{z\in [x,y]}\big\vert \nabla\hat\psi(z)-\nabla\hat\psi(x_0)\big\vert
\]
(voir \cite[(8.6.2) p. 164]{dieudonne81}), donc 
\begin{equation}\label{accroiss inegalite}
\big\vert \hat\psi(x)-\hat\psi(y)-\nabla\hat\psi(x_0)(x-y)\big\vert\le \frac{1}{2k}\vert x-y\vert
\end{equation}
par \eqref{continuite de la derivee}. De \eqref{inegalite approximation} on d\'eduit que $\hat\psi_n-\hat\psi$ est Lipschitzienne de rapport ${1\over 2k}$ sur $\Sigma$ (voir \cite[Corollaire 2.4 p. 288]{ekeland-temam74}), donc  
\[
\big\vert \hat\psi_n(x)-\hat\psi_n(y)\big\vert
\ge\big\vert \hat\psi(x)-\hat\psi(y)\big\vert-\frac{1}{2k}\vert x-y\vert.
\]
Mais par \eqref{inegalite derive} et \eqref{accroiss inegalite} on a 
\begin{eqnarray*}
\big\vert \hat\psi(x)-\hat\psi(y)\big\vert
&\ge& \big\vert \nabla\hat\psi(x_0)(x-y)\big\vert-\big\vert \hat\psi(x)-\hat\psi(y)-\nabla\hat\psi(x_0)(x-y)\big\vert\\
&\ge& c_0\vert x-y\vert-\frac{1}{2k}\vert x-y\vert,
\end{eqnarray*}
d'o\`u $\vert \hat\psi_n(x)-\hat\psi_n(y)\vert\ge (c_0-\frac{1}{k})\vert x-y\vert\ge \frac{c_0}{2}\vert x-y\vert$,
ce qui montre que $\hat\psi_n$ est localement injective en $x_0$ d\`es que  $n\geq n_k$.
\end{proof}

Le r\'esultat suivant est une cons\'equence du th\'eor\`eme \ref{GEapproxTheo}-bis.
\begin{theorem}\label{GEdensityTheo}
$\Aff_{\rm li}^{\rm reg}(\Sigma;\RR^3)$ est fortement dense dans $W^{1,p}(\Sigma;\RR^3)$.
\end{theorem}
\begin{proof}
Il suffit d'utiliser le fait que $C^1(\overline{\Sigma};\RR^3)$ est fortement dense dans $W^{1,p}(\Sigma;\RR^3)$ avec le th\'eor\`eme \ref{GEapproxTheo}-bis et la proposition \ref{ElementsFinisProp}.
\end{proof}

\subsubsection{Pr\'eliminaires pour la d\'emonstration du th\'eor\`eme \ref{GEapproxTheo}-bis} (Pour la preuve du th\'eor\`eme \ref{GEapproxTheo}-bis voir \S 4.1.2.) La d\'efinition suivante est d\^u \`a Whitney \cite{whitney55} (voir aussi \cite{gromov-eliashberg71}).
\begin{definition}\label{DefTypicAlMaPs}
On dit que $h\in C^\infty(\RR^2;\RR^2)$ est typique (ou stable ou encore g\'en\'erique) si les conditions suivantes sont satisfaites :
\begin{itemize}
\item[$\diamond$] $\{x\in\RR^2: \hbox{dim Ker}\nabla h(x)=2\}=\emptyset$ ;
\item[$\diamond$] l'ensemble $S:=\{x\in\RR^2: \hbox{dim Ker}\nabla h(x)=1\}$ est une sous-vari\'et\'e ferm\'ee de dimension $1$ telle que $S=S^{0}\cup S^{1}$, o\`u $S^{0}$ est un sous-ensemble discret et $S^{1}$ est une sous-vari\'et\'e de dimension $1$ donn\'es par 
$$
S^0:=\big\{x\in\RR^2: \hbox{Ker}\nabla h(x)=T_xS\big\}\mbox{ et }S^1:=\big\{x\in\RR^2: \hbox{Ker}\nabla h(x)+T_xS=\RR^2\big\},
$$
o\`u $T_xS$ est l'espace tangent \`a $S$ en $x$.
\end{itemize}
On note $C^\infty_{\rm typ}(\RR^2;\RR^2)$ l'ensemble des applications typiques.
\end{definition}
\begin{remark}
En fait, $S^{0}$ correspond aux singularit\'es de type {\em fronce} et $S^{1}$ aux singularit\'es de type {\em pli} (voir \cite{whitney55}).
\end{remark}
Le lemme suivant est un cas particulier du th\'eor\`eme de densit\'e de Whitney (voir \cite[Th\'eor\`eme IV.8.4 p. 128]{demazure89}, voir aussi\cite[Chapter IV \textsection 2 p. 145]{golubitsky-guillemin73})
\begin{lemma}\label{ThomBordman}
L'ensemble $C^\infty_{\rm typ}(\RR^2;\RR^2)$ est dense dans $C^\infty(\RR^2;\RR^2)$ pour la topologie de  la convergence uniforme de toutes les d\'eriv\'ees.
\end{lemma}
On pose $C^\infty_{\rm typ}(\overline{\Sigma};\RR^2)$ l'ensemble des restrictions \`a $\overline{\Sigma}$ des applications de $C^\infty_{\rm typ}(\RR^2;\RR^2)$, i.e.,
\[
C^\infty_{\rm typ}(\overline{\Sigma};\RR^2):=\Big\{h\lfloor_{\overline{\Sigma}}: h\in C^\infty_{\rm typ}(\RR^2;\RR^2)\Big\}.
\]
Le r\'esultat suivant (qui sera utilis\'e dans la d\'emonstration du th\'eor\`eme \ref{GEapproxTheo}-bis) est une cons\'equence imm\'ediate du lemme \ref{ThomBordman}.
\begin{corollary}\label{ThomBordmanBis}
 L'ensemble $C^\infty_{\rm typ}(\overline{\Sigma};\RR^2)$ est dense dans $C^\infty(\overline{\Sigma};\RR^2)$. \end{corollary}
Dans ce qui suit, on aura besoin de la d\'efinition suivante (pour plus de d\'etails sur le concept de transversalit\'e voir \cite{golubitsky-guillemin73}).
\begin{definition} Soient $S\subset\overline{\Sigma}$ une sous-vari\'et\'e et $L\in C(S;\RR^2)$. On dit que $L$ est transversal \`a $S$ et on note $L \trans S$ si pour chaque $x\in S$,
\[
L(x)\notin T_xS \hbox{, i.e., } L(x)\RR+T_xS=\RR^2.
\]
\end{definition}
Pour $g\in C^\infty(\overline{\Sigma})$ et $h=(h_1,h_2)\in C^\infty(\overline{\Sigma};\RR^2)$, on note $g\oplus h\in C^\infty(\overline{\Sigma};\RR^3)$ un assemblage quelconque de $g$ et $h$, i.e., $g\oplus h=(g,h)$ ou $g\oplus h=(h_1,g,h_2)$ ou encore $g\oplus h=(h,g)$. Les deux lemmes suivants ainsi que la proposition \ref{approxchampsvecteurs} ci-dessous seront utilis\'es dans la d\'emonstration du th\'eor\`eme \ref{GEapproxTheo}-bis.
\begin{lemma}\label{selectionchampsvecteurs} Soient $K\subset \overline{\Sigma}$ un compact, $g\in C^\infty(\overline{\Sigma})$ et $h\in C^\infty(\overline{\Sigma};\RR^2)$ tels que 
\begin{equation}\label{selectionchampsvecteursHyp}
K\subset\Big\{x\in\RR^2:\hbox{dimKer}\nabla h(x)\geq1\Big\} \hbox{ et }g\oplus h\in C^\infty_*(\overline{\Sigma};\RR^3).
\end{equation}
Alors, il existe $L\in C(K;\RR^2)$ tel que 
\[
\nabla h(x)L(x)=0\hbox{ et }\nabla g(x)L(x)=1\hbox{ pour tout }x\in K.
\]
\end{lemma}
\begin{proof}Pour chaque $x\in K$, on consid\`ere $v_x\in{\rm Ker}\nabla h(x)\cap \SS^1$ (avec $\SS^1:=\{v\in\RR^2:|v|=1\}$) et on d\'efinit $L:K\to\RR^2$ par 
\[
L(x):={v_x\over \nabla g(x)v_x}.
\]
(Noter que $\nabla g(x)v_x\not=0$ pour tout $x\in K$ gr\^ace \`a \eqref{selectionchampsvecteursHyp}.) Il est facile de voir que $\nabla h(x)L(x)=0$ et $\nabla g(x)L(x)=1$ pour tout $x\in K$. D'autre part, il est clair que 
\begin{itemize}
\item[$\diamond$] $\Big\{{v\over\nabla g(x) v}:(x,v)\in K\times \SS^1\Big\}$ est compact
\end{itemize}
et, par \eqref{selectionchampsvecteursHyp}, on a 
\begin{itemize}
\item[$\diamond$] ${v\over\nabla g(x) v}={v_x\over\nabla g(x) v_x}$ pour tout  $v\in {\rm Ker}\nabla h(x)\cap \SS^1$ et tout $x\in K$.
\end{itemize}
Donc, \'etant donn\'es $x\in K$ et $\{x_n\}_{n\geq 1}\subset K$ tels que $x_n\to x$, il est ais\'e de voir que $L(x)$ est l'unique valeur d'adh\'erence de $\{L(x_n)\}_{n\geq 1}$, d'o\`u $L$ est continue.
\end{proof}
\begin{lemma}\label{conservimm}Soient $g\in C^\infty(\overline{\Sigma})$, $\{h^n\}_{n\ge 1}\subset C^\infty(\overline{\Sigma};\RR^2)$ et $h\in C^\infty(\overline{\Sigma};\RR^2)$ tels que 
\[
\nabla h^n\to\nabla h\hbox{ uniform\'ement dans }\overline{\Sigma}\hbox{ et }g\oplus h\in C^\infty_\ast(\overline{\Sigma};\RR^3).
\]
Alors, il existe $n_0\ge 1$ tel que $g\oplus h^n\in C^\infty_\ast(\overline{\Sigma};\RR^3)$ pour tout $n\ge n_0$.
\end{lemma}
\begin{proof} En effet, supposons qu'il existe $\{x_n\}_{n\ge 1}\subset \overline{\Sigma}$ tel que $\mbox{Ker}\nabla (g\oplus h^n)(x_n)\not=\{0\}$ pour tout $n\ge 1$. Pour chaque $n\ge 1$, on consid\`ere $L_n\in \mbox{Ker}\nabla (g\oplus h^n)(x_n)\cap\SS^1$. Alors, pour chaque $n\geq 1$,
\begin{eqnarray*}
\vert \nabla (g\oplus h)(x_{n})L_n\vert&\le& \vert \nabla (g\oplus h^{n})(x_{n})-\nabla (g\oplus h)(x_{n})\vert\vert L_n\vert+\vert \nabla (g\oplus h^{n})(x_n) L_n\vert\\
&\le&\Vert \nabla (g\oplus h^{n})-\nabla (g\oplus h)\Vert_\infty\le \Vert \nabla h^{n}-\nabla h\Vert_\infty.
\end{eqnarray*}
Comme $\SS^1$ est compact et $\nabla (g\oplus h)$ est continue on en d\'eduit qu'il existe $x\in\overline{\Sigma}$ et $L\in\SS^1$ tels que $\nabla (g\oplus h)(x)L=0$ ce qui impossible puisque $g\oplus h\in C^\infty_\ast(\overline{\Sigma};\RR^3)$. 
\end{proof}

\begin{proposition}\label{approxchampsvecteurs} 
Soient $f\in C^\infty(\overline{\Sigma})$, $S\subset\overline{\Sigma}$ une sous-vari\'et\'e compacte de dimension $1$ avec $S=S^{0}\cup S^{1}$, o\`u $S^{0}$ est un sous-ensemble discret et $S^{1}$ est une sous-vari\'et\'e de dimension $1$, et $L\in C(S;\RR^2)$ tels que $L\trans S^{0}$ et $L\trans S^{1}$. Alors, il existe 
$\{f^n\}_{n\ge 1}\subset C^\infty(\overline{\Sigma})$ tel que 
\[
\begin{cases}
\lim\limits_{n\to+\infty}\|f^{n}-f\|_{W^{1,p}(\Sigma)}=0\\
\nabla f^{n}(x)L(x)>0 \hbox{ pour tout }x\in S\hbox{ et tout } n\ge 1.
\end{cases}
\]
\end{proposition}
(Pour une preuve de la proposition \ref{approxchampsvecteurs} voir \S 4.1.3.)

\subsubsection{D\'emonstration du th\'eor\`eme \ref{GEapproxTheo}-bis} 
On suit la preuve de \cite{gromov-eliashberg71} (avec les simplifications qui s'imposent dans notre cas) qui consiste \`a \'eliminer les singularit\'es en exploitant la sym\'etrie de la relation (diff\'erentielle) d'immersion sous certaines transformations de l'espace d'arriv\'ee. Pour un expos\'e de la m\'ethode d'\'elimination des singularit\'es voir \cite[Part 2]{gromov86}.

Comme $C^\infty(\overline{\Sigma};\RR^3)$ est dense dans $C^1(\overline{\Sigma};\RR^3)$ pour la topologie de la convergence uniforme de toutes les  d\'eriv\'ees, on peut supposer que $\psi=(\psi_1,\psi_2,\psi_3)\in C^\infty(\overline{\Sigma};\RR^3)$. 

\subsubsection*{\'Etape 1} Soit $h\in C^\infty(\overline{\Sigma};\RR^2)$ d\'efinie par $h(x):=(h_2(x),h_3(x))=(x_2,\psi_3(x))$. Par le corollaire  \ref{ThomBordmanBis} il existe $\{h^k\}_{k\ge 1}\subset C^\infty_{\rm typ}(\overline{\Sigma};\RR^2)$ avec $h^k=(h^k_2,h^k_3)$ tel que 
\begin{equation}\label{convergenceGE}
h^k\to h\mbox{ et }\nabla h^k\to\nabla h\hbox{ uniform\'ement dans }\overline{\Sigma}.
\end{equation}

\subsubsection*{\'Etape 2} On pose  $\psi^k:=(x_1,h^k)=x_1\oplus h^k$. On a $x_1\oplus h\in C^\infty_*(\overline{\Sigma};\RR^3)$, donc, en tenant compte de \eqref{convergenceGE}, le lemme \ref{conservimm} implique qu'il existe $k_0\ge 1$ tel que $\psi^k\in C^\infty_*(\overline{\Sigma};\RR^3)$ pour tout $k\ge k_0$. 

Soit $k\ge k_0$. On note $S_k$ la sous-vari\'et\'e compacte de dimension $1$ associ\'e \`a $h^k$ (voir la d\'efinition \ref{DefTypicAlMaPs}). Par le lemme \ref{selectionchampsvecteurs} il existe $L^k\in C(S_k;\RR^2)$ tel que 
\begin{equation*}\label{positive champs de vecteurs}
\hbox{Ker}\nabla h^k(x)=L^k(x)\RR\hbox{ et }L^k(x)\not=0\hbox{ pour tout }x\in S_k.
\end{equation*}
Donc $L^k\trans S^0_k$ et $L^k\trans S^1_k$.

\subsubsection*{\'Etape 3} Par la proposition \ref{approxchampsvecteurs} il existe $\{h_1^{k,l}\}_{l\ge 1}\subset C^\infty(\overline{\Sigma})$ tel que :
\begin{eqnarray}
&& \lim_{l\to +\infty}\Vert h_1^{k,l}-\psi_1\Vert_{W^{1,p}(\Sigma)}=0\hbox{ pour tout }k\ge k_0\ ;\label{convergenceW}\\
&& \nabla h_1^{k,l}(x)L^k(x)>0 \hbox{ pour tout }x\in S_k\hbox{ et tout } l\ge 1.\label{convergenceW-bis}
\end{eqnarray}
On pose $\psi^{k,l}:=(h_1^{k,l},h^k)$. Par d\'efinition de $S_k$, on a  $\mbox{Ker}\nabla h^{k}(x)=\{0\}$ pour tout $x\in\overline{\Sigma}\setminus S_k$. D'autre part, \'etant donn\'es $x\in S_k$ et $v\in\mbox{Ker}\nabla\psi^{k,l}(x)=\mbox{Ker}\nabla h_1^{k,l}(x)\cap\mbox{Ker}\nabla h^k(x)$, si $v\not=0$ alors $v=\lambda_x^k L^k(x)$ pour un certain $\lambda_x^k\in\RR^*$, ce qui contredit \eqref{convergenceW-bis}. Ainsi, on a 
\begin{equation}\label{psikl_imm}
\psi^{k,l}\in C^\infty_*(\overline{\Sigma};\RR^3) \hbox{ pour tout $l\ge 1$ et tout $k\ge k_0$.}
\end{equation}

Maintenant on r\'ep\`ete les trois \'etapes pr\'ec\'edentes en prenant $h^{k,l}$ \`a la place de $x_2$ et $h_3^k$ \`a la place de $\psi_3$. Fixons $k\ge k_0$ et $l\ge 1$. 

\subsubsection*{\'Etape 3.1} Soit $h^{k,l}\in C^\infty(\overline{\Sigma};\RR^2)$ d\'efinie par $h^{k,l}(x):=(h_1^{k,l}(x),h_3^{k}(x))$. Par le corollaire  \ref{ThomBordmanBis} il existe $\{h^{k,l,m}\}_{m\ge 1}\subset C^\infty_{\rm typ}(\overline{\Sigma};\RR^2)$ avec $h^{k,l,m}=(h^{k,l,m}_1,h^{k,m}_3)$ tel que 
\begin{equation}\label{convergenceGE2}
h^{k,l,m}\to h^{k,l}\mbox{ et } \nabla h^{k,l,m}\to \nabla h^{k,l}\hbox{ uniform\'ement dans }\overline{\Sigma}.
\end{equation}

\subsubsection*{\'Etape 3.2} On pose $\psi^{k,l,m}:=(h^{k,l,m}_1,h^k_2,h^{k,m}_3)=h_2^{k}\oplus h^{k,l,m}$. Prenant en compte \eqref{convergenceGE2} et \eqref{psikl_imm}, du lemme \ref{conservimm} on d\'eduit qu'il existe $m_{k,l}\ge 1$ tel que $\psi^{k,l,m}\in C^\infty_*(\overline{\Sigma};\RR^3)$ pour tout $m\ge m_{k,l}$. 

Soit $m\ge m_{k,l}$. Comme dans l'\'etape 2, on note $S_{k,l,m}$ la sous-vari\'et\'e compacte de dimension $1$ associ\'ee \`a $h^{k,l,m}$ et, de la m\^eme fa\c con, par le lemme \ref{selectionchampsvecteurs} il existe $L^{k,l,m}\in C(S_{k,l,m};\RR^2)$ tel que $L^{k,l,m}\trans S^0_{k,l,m}$ et $L^{k,l,m}\trans S^1_{k,l,m}$.

\subsubsection*{\'Etape 3.3} Par la proposition \ref{approxchampsvecteurs} il existe $\{h_2^{k,l,m,n}\}_{n\ge 1}\subset C^\infty(\overline{\Sigma})$ tel que :
\begin{eqnarray}
&& \lim_{n\to +\infty}\Vert h_2^{k,l,m,n}-\psi_2\Vert_{W^{1,p}(\Sigma)}=0\ ;\label{convergenceW2}\\
&& \nabla h_2^{k,l,m,n}(x)L^{k,l,m}(x)>0 \hbox{ pour tout }x\in S_{k,l,m}\hbox{ et tout } n\ge 1.\label{convergenceW2-bis}
\end{eqnarray}
On pose $\psi^{k,l,m,n}:=(h^{k,l,m}_1, h_2^{k,l,m,n}, h^{k,m}_3)$. Comme dans l'\'etape 3, gr\^ace \`a \eqref{convergenceW2-bis} on peut montrer que 
\[
\psi^{k,l,m,n}\in C^\infty_*(\overline{\Sigma};\RR^3)\hbox{ pour tous $k\ge k_0,l\ge 1,m\ge m_{k,l}\hbox{ et }n\ge 1$.}
\]

\subsubsection*{\'Etape 4}  Pour chaque $k\ge k_0,l\ge 1,m\ge m_{k,l}\hbox{ et }n\ge 1$, on a 
\begin{eqnarray*}
\Vert \psi^{k,l,m,n}-\psi \Vert_{W^{1,p}(\Sigma;\RR^3)}&\le& \Vert \psi^{k,l,m,n}-(h_1^{k,l,m},\psi_2,h_3^{k,m}) \Vert_{W^{1,p}(\Sigma;\RR^3)}\\
&&+\Vert (h_1^{k,l,m},\psi_2,h_3^{k,m})-(h_1^{k,l},\psi_2,h_3^k) \Vert_{W^{1,p}(\Sigma;\RR^3)}\\
&&+\Vert (h_1^{k,l},\psi_2,h_3^k)-(\psi_1,\psi_2,h_3^k) \Vert_{W^{1,p}(\Sigma;\RR^3)}\\
&&+\Vert (\psi_1,\psi_2,h_3^k)-\psi \Vert_{W^{1,p}(\Sigma;\RR^3)}.
\end{eqnarray*}
En utilisant successivement \eqref{convergenceW2}, \eqref{convergenceGE2}, \eqref{convergenceW} et \eqref{convergenceGE}, i.e., en faisant successivement $n\to +\infty$, $m\to +\infty$, $l\to +\infty$ et $k\to +\infty$, on obtient 
$$
\lim_{k\to +\infty}\lim_{l\to +\infty}\lim_{m\to +\infty}\lim_{n\to +\infty}\Vert \psi^{k,l,m,n}-\psi \Vert_{W^{1,p}(\Sigma;\RR^3)}=0,
$$
et le r\'esultat suit par diagonalisation. \hfill{$\square$}

\subsubsection{D\'emonstration de la proposition \ref{approxchampsvecteurs}} On commence par prouver le lemme suivant.

\begin{lemma}\label{champ de vecteurs} Soient $f\in C^\infty(\overline{\Sigma})$, $S\subset\overline{\Sigma}$ et  $L\in C(S;\RR^2)$ comme dans la proposition \ref{approxchampsvecteurs}.
\begin{itemize}
\item[(a)] Il existe $f^{0}\in C^\infty(\overline{\Sigma})$ et un voisinage ouvert ${\mathcal V}(S^0)\subset S$ de $S^0$ tels que 
\[
\begin{cases}
f^{0}-f=0\hbox{ sur }S^{0}\\
\nabla f^{0}(x)L(x)>0\hbox{ pour tout }x\in {\mathcal V}(S^0).
\end{cases}
\]
\item[(b)] S'il existe un voisinage ouvert ${\mathcal V}(S^0)\subset S$ de $S^0$ tel que  $\nabla f(\cdot)L(\cdot)\lfloor_{{\mathcal V}(S^0)}>0$, alors il existe $f^{1}\in C^\infty(\overline{\Sigma})$ tel que 
\[
\begin{cases}
f^{1}-f=0\hbox{ sur }S\\
\nabla f^{1}(x)L(x)>0\hbox{ pour tout }x\in S.
\end{cases}
\]
\end{itemize}
\end{lemma}
\begin{proof}
Soit $\widehat{f}\in C^\infty(\RR^2)$ tel que $\widehat{f}\lfloor_{\overline{\Sigma}}=f$. 

(a) Comme $S^0\subset S$ est discret, on a 
\[
S^0=\{a^0,a^1,\cdots\}=\cupp_{i\in \NN}\{a^i\},
\]
o\`u les $a^i$ sont des points isol\'es. Soit une famille $\{B_i\}_{i\in \NN}$ de boules de centre $a^i$ deux \`a deux disjointes. Pour chaque $i\in\NN$, on consid\`ere  une boule $B_i^\prime\subset B_i$ de m\^eme centre et on d\'efinit $\{\varphi_i\}_{i\in \NN}\subset C^\infty_c(\RR^2;[0,1])$ par 
\[
\varphi_i(x):=\begin{cases}
1&\hbox{ si }x\in B_i^\prime\\
0&\hbox{ si }x\in \RR^2\setminus B_i.
\end{cases}
\]
Comme $L\trans S^0$ on a $L(a^i)=(L_1(a^i),L_2(a^i))\not=0$ pour tout $i\in\NN$. Soient $I_1:=\{i\in\NN:L_1(a^i)\not=0\}$, $I_2:=\NN\setminus I_1$ et $\widehat{f^0}\in C^\infty(\RR^2)$ donn\'ee par 
\[
\widehat{f^0}(x):=\begin{cases}
\widehat{f}(x)+\sum_{i\in I_1} k_i\varphi_i(x)(x_1-a^i_1)&\hbox{ si }x\in \cup_{i\in I_1}B_i\\
\widehat{f}(x)+\sum_{i\in I_2} k_i\varphi_i(x)(x_2-a^i_2)&\hbox{ si }x\in \cup_{i\in I_2}B_i\\
\widehat{f}(x)&\hbox{ sinon}
\end{cases}
\] 
avec 
\[
k_i:=\begin{cases}
\left\vert\frac{1-\nabla\widehat{f}(a^i)L(a^i)}{L_1(a^i)}\right\vert &\hbox{ si }L_1(a^i)>0\hbox{ et }i\in I_1\\
-\left\vert\frac{1-\nabla\widehat{f}(a^i)L(a^i)}{L_1(a^i)}\right\vert &\hbox{ si }L_1(a^i)<0\hbox{ et }i\in I_1\\
\left\vert\frac{1-\nabla\widehat{f}(a^i)L(a^i)}{L_2(a^i)}\right\vert &\hbox{ si }L_2(a^i)>0\hbox{ et }i\in I_2\\
-\left\vert\frac{1-\nabla\widehat{f}(a^i)L(a^i)}{L_2(a^i)}\right\vert &\hbox{ si }L_2(a^i)<0\hbox{ et }i\in I_2.\\
\end{cases}
\]
Alors, 
$
\widehat{f^0}(a_i)=\widehat{f}(a_i) \hbox{ et }\nabla\widehat{f^0}(a_i)L(a_i)\ge 1$ pour tout $i\in \NN$. Comme $\nabla\widehat{f^0}(\cdot)L(\cdot)$ est continue, il suit que pour chaque $i\in\NN$, il existe $\rho_i>0$ tel que $\nabla\widehat{f^0}(x)L(x)>0$ pour tout $x\in B_{\rho_i}(a^i)\cap S$, o\`u $B_{\rho_i}(a^i)$ est la boule ouverte de centre $a^i$ et de rayon $\rho_i$. On obtient (a) en posant $f^0:=\widehat{f^0}\lfloor_{\overline{\Sigma}}\in C^\infty(\overline{\Sigma})$ et ${\mathcal V}(S^0):=\cup_{i\in\NN}B_{\rho_i}(a_i)\cap S$. 

\smallskip

(b) Pour chaque $x\in S^1$, il existe une boule ouverte $B_x\subset \RR^2$ de centre $x$ et un diff\'eomorphisme $\phi^x=(\phi_1^x,\phi_2^x):B_x\to \RR^2$ tel que $\phi^x(B_x\cap S^1)=\phi^x(B_x)\cap(\RR\times\{0\})$ et on fait en sorte que $B_x\cap S^1$ soit connexe. 

Comme $\phi^x_2\lfloor_{B_x\cap S^1}=0$, $L\trans S^1$ et $\phi^x$ est un diff\'eomorphisme, on a $\nabla\phi^x_2(y)L(y)\not=0$ pour tout $y\in {B_x}\cap S^1$. En effet, on peut remarquer que $\nabla\phi^x_2(y)v=0$ pour tout $v\in T_yS^1$ et tout $y\in {B_x}\cap S^1$. Soit $y\in {B_x}\cap S^1$, si $\nabla\phi^x_2(y)L(y)=0$ alors on aurait $\hbox{Ker}\nabla\phi^x_2(y)=L(y)\RR+T_yS^1$ qui serait de dimension $2$ puisque $L\trans S^1$ ce qui contredirait le fait que $\phi^x$ est un diff\'eomorphisme. Maintenant puisque $B_x\cap S^1$ est connexe et que $B_x\cap S^1\ni y\mapsto \nabla\phi^x_2(y)L(y)$ est continue, alors $\nabla\phi^x_2(\cdot)L(\cdot)\lfloor_{B_x\cap S^1}$ est de signe constant.

Soit une boule ouverte $Q_x$ de centre $x$ tel que $\overline{Q}_x\subset B_x$. On pose 
\begin{equation}\label{def de kU}
k_{x}:=\begin{cases}
\sup_{y\in \overline{Q}_x\cap S^1} \frac{\vert 1-\nabla\widehat{f}(y)L(y)\vert}{\vert \nabla\phi^x_2(y)L(y)\vert}&\hbox{ si }\nabla\phi^x_2(\cdot)L(\cdot)\lfloor_{B_x\cap S^1}>0\\
-\sup_{y\in \overline{Q}_x\cap S^1} \frac{\vert 1-\nabla\widehat{f}(y)L(y)\vert}{\vert \nabla\phi^x_2(y)L(y)\vert}&\hbox{ si }\nabla\phi^x_2(\cdot)L(\cdot)\lfloor_{B_x\cap S^1}<0,
\end{cases}
\end{equation} 
et on d\'efinit $\widehat{f^1_x}:=\widehat{f}+k_{x}\phi_2^x$. Par \eqref{def de kU} on a 
\begin{equation}\label{proprieties local U}
\widehat{f^1_x}(y)=\widehat{f}(y)\mbox{ et }\nabla\widehat{f^1_x}(y) L(y)=\nabla\widehat{f}(y)L(y)+k_{x}\nabla\phi_2^x(y)L(y)\ge 1
\end{equation}
pour tout $y\in Q_x\cap S^1$. La famille $\{Q_x\}_{x\in S^1\setminus {\mathcal V}(S^0)}$ est un recouvrement d'ouverts de $S^1\setminus {\mathcal V}(S^0)=S\setminus {\mathcal V}(S^0) $ qui est compact. On consid\`ere un sous-recouvrement fini $\{Q_{x_i}\}_{i\in\{1,\cdots,s\}}$  et un voisinage ouvert ${\mathcal U}(S^0)\supset {\mathcal V}(S^0)$ dans  $\overline{\Sigma}$, i.e.,
\[
S\subset \cupp_{i=1}^{s}Q_{x_i}\cup {\mathcal U}(S^0).
\]
Soient $\varphi_0,\varphi_1,\dots,\varphi_s,\varphi_{s+1}\in C^\infty(\RR^2;[0,1])$ telles que :
\begin{itemize}
\item[$\diamond$] $\varphi_0(x)=0$ dans un voisinage ${\mathcal V}(S)\subset \overline{\Sigma}$ de $S$  ;
\item[$\diamond$] $\varphi_i(x)=0$ dans $\RR^2\setminus Q_{x_i}$ pour tout $i\in\{1,\dots,s\}$ ;
\item[$\diamond$] $\varphi_{s+1}(x)=0$ dans $\RR^2\setminus {\mathcal U}(S^0)$ ;
\item[$\diamond$] $\sum_{i=0}^{s+1}\varphi_i(x)=1$ dans $\RR^2$.
\end{itemize}
On d\'efinit 
$
\widehat{f^1}:=\varphi_0\widehat{f}+\sum_{i=1}^s\varphi_i\widehat{f_{x_i}^1}+\varphi_{s+1}\widehat{f}\in C^\infty(\RR^2)
$
qui satisfait $\widehat{f^1}=\widehat{f}$ sur $S$. Alors, en utilisant \eqref{proprieties local U}, on obtient 
\begin{eqnarray*}
\nabla\widehat{f^1}(x)L(x) \ge \begin{cases}
\displaystyle\sum_{i=1}^s\varphi_i(x)=1>0&\hbox{ si }x\in S\setminus {\mathcal U}(S^0)\\
\min\big\{1,\nabla\widehat{f}(x)L(x)\big\}>0&\hbox{ si }x\in {\mathcal V}(S^0),
\end{cases}
\end{eqnarray*}
et (b) suit en posant $f^1:=\widehat{f^1}\lfloor_{\overline{\Sigma}}\in C^\infty(\overline{\Sigma})$. 
\end{proof}

Le lemme suivant est une cons\'equence du th\'eor\`eme de synth\`ese spectral \cite[Theorem $5$]{hedberg-wolff83} (voir aussi \cite{hedberg81}). Pour une preuve dans le cas o\`u $K$ est une sous-vari\'et\'e compacte de dimension $1$ voir \cite{gromov-eliashberg71}.

%%%
 \begin{lemma}\label{capacite} Soient $K\subset \overline{\Sigma}$ compact et $\alpha\in  C^\infty(\overline{\Sigma})$ tels que $\alpha{\lfloor_K}=0$. Alors, il existe une suite $\{\omega_n\}_{n\ge 1}\subset C^\infty_c(\overline{\Sigma})$ et une suite de voisinages $\{V_n\}_{n\ge 1}$ de $K$ telles que pour chaque $n\ge 1$,
\[
\begin{cases}
\omega_n=1\mbox{ sur }V_n\\
\Vert\omega_n\alpha\Vert_{W^{1,p}(\Sigma)}<\frac{1}{n}.
\end{cases}
\]
\end{lemma}

Le r\'esultat suivant se d\'eduit du lemme \ref{champ de 
vecteurs} et du lemme \ref{capacite}.

\begin{lemma}\label{consequencecvc} Soient $f\in C^\infty(\overline{\Sigma}),S\subset\overline{\Sigma}$ et $L\in C(\overline{\Sigma};\RR^2)$ comme dans la proposition \ref{approxchampsvecteurs}. 
\begin{itemize}
\item[(a)] Il existe un voisinage ouvert ${\mathcal V}(S^0)\subset S$ de $S^0$ et $\{f^{m,0}\}_{m\ge 1}\subset C^\infty(\overline{\Sigma})$ tels que 
\[
\begin{cases}
\lim\limits_{m\to+\infty}\|f^{m,0}-f\|_{W^{1,p}(\Sigma)}=0\\
\nabla f^{m,0}(x)L(x)>0 \hbox{ pour tout }x\in {\mathcal V}(S^0)\hbox{ et tout } m\ge 1.
\end{cases}
\]
\item[(b)] S'il existe un voisinage ouvert ${\mathcal V}(S^0)\subset S$ de $S^0$ tel que  $\nabla f(\cdot)L(\cdot)\lfloor_{{\mathcal V}(S^0)}>0$, alors il existe  $\{f^{n,1}\}_{n\ge 1}\subset C^\infty(\overline{\Sigma})$ tel que 
\[
\begin{cases}
\lim\limits_{n\to+\infty}\|f^{n,1}-f\|_{W^{1,p}(\Sigma)}=0\\
\nabla f^{n,1}(x)L(x)>0 \hbox{ pour tout }x\in S\hbox{ et tout } n\ge 1.
\end{cases}
\]
\end{itemize}
\end{lemma}
\begin{proof}
En utilisant le lemme \ref{capacite} avec $K=S$ et le lemme \ref{champ de vecteurs} on pose $f^{m,0}:=f+\omega_m(f^0-f)$ pour tout $m\ge 1$ qui satisfait (a). En utilisant les lemmes \ref{capacite} et \ref{champ de vecteurs} on pose $f^{n,1}:=f+\omega_n(f^1-f)$ pour tout $n\ge 1$ qui satisfait (b).
\end{proof}

En combinant (a) et (b) du lemme \ref{consequencecvc} on obtient $\{f^{n,m,0,1}\}_{n,m\ge 1}\subset C^\infty(\overline{\Sigma})$ tel que 
\[
\begin{cases}
\lim\limits_{m\to+\infty}\lim\limits_{n\to+\infty}\|f^{n,m,0,1}-f\|_{W^{1,p}(\Sigma)}=0\\
\nabla f^{n,m,0,1}(x)L(x)>0 \hbox{ pour tout }x\in S\hbox{ et tous } n,m\ge 1,
\end{cases}
\]
et la proposition \ref{approxchampsvecteurs} suit par diagonalisation.\hfill$\square$

\subsubsection{Fonctions continues et affines par morceaux}

Un maillage r\'egulier (dans $\RR^N$) est une famille finie $\{V_i\}_{i\in I}\subset\RR^N$ de simplexes (triangles si $N=2$) ouverts et disjoints telle que pour chaque $i,j\in I$ avec $i\not= j$, l'intersection $\overline{V_i}\cap \overline{V_j}$ est soit vide, soit r\'eduite \`a un sommet commun, soit r\'eduite \`a une face (ar\^ete si $N=2$) commune.  

On dit que $\psi:V\to\RR^m$, avec $V$ un ouvert de $\RR^N$, est affine si c'est la restriction \`a $V$ d'une fonction affine de $\RR^N$ dans $\RR^m$. On note $\AffET_{\rm c}(\RR^N;\RR^m)$ l'ensemble des fonctions continues $\psi:\RR^N\to\RR^m$ pour lesquelles il existe un maillage r\'egulier $\{V_i\}_{i\in I}$ tel que pour chaque $i\in I$, $\psi\lfloor_{{V_i}}$ est affine et $\psi=0$ dans $\RR^N\setminus\cup_{i\in I}\overline{V}_i$. \'Etant donn\'e un ouvert born\'e $\Omega\subset\RR^N$, on d\'efinit :
\begin{itemize}
\item[] $\AffET(\Omega;\RR^m):=\Big\{\psi\lfloor_{{\Omega}}:\psi\in\AffET_{\rm c}(\RR^N;\RR^m)\Big\}$ ;
\item[] $\AffET_0(\Omega;\RR^m):=\Big\{\psi\in\AffET(\Omega;\RR^m):\psi=0\hbox{ sur }\partial\Omega\Big\}$.
\end{itemize}
On dit que $\psi:\RR^N\to\RR^m$ est localement injective en $x\in\RR^N$ s'il existe $\rho>0$ tel que $\psi\lfloor_{{B_\rho(x)}}$ est injective, o\`u $B_\rho(x)$ d\'esigne la boule ouverte de centre $x$ et rayon $\rho$. Si pour tout $x\in E\subset\RR^N$, $\psi$ est localement injective en $x$,  on dit $\psi$ est localement injective sur $E$.  On pose 
$$
\AffETli(\Omega;\RR^m):=\Big\{\psi\lfloor_{{\Omega}}:\AffET_{\rm c}(\RR^N;\RR^m)\ni\psi\hbox{ est localement injective sur }\overline{\Omega}\Big\}.
$$

\begin{theorem}\label{Ekeland-TemamDensityTheorem}
Pour tout $p\in[1,\infty]$, $\AffET(\Omega;\RR^m)$ est fortement dense dans $W^{1,p}(\Omega;\RR^m)$. 
\end{theorem}
\begin{proof}
Soit $\psi\in W^{1,p}(\Omega;\RR^m)$. Comme $C^\infty(\overline{\Omega};\RR^m)$ est fortement dense dans $W^{1,p}(\Omega;\RR^m)$, on peut supposer que $\psi\in C^\infty(\overline{\Omega};\RR^m)$. Par d\'efinition, il existe $\hat\psi\in C^\infty_{\rm c}(\RR^N;\RR^m)$ (l'espace des fonctions $C^\infty$-diff\'erentiables de $\RR^N$ dans $\RR^m$ \`a support compact) tel que $\psi=\hat\psi\lfloor_{\overline{\Omega}}$. Soit $Q\subset\RR^N$ un cube ouvert  tel que $Q\supset\overline{\Omega}$ et $\hat\psi=0$ dans $\RR^N\setminus Q$. Par interpolation, on peut construire une suite $\{\hat\psi_n\}_n\subset\AffET_{\rm c}(\RR^N;\RR^m)$ telle que  $\hat\psi_n=0$ dans $\RR^N\setminus Q$ pour tout $n\geq 1$ et $\hat\psi_n\to\hat\psi$ dans $W^{1,p}(\RR^N;\RR^m)$ (voir \cite[\S 2.1-2.4 p. 285-297]{ekeland-temam74}). Posons $\psi_n:=\hat\psi_n\lfloor{_{\Omega}}$ pour tout $n\geq 1$. Alors, $\{\psi_n\}_{n\geq 1}\subset\AffET(\Omega;\RR^m)$ et $\psi_n\to\psi$ dans $W^{1,p}(\Omega;\RR^m)$. 
\end{proof}

\begin{remark}\label{Ekeland-TemamDensityRemark}
On a $\AffET(\Omega;\RR^m)\subset\Aff(\Omega;\RR^m)$ avec $\Aff(\Omega;\RR^m)$ d\'esignant l'en-semble des fonctions continues et affines par morceaux de $\Omega$ dans $\RR^m$, i.e., $\psi\in\Aff(\Omega;\RR^m)$ si et seulement si $\psi$ est continue et il existe une famille finie $\{V_i\}_{i\in I}$ de sous-ensembles ouverts et disjoints de $\Omega$ telle que $|\Omega\setminus\cup_{i\in I} V_i|=0$ et pour chaque $i\in I$, $|\partial V_i|=0$ et $\nabla\psi(x)=\xi_i$ dans $V_i$  avec $\xi_i\in\MM^{m\times N}$.  Du th\'eor\`eme \ref{Ekeland-TemamDensityTheorem} il suit que  $\Aff(\Omega;\RR^m)$ est aussi fortement dense dans $W^{1,p}(\Omega;\RR^m)$ pour tout $p\in[1,\infty]$.
\end{remark}

\subsection{Th\'eor\`eme de Ben Belgacem-Bennequin}

En 1996, Ben Belgacem et Bennequin  ont d\'emontr\'e le th\'eor\`eme suivant (voir  \cite[Lemma 8 p. 114]{benbelgacem96}).
\begin{theorem}\label{BBBapproxTheo}
Pour chaque $\psi\in\AffETli(\Sigma;\RR^3)$ il existe $\{\psi_n\}_{n\geq 1}\subset C^1(\overline{\Sigma};\RR^3)$ tel que {:} 
\begin{eqnarray}\label{BB_1}
&&\psi_n\to \psi\hbox{ dans }W^{1,p}(\Sigma;\RR^3)\hbox{ ;}\\
&&|\partial_1\psi_n(x)\land\partial_2\psi_n(x)|\geq\delta\hbox{ pour tout }x\in\overline{\Sigma}\hbox{ et tout }n\geq 1\hbox{ avec }\delta>0.\label{BB_2}
\end{eqnarray}
\end{theorem}
Ben Belgacem et Bennequin ont donn\'e une d\'emonstration tr\`es g\'eom\'etrique de leur th\'eor\`eme (voir \cite{benbelgacem96} pour plus de d\'etails). Nous donnons ici une preuve plus analytique.
\begin{proof}
Soit  $\psi\in\AffETli(\Sigma;\RR^3)$. Par d\'efinition, il existe $\varphi\in\AffET_{\rm c}(\RR^N;\RR^m)$ tel que $\psi=\varphi\lfloor_{\Sigma}$ et $\varphi$ est localement injective sur $\overline{\Sigma}$. On a donc un maillage r\'egulier $\{V_i\}_{i\in I}$ dans $\RR^2$ (pour la d\'efinition voir \S 4.1.4) avec $\overline{\Sigma}\cap\overline{V}_i\not=\emptyset$ pour tout $i\in I$ tel que $\overline{\Sigma}\subset{\rm int}(\cup_{i\in I}\overline{V}_i)$, et pour chaque $i\in I$, $\varphi\lfloor_{V_i}$ est affine. Il suit qu'il existe $\delta_1>0$ tel que 
\begin{eqnarray}\label{PoofBBLemma1EqUat0}
&&|\partial_1\varphi(x)\land\partial_2\varphi(x)|\geq\delta_1\hbox{ pour tout }x\in \cup_{i\in I}V_i.
\end{eqnarray}
On d\'efinit deux ensembles finis $A\subset I\times I$ et $S\subset \overline{\Sigma}$ par :
\begin{itemize}
\item[] $A:=\big\{(i,j)\in I\times I:\overline{V}_i\cap \overline{V}_j\hbox{ est r\'eduit \`a une ar\^ete commune}\big\}$ ;
\item[] $S:=\big\{x\in \overline{\Sigma}:\hbox{il existe }i,j\in I\hbox{ avec }i\not=j\hbox{ tel que }\overline{V}_i\cap \overline{V}_j=\{x\}\big\}$.
\end{itemize}
On num\'erote les \'el\' ements de $S$, i.e., $S=\{s_1,\cdots,s_q\}$. On peut montrer le lemme suivant.

\begin{lemma}[lissage des ar\^etes]\label{PoofBBLemma1} 
Il existe $\{U^{i,j}_n\}_{n\geq 1}^{(i,j)\in A}$ et $\{O^k_n\}^{k\in\{1,\cdots,q\}}_{n\geq 1}$ deux familles de parties de $\RR^2$ telles que :  
\begin{eqnarray}
&& \hbox{pour chaque }n\geq 1\hbox{,} \hbox{ les }U^{i,j}_n\hbox{ sont deux \`a deux disjoints ;}\label{EquaTioN(a)PoofBBLemma1}\\
&& \hbox{pour chaque }n\geq 1\hbox{,} \hbox{ les }O^k_n\hbox{ sont deux \`a deux disjoints ;}\label{EquaTioN(b)PoofBBLemma1}\\
&&\hbox{pour chaque }(i,j)\in A\hbox{ et chaque }n\geq 1\hbox{, } U^{i,j}_n\subset{\rm int}(\overline{V}_i\cup \overline{V}_j)\hbox{ est un voisi-}\label{EquaTioN(c)PoofBBLemma1}\\
&&\hbox{nage ouvert de }{\rm int}(\overline{V}_i\cap \overline{V}_j)\hbox{ ;}\nonumber\\
&&\hbox{pour chaque }k\in\{1,\cdots,q\}\hbox{ et chaque }n\geq 1\hbox{, }O^k_n\hbox{ est un voisinage ouvert}\label{EquaTioN(d)PoofBBLemma1}\\ &&\hbox{de }s_k\hbox{ ;}\nonumber\\
&&\hbox{pour chaque }(i,j)\in A\hbox{, }U^{i,j}_1\supset\cdots\supset U^{i,j}_n\supset\cdots\hbox{ et }\lim\limits_{n\to+\infty}|U^{i,j}_n|=0\hbox{ ;}\label{EquaTioN(e)PoofBBLemma1}\\
&&\hbox{pour chaque }k\in\{1,\cdots,q\}\hbox{, }O^k_1\supset\cdots\supset O^k_n\supset\cdots\hbox{ et }\lim\limits_{n\to+\infty}|O^k_n|=0\hbox{,}\label{EquaTioN(f)PoofBBLemma1}
\end{eqnarray}
et $\{\varphi^{i,j}_n\}^{(i,j)\in A}_{n\geq 1}$ une famille de fonctions de $\RR^2$ dans $\RR^3$ telle que :
\begin{eqnarray}
&&\hbox{pour chaque }(i,j)\in A\hbox{ et chaque }n\geq 1\hbox{, }\varphi^{i,j}_n=\varphi\hbox{ dans }(\overline{V}_i\cup\overline{V}_j)\setminus U^{i,j}_n\label{EquaTioN(g)PoofBBLemma1}\\
&&\hbox{et }\varphi^{i,j}_n\in C(\overline{V}_i\cup\overline{V}_j;\RR^3)\cap C^1({\rm int}(\overline{V}_i\cup\overline{V}_j);\RR^3)\hbox{ ;}\nonumber\\
&&\hbox{pour chaque }(i,j)\in A\hbox{, }\sup\limits_{n\geq 1}\sup\limits_{x\in U^{i,j}_n}(|\varphi^{i,j}_n(x)|+|\nabla\varphi^{i,j}_n(x)|)<+\infty\hbox{ ;}\label{EquaTioN(h)PoofBBLemma1}\\
&&\hbox{pour chaque }(i,j)\in A\hbox{, il existe }\alpha_{i,j}>0\hbox{ tel que pour tout }n\geq 1\hbox{ et tout }\label{EquaTioN(i)PoofBBLemma1}\\
&&x\in U^{i,j}_n\hbox{, }|\partial_1\varphi^{i,j}_n(x)\land\partial_2\varphi^{i,j}_n(x)|\geq\alpha_{i,j}\hbox{ ;}\nonumber\\
&&\hbox{pour chaque }k\in\{1,\cdots,q\}\hbox{ et chaque }n\geq 1\hbox{, }\hat\varphi_n=\hat\varphi_1\hbox{ dans }O^k_n\hbox{ avec }\label{EquaTioN(l)PoofBBLemma1}\\
&&\hat\varphi_n:{\rm int}(\cup_{i\in I}\overline{V}_i)\to\RR^3\hbox{ d\'efinie par }\nonumber\\
&&
\hat\varphi_n(x):=\left\{
\begin{array}{ll}
\displaystyle\varphi^{i,j}_n(x)&\displaystyle\hbox{si }x\in U^{i,j}_n\\
\displaystyle\varphi(x)&\displaystyle\hbox{si }x\in {\rm int}(\cup_{i\in I}\overline{V}_i)\setminus \cup_{(i,j)\in A}U^{i,j}_n
\end{array}
\right.
\nonumber\\
&&\hbox{(les fonctions }\hat\varphi_n\hbox{ sont bien d\'efinies gr\^ace \`a (\ref{EquaTioN(a)PoofBBLemma1}) et (\ref{EquaTioN(g)PoofBBLemma1})) ;}\nonumber\\
&&\hbox{pour chaque }k\in\{1,\cdots,q\}\hbox{ et chaque }\lambda>0\hbox{, }\hat\varphi_1(h^k_\lambda(x))=H^k_\lambda(\hat\varphi_1(x))\label{EquaTioN(k)PoofBBLemma1}\\
&&\hbox{pour tout }x\in O^k_1\hbox{ tel que }h^k_\lambda(x)\in O^k_1\hbox{ avec }h^k_\lambda:\RR^2\to\RR^2\hbox{\ (resp. }\nonumber\\
&&H^k_\lambda:\RR^3\to\RR^3\hbox{)\ }\hbox{d\'esignant l'homoth\'etie de centre }s_k\hbox{\ (resp.\ }\hat\varphi_1(s_k)\hbox{)\ et}\nonumber\\
&&\hbox{de rapport }\lambda\hbox{ ;}\nonumber\\
&&\hbox{pour chaque }k\in\{1,\cdots,q\}\hbox{, }\hat\varphi_1\hbox{ est localement injective sur }O^k_1.\label{EquaTioN(j)PoofBBLemma1}
\end{eqnarray}
 \end{lemma}

(Pour une preuve du lemme \ref{PoofBBLemma1} voir \S 4.2.1.) De (\ref{EquaTioN(c)PoofBBLemma1}) et (\ref{EquaTioN(g)PoofBBLemma1}) il suit que pour chaque $n\geq 1$, $\hat\varphi_n\in C({\rm int}(\cup_{i\in I}\overline{V}_i);\RR^3)\cap C^1({\rm int}(\cup_{i\in I}\overline{V}_i)\setminus S;\RR^3)$. De (\ref{PoofBBLemma1EqUat0}) et (\ref{EquaTioN(i)PoofBBLemma1}) on d\'eduit que 
\begin{eqnarray}
&& |\partial_1\hat\varphi_n(x)\land\partial_2\hat\varphi_n(x)|\geq\delta_2\hbox{ pour tout }n\geq 1\hbox{ et tout } x\in {\rm int}(\cup_{i\in I}\overline{V}_i)\setminus S\label{PoofBBLemma1EqUat6}
\end{eqnarray}
avec $\delta_2:=\min\{\delta_1,\min\{\alpha_{i,j}:(i,j)\in A\}\}$. De plus, par (\ref{EquaTioN(e)PoofBBLemma1}) et  (\ref{EquaTioN(h)PoofBBLemma1}) on voit facilement que 
\begin{equation}\label{CvgFirstProofBBBTheo}
\hat\varphi_n\to\varphi\hbox{ dans }W^{1,p}({\rm int}(\cup_{i\in I}\overline{V}_i);\RR^3).
\end{equation}
On peut prouver (en utilisant (\ref{EquaTioN(k)PoofBBLemma1}) et (\ref{EquaTioN(j)PoofBBLemma1})) le lemme suivant.
\begin{lemma}[lissage des sommets]\label{PoofBBLemma2}
 Pour tout $k\in\{1,\cdots,q\}$, il existe une $C^1$-immersion $\hat\varphi_1^k:O^k_1\to\RR^3$  telle que $\hat\varphi_1^k=\hat\varphi_1$ dans $O^k_1\setminus B^k_1$, o\`u $B^k_1$ est un voisinage ouvert de $s_k$ tel que $\overline{B}^k_1\subset O^k_1$.
\end{lemma}
(Pour une preuve du lemme \ref{PoofBBLemma2} voir \S 4.2.2.) On pose $B^k_\lambda:=h^k_\lambda(B^k_1)$ pour tout $k\in\{1,\cdots,q\}$ et tout $\lambda>0$. Prenant en compte (\ref{EquaTioN(d)PoofBBLemma1}) et (\ref{EquaTioN(f)PoofBBLemma1}), il est facile de voir que pour chaque $k\in\{1,\cdots,q\}$, il existe une suite $\{\lambda_n^k\}_{n\geq 1}\subset]0,1]$ telle que $\lambda^k_1=1$, $\lim_{n\to+\infty}\lambda_n^k=0$ et $\overline{B}^k_{\lambda_n^k}\subset O^k_n$ pour tout $n\geq 1$. Pour chaque $k\in\{1,\cdots,q\}$ et chaque $n\geq 1$, on d\'efinit $\hat\varphi_n^k\in C^1(O^k_n;\RR^3)$ par 
$$
\hat\varphi_n^k(x):=H^k_{\lambda^k_n}\big(\hat\varphi^k_1(h_{{1/\lambda^k_n}}^k(x))\big).
$$
Alors $\hat\varphi^k_n=\hat\varphi_n$ dans $O^k_n\setminus B^k_{\lambda^k_n}$ (puisque  si $x\in O^k_n\setminus B^k_{\lambda^k_n}$ alors $h^k_{{1/\lambda^k_n}}(x)\in O^k_1\setminus B^k_1$, donc $\hat\varphi^k_1(h_{{1/\lambda^k_n}}^k(x))=\hat\varphi_1(h_{{1/\lambda^k_n}}^k(x))$ d'apr\`es le lemme \ref{PoofBBLemma2}, d'o\`u $\hat\varphi^k_1(h_{{1/\lambda^k_n}}^k(x))=H^k_{1/\lambda^k_n}(\hat\varphi_1(x))$ par (\ref{EquaTioN(k)PoofBBLemma1}) et on obtient $\hat\varphi_n^k(x)=\hat\varphi_n(x)$ en utilisant (\ref{EquaTioN(l)PoofBBLemma1})). Pour chaque $n\geq 1$, on peut ainsi d\'efinir $\varphi_n\in C^1({\rm int}(\cup_{i\in I}\overline{V}_i);\RR^3)$ par  
$$
\varphi_n(x):=\left\{
\begin{array}{ll}
\hat\varphi_n^k(x)&\hbox{si }x\in B^k_{\lambda^k_n}\\
\hat\varphi_n(x)&\hbox{si }x\in {\rm int}(\cup_{i\in I}\overline{V}_i)\setminus\cup_{k=1}^qB^k_{\lambda^k_n}.
\end{array}
\right.
$$
(Les fonctions $\varphi_n$ sont bien d\'efinies gr\^ace \`a (\ref{EquaTioN(d)PoofBBLemma1}).) D'apr\`es lemme \ref{PoofBBLemma2}, $\hat\varphi^k_1$ est une $C^1$-immersion, donc pour chaque $k\in\{1,\cdots,q\}$, il existe $\beta_k>0$ tel que $|\partial_1\hat\varphi_1^k(y)\land\partial_2\hat\varphi_1^k(y)|\geq\beta_k$ pour tout $y\in\overline{B}^k_1$. De plus, on a 
\begin{eqnarray}
&&\hbox{pour chaque }k\in\{1,\cdots,q\}\hbox{, chaque }n\geq 1\hbox{ et chaque }x\in \overline{B}^k_{\lambda^k_n}\hbox{, }\label{EquaTioN(+++)PoofBBLemma1}\\
&&\nabla\hat\varphi^k_n(x)=\nabla\hat\varphi_1^k(y)\hbox{ avec }y=h^k_{1/\lambda^k_n}(x)\in\overline{B}^k_1,\nonumber
\end{eqnarray}
(donc $|\partial_1\hat\varphi_n^k(y)\land\partial_2\hat\varphi_n^k(y)|\geq\beta_k$ pour tout $n\geq 1$ et tout $x\in\overline{B}^k_{\lambda^k_n}$). Prenant en compte (\ref{PoofBBLemma1EqUat6}), il suit que 
\begin{equation}\label{BBBTheorPropII}
|\partial_1\varphi_n(x)\land\partial_2\varphi_n(x)|\geq\delta\hbox{ pour tout }n\geq 1\hbox{ et tout }x\in{\rm int}(\cup_{i\in I}\overline{V}_i) 
\end{equation}
avec $\delta:=\min\{\delta_2,\min\{\beta_k:k\in\{1,\cdots,q\}\}$ ($\delta$ ne d\'epend pas de $n$). D'autre part, pour chaque $k\in\{1,\cdots,q\}$, $\sup_{y\in \overline{B}^k_1}(|\hat\varphi^k_1(y)|+|\nabla\hat\varphi^k_1(y)|)<+\infty$, donc 
\begin{eqnarray}
&&\sup_{n\geq 1}\sup_{x\in \overline{B}^k_{\lambda_n^k}}(|\hat\varphi^k_n(x)|+|\nabla\hat\varphi^k_n(x)|)<+\infty\hbox{ pour tout }k\in\{1,\cdots,q\}\label{EquaTioN(++++)PoofBBLemma1}
\end{eqnarray}
gr\^ace \`a (\ref{EquaTioN(+++)PoofBBLemma1}). Utilisant  (\ref{EquaTioN(++++)PoofBBLemma1}), (\ref{EquaTioN(f)PoofBBLemma1}) et (\ref{CvgFirstProofBBBTheo}) on d\'eduit que 
\begin{equation}\label{BBBTheorPropI}
\varphi_n\to\varphi\hbox{ dans }W^{1,p}({\rm int}(\cup_{i\in I}\overline{V}_i);\RR^3).
\end{equation}
On d\'efinit $\{\psi_n\}_{n\geq 1}\subset C^1(\overline{\Sigma};\RR^3)$ par $\psi_n:=\varphi_n\lfloor_{\overline{\Sigma}}$. Alors la suite $\{\psi_n\}_{n\geq 1}$ satisfait (\ref{BB_1}) et (\ref{BB_2}) (puisque trivialement (\ref{BBBTheorPropI}) implique (\ref{BB_1}) et (\ref{BBBTheorPropII}) implique (\ref{BB_2})). 
\end{proof}

\subsubsection{D\'emonstration du lemme \ref{PoofBBLemma1}}

\'Etant donn\'e $(i,j)\in A$, on a 
$$
\varphi(x)=\left\{
\begin{array}{ll}
\xi_i x+a_i&\hbox{si }x\in\overline{V}_i\\
\xi_j x+a_j&\hbox{si }x\in\overline{V}_j.
\end{array}
\right.
$$
avec $\xi_i,\xi_j\in\MM^{3\times 2}$ ($\rank(\xi_i)=\rank(\xi_j)=2$) et $a_i,a_j\in\RR^2$. On note $\theta^{i,j}\in\RR^2$ le milieu de l'ar\^ete $\overline{V}_i\cap\overline{V}_j$.  Alors 
$$
\varphi(x)=\left\{
\begin{array}{ll}
\xi_i x^{i,j}+\xi_i \theta^{i,j}+a_i&\hbox{si }x\in\overline{V}_i\\
\xi_j x^{i,j}+\xi_j \theta^{i,j}+a_j&\hbox{si }x\in\overline{V}_j.
\end{array}
\right.
$$
avec $x^{i,j}:=x-\theta^{i,j}=(x_1^{i,j},x_2^{i,j})$, o\`u $(x^{i,j}_1,x^{i,j}_2)$ sont les coordonn\'ees de $x^{i,j}$ dans la base orthonorm\'ee $(e_1^{i,j},e_2^{i,j})$ avec $e^{i,j}_1\in {\rm vect}(\overline{V}_i\cap\overline{V}_j-\theta^{i,j})$. Comme $\varphi$ est continue, on a $\xi_i 0+\xi_i \theta^{i,j}+a_i=\xi_i \theta^{i,j}+a_i=\xi_j 0+\xi_j \theta^{i,j}+a_j=\xi_j \theta^{i,j}+a_j=:b^{i,j}$. Ainsi 
$$
\varphi(x)=\left\{
\begin{array}{ll}
\xi_i x^{i,j}+b^{i,j}&\hbox{si }x\in\overline{V}_i\\
\xi_j x^{i,j}+b^{i,j}&\hbox{si }x\in\overline{V}_j
\end{array}
\right.
=\left\{
\begin{array}{ll}
x^{i,j}_1\xi_i^1+x^{i,j}_2\xi_i^2+b^{i,j}&\hbox{si }x\in\overline{V}_i\\
x^{i,j}_1\xi_j^1+x^{i,j}_2\xi_j^2+b^{i,j}&\hbox{si }x\in\overline{V}_j
\end{array}
\right.
$$
avec $\xi_i=(\xi_i^1\mid\xi_i^2)$ et $\xi_j=(\xi_j^1\mid\xi_j^2)$. Puisque $\eps e^{i,j}_1+\theta^{i,j}\in\overline{V}_i\cap\overline{V}_j$ pour $\eps>0$ assez petit, utilisant \`a nouveau la continuit\'e de $\varphi$, on d\'eduit que $\eps\xi_i^1+0\xi_i^2+b^{i,j}=\xi_i^1+b^{i,j}=\eps\xi_j^1+0\xi_j^2+b^{i,j}=\eps\xi_j^1+b^{i,j}$, c'est \`a dire, $\xi_i^1=\xi_j^1$. Finalement, on a 
$$
\varphi(x)=\left\{
\begin{array}{ll}
x^{i,j}_1\xi_i^1+x^{i,j}_2\xi_i^2+b^{i,j}&\hbox{si }x\in\overline{V}_i\\
x^{i,j}_1\xi_i^1+x^{i,j}_2\xi_j^2+b^{i,j}&\hbox{si }x\in\overline{V}_j.
\end{array}
\right.
$$
Comme $\varphi$ est localement injective on a n\'ecessairement $\xi^2_j\not\in {\rm vect}(\xi^1_i,\xi^2_i)$. Notons $\underline{s}_{i,j},\overline{s}_{i,j}\in S$ les extr\'emit\'es de l'ar\^ete $\overline{V}_i\cap\overline{V}_j$ et posons $\underline{s}^{i,j}:=\underline{s}_{i,j}-\theta^{i,j}$ et $\overline{s}^{i,j}:=\overline{s}_{i,j}-\theta^{i,j}$ (comme $\underline{s}_{i,j},\overline{s}_{i,j}\in \overline{V}_i\cap\overline{V}_j$ on a $\underline{s}^{i,j}=\underline{s}^{i,j}_1e^{i,j}_1$ et $\overline{s}^{i,j}=\overline{s}^{i,j}_1e^{i,j}_1$). Soit $\{f^{i,j}_n\}_{n\geq 1}\subset C^\infty([\underline{s}^{i,j}_1,\overline{s}^{i,j}_1];\RR)$ satisfaisant les propri\'et\'es \hbox{suivantes :}
\begin{eqnarray}
&&\  \ \ \hbox{pour chaque }n\geq 1,\  f^{i,j}_n\leq f^{i,j}_{n+1},\ 0\leq f^{i,j}_n\leq\hbox{$1\over n$}\hbox{ et }f^{i,j}_n(\underline{s}^{i,j}_1)=f^{i,j}_n(\overline{s}^{i,j}_1)=0, \label{B-B-BLeMMa1}\\ 
&&\ \ \ \hbox{et pour chaque }x_1^{i,j}\in[\underline{s}^{i,j}_1,\overline{s}^{i,j}_1], \lim_{n\to+\infty}f^{i,j}_n(x_1^{i,j})=0\ ;\cr
&&\  \ \ \hbox{il existe }\{\underline{I}_{i,j,n}\}_{n\geq 1}\subset[\underline{s}^{i,j}_1,\overline{s}^{i,j}_1] \hbox{ (resp. }\{\overline{I}_{i,j,n}\}_{n\geq 1}\subset[\underline{s}^{i,j}_1,\overline{s}^{i,j}_1]\hbox{)}\hbox{ tel que :}\label{B-B-BLeMMa2}\\
&&\ \ \ \diamond\hbox{ pour chaque }n\geq 1,\ \underline{I}_{i,j,n}\hbox{ (resp. }\overline{I}_{i,j,n}\hbox{) est un voisinage ouvert de}\cr
&&\quad\ \ \hskip0.7mm \underline{s}^{i,j}_1\hbox{ (resp. }\overline{s}^{i,j}_1\hbox{)},\  \underline{I}_{i,j,1}\supset\cdots\supset\underline{I}_{i,j,n}\supset\cdots\hbox{ et }\lim_{n\to+\infty}|\underline{I}_{i,j,n}|=0\cr
&&\quad\ \ \hskip0.4mm \hbox{(resp. }\overline{I}_{i,j,1}\supset\cdots\supset\overline{I}_{i,j,n}\supset\cdots\hbox{ et }\lim_{n\to+\infty}|\overline{I}_{i,j,n}|=0\hbox{) ;}\cr
&&\ \ \ \diamond\  \hbox{pour chaque }n\geq 1,\  f^{i,j}_n=f^{i,j}_1\hbox{ sur } \underline{I}_{i,j,n}\hbox{ (resp. }\overline{I}_{i,j,n}\hbox{) ;}\cr
&&\ \ \ \diamond\  f^{i,j}_1\hbox{ est affine sur }\underline{I}_{i,j,n}\hbox{ (resp. }\overline{I}_{i,j,n}\hbox{) ;}\cr
&&\ \ \ \hbox{pour chaque }n\geq 1\hbox{ et chaque }x^{i,j}_1\in[\underline{s}^{i,j}_1,\overline{s}^{i,j}_1],\ \big|(f^{i,j}_n)^\prime(x^{i,j}_1)\big|\leq 1\ ; \label{DErivEEBorNEEeq}\\
&&\  \ \ \hbox{pour chaque }n\geq 1,\ U^{i,j}_n\subset{\rm int}(\overline{V}_i\cup\overline{V}_j) \hbox{ avec }\label{B-B-BLeMMa3}\\
&&\ \ \ U^{i,j}_n:=\Big\{(x_1^{i,j},x_2^{i,j})+\theta^{i,j}\in {\rm int}(\overline{V}_i\cap\overline{V}_j):-f^{i,j}_n(x^{i,j}_1)<x^{i,j}_2<f^{i,j}_n(x^{i,j}_1)\Big\}.\nonumber
\end{eqnarray}
De (\ref{B-B-BLeMMa1}) et (\ref{B-B-BLeMMa3}) on d\'eduit (\ref{EquaTioN(c)PoofBBLemma1}) et (\ref{EquaTioN(e)PoofBBLemma1}). Comme $A$ est fini et les $V_i$ sont des triangles, il est possible de choisir les $f^{i,j}_n$ de mani\`ere \`a avoir aussi (\ref{EquaTioN(a)PoofBBLemma1}). 

Consid\'erons $g\in C^{\infty}([-1,1];\RR)$ donn\'ee par 
$$
g(t):=\left\{
\begin{array}{ll}
\exp^{{2(t-1)\over t+1}}&\hbox{si }t\in]-1,1]\\
0&\hbox{si }t=-{1}
\end{array}
\right.
$$
($g\geq 0$ et $\sup_{t\in[-1,1]}g(t)\leq 1$) et d\'efinissons $\varphi^{i,j}_n:\overline{V}_i\cup\overline{V}_j\to\RR^3$ par 
\begin{equation}\label{B-B-BMainFunct}
\varphi^{i,j}_n(x):=\left\{
\begin{array}{ll}
\varphi(x)&\hbox{si }x\in \overline{V}_i\cup\overline{V}_j\setminus U^{i,j}_n\\
\phi^{i,j}_n(x)&\hbox{si }x\in U^{i,j}_n
\end{array}
\right.
\end{equation}
avec $\phi^{i,j}_n:U^{i,j}_n\to\RR^3$ donn\'ee par 
$$
\phi^{i,j}_n(x):=x^{i,j}_1\xi^1_i+f^{i,j}_n(x^{i,j}_1)g\left({x^{i,j}_2\over f^{i,j}_n(x^{i,j}_1)}\right)\xi^2_i-f^{i,j}_n(x^{i,j}_1)g\left({-x^{i,j}_2\over f^{i,j}_n(x^{i,j}_1)}\right)\xi^2_j+b^{i,j}.
$$
Posons $L:=\max\{L_{i,j}:(i,j)\in A\}$, o\`u $L_{i,j}$ est la longueur de l'ar\^ete $\overline{V}_i\cap\overline{V}_j$, $\beta:=\max\{|b^{i,j}|:(i,j)\in A\}$ et $\gamma:=\max\{|\xi_i|,|\xi_j|:(i,j)\in A\}$. Il est facile de voir que :
\begin{eqnarray}
&&\hbox{pour chaque }(i,j)\in A,\hbox{ chaque }n\geq 1 \hbox{ et chaque } \overline{x}\in\partial U^{i,j}_n,\label{FiNProoFEqUAtIoN1}\\
&& \lim_{x\to\overline{x} }\phi^{i,j}_n(x)=\varphi(\overline{x})\ ;\nonumber\\
&&\hbox{pour chaque }(i,j)\in A,\hbox{ chaque }n\geq 1 \hbox{ et chaque } x\in U^{i,j}_n,\label{FiNProoFEqUAtIoN2}\\
&&|\phi^{i,j}_n(x)|\leq (L+2+\beta)\gamma.\nonumber
\end{eqnarray}
De plus, on a :
\begin{eqnarray*}
\diamond\ \partial_1\phi^{i,j}_n(x)&=&\xi^1_i+(f_n^{i,j})^\prime(x_1^{i,j})\left[g\left({x_2^{i,j}\over f^{i,j}_n(x_1^{i,j})}\right)-{x_2^{i,j}\over f^{i,j}_n(x_1^{i,j})}h\left({x_2^{i,j}\over f^{i,j}_n(x_1^{i,j})}\right)\right]\xi^2_i\\
&&-(f_n^{i,j})^\prime(x_1^{i,j})\left[g\left({-x_2^{i,j}\over f^{i,j}_n(x_1^{i,j})}\right)+{x_2^{i,j}\over f^{i,j}_n(x_1^{i,j})}h\left({-x_2^{i,j}\over f^{i,j}_n(x_1^{i,j})}\right)\right]\xi^2_j\ ;\\
\diamond\ \partial_2\phi^{i,j}_n(x)&=& h\left({x_2^{i,j}\over f^{i,j}_n(x_1^{i,j})}\right)\xi^2_i+h\left({-x_2^{i,j}\over f^{i,j}_n(x_1^{i,j})}\right)\xi^2_j
\end{eqnarray*}
avec $h\in C^\infty([-1,1];\RR)$ donn\'ee par 
$$
h(t):=\left\{
\begin{array}{ll}
{4\over (t+1)^2}g(t)&\hbox{si }t\in]-1,1]\\
0&\hbox{si }t=-{1}
\end{array}
\right.
$$
($h\geq 0$ et $\sup_{t\in[-1,1]}h(t)\leq 1$). D'o\`u (utilisant (\ref{DErivEEBorNEEeq})) on voit que :
\begin{eqnarray}
&& \hbox{pour chaque }(i,j)\in A,\hbox{ chaque }n\geq 1 \hbox{ et chaque } \overline{x}\in\partial U^{i,j}_n\setminus\{\underline{s}_{i,j},\overline{s}_{i,j}\},\label{FiNProoFEqUAtIoN4}\\
&& \lim_{x\to\overline{x} }\partial_1\phi^{i,j}_n(x)=\partial_1\varphi(\overline{x})\hbox{ et }\lim_{x\to\overline{x} }\partial_2\phi^{i,j}_n(x)=\partial_2\varphi(\overline{x})\ ;\nonumber\\
&& \hbox{pour chaque }(i,j)\in A,\hbox{ chaque }n\geq 1 \hbox{ et chaque } x\in U^{i,j}_n,\label{FiNProoFEqUAtIoN3}\\
&& |\partial_1\phi^{i,j}_n(x)|\leq 5\gamma\hbox{ et } |\partial_2\phi^{i,j}_n(x)|\leq 2\gamma.\nonumber
\end{eqnarray}
De \eqref{FiNProoFEqUAtIoN1} et \eqref{FiNProoFEqUAtIoN4} (resp. \eqref{FiNProoFEqUAtIoN2} et \eqref{FiNProoFEqUAtIoN3}) on d\'eduit (\ref{EquaTioN(g)PoofBBLemma1}) (resp. \eqref{EquaTioN(h)PoofBBLemma1}).

Posons $M_{i,j}:=(\xi^1_i\mid\xi^2_i\mid\xi^2_j)\in\MM^{3\times 3}$ ($M_{i,j}$ est inversible puisque $\xi^2_j\not\in{\rm vect}(\xi^1_i,\xi^2_i)$). Soit $(e_1,e_2,e_3)$ la base canonique de $\RR^3$ (alors $M_{i,j}^{-1}\xi^1_i=e_1$, $M^{-1}_{i,j}\xi^2_i=e_2$ et  $M^{-1}_{i,j}\xi^2_j=e_3$). Il suit que 
\begin{eqnarray}
&&\hbox{pour chaque }(i,j)\in A,\hbox{ chaque }n\geq 1 \hbox{ et chaque } x\in U^{i,j}_n,\label{FiNProoFEqUAtIoN4}\\
&&\big|M^{-1}_{i,j}\partial_1\phi^{i,j}_n(x)\land M^{-1}_{i,j}\partial_2\phi^{i,j}_n(x)\big|^2\geq h^2\left({x_2^{i,j}\over f^{i,j}_n(x_1^{i,j})}\right)+h^2\left({-x_2^{i,j}\over f^{i,j}_n(x_1^{i,j})}\right).\nonumber
\end{eqnarray}
Posons $U^{i,j,+}_n:=\{x\in U^{i,j}_n:x^{i,j}_2\geq 0\}$ et $U^{i,j,-}_n:=\{x\in U^{i,j}_n:x^{i,j}_2\leq 0\}$ ($U^{i,j}=U^{i,j,+}\cup U^{i,j,-}$). Remarquant que $h(t)\geq h(0)=\exp^{-2}$ pour tout $t\in[0,1]$ et que ${x_2^{i,j}\over f^{i,j}_n(x_1^{i,j})}\in[0,1]$ (resp. ${-x_2^{i,j}\over f^{i,j}_n(x_1^{i,j})}\in[0,1]$) d\`es que $x\in U^{i,j,+}_n$ (resp. $x\in U^{i,j,-}_n$), de \eqref{FiNProoFEqUAtIoN4} on d\'eduit que pour chaque $(i,j)\in A$ et chaque $n\geq 1$, on a :
\begin{itemize}
\item[$\diamond$] $|M^{-1}_{i,j}\partial_1\phi^{i,j}_n(x)\land M^{-1}_{i,j}\partial_2\phi^{i,j}_n(x)|\geq \exp^{-2}$ pour tout $x\in U^{i,j,+}_n$ ;
\item[$\diamond$] $|M^{-1}_{i,j}\partial_1\phi^{i,j}_n(x)\land M^{-1}_{i,j}\partial_2\phi^{i,j}_n(x)|\geq \exp^{-2}$ pour tout $x\in U^{i,j,-}_n$,
\end{itemize}
donc $|M^{-1}_{i,j}\partial_1\phi^{i,j}_n(x)\land M^{-1}_{i,j}\partial_2\phi^{i,j}_n(x)|\geq \exp^{-2}$ pour tout $(i,j)\in A$, tout $n\geq 1$ et tout $x\in U^{i,j}_n$, et \eqref{EquaTioN(i)PoofBBLemma1} suit en utilisant (pour chaque $(i,j)\in A$ et chaque $n\geq 1$) le lemme suivant (en remarquant que $\{(\partial_1\phi^{i,j}_n(x),\partial_2\phi^{i,j}_n(x)):x\in U^{i,j}_n\}\subset K$ avec $M=M_{i,j}$, $c=5|M^{-1}_{i,j}|\gamma$ et $\eta=e^{-2}$).
\begin{lemma}\label{LeMMaFinPrOOfBBB}
Soient $M\in\MM^{3\times 3}$ une matrice inversible et $K\subset\RR^3\times\RR^3$ (le compact) donn\'e par 
$$
K:=\Big\{(u_1,u_2)\in\RR^3\times\RR^3:\max\{|M^{-1}u_1|,|M^{-1}u_2|\}\leq c\hbox{ et } |M^{-1}u_1\land M^{-1}u_2|\geq\eta\Big\}
$$ 
avec $c,\eta>0$. Alors, il existe $\delta>0$ tel que $|u_1\land u_2|\geq\delta$ pour tout $(u_1,u_2)\in K$.
\end{lemma}
\begin{proof}[D\'emonstration du lemme \ref{LeMMaFinPrOOfBBB}]
Sinon, il existe une suite $\{(u^n_1,u^n_2)\}\subset K$ telle que $|u^n_1\land u^n_2|\leq{1\over n}$ pour tout $n\geq 1$. Comme $K$ est compact, on a (\`a une sous-suite pr\`es) $(u^n_1,u^n_2)\to(\overline{u}_1,\overline{u}_2)\in K$, donc $\lim_{n\to+\infty}|u^n_1\land u^n_2|=0=|\overline{u}_1\land\overline{u}_2|$, i.e., les vecteurs $\overline{u}_1$ et $\overline{u}_2$ sont colin\'eaires. Il suit que $|M^{-1}\overline{u}_1\land M^{-1}\overline{u}_2|=0$, ce qui est impossible. 
\end{proof}
Comme $\underline{s}_{i,j}$ et $\overline{s}_{i,j}$ sont des sommets quelconques, on peut noter $s_k$ \`a la place de $\underline{s}_{i,j}$ ou $\overline{s}_{i,j}$ et $I^k_{i,j,n}$ \`a la place de $\underline{I}_{i,j,n}$ ou $\overline{I}_{i,j,n}$. Prenant en compte le deuxi\`eme point de (\ref{B-B-BLeMMa2}) on voit que 
\begin{eqnarray}
&&\hbox{pour chaque }(i,j)\in A,\hbox{ chaque }k\in\{1,\cdots,q\}\hbox{ et chaque }n\geq 1,\label{3BAssERtION1}\\
&&\phi^{i,j}_n=\phi^{i,j}_1\hbox{ sur }{U}_{i,j,n}^k\hbox{avec ${U}_{i,j,n}^k:=\{x\in U^{i,j}_n:x^{i,j}_1\in{I}_{i,j,n}^k\}$.}\nonumber
\end{eqnarray}
De plus, gr\^ace au premier point de (\ref{B-B-BLeMMa2}), on a $U^k_{i,j,1}\supset\cdots\supset U^k_{i,j,n}\supset\cdots$ et $\lim_{n\to+\infty}|U^k_{i,j,n}|=0$. Il suit que l'on peut construire $\{O^k_n\}^{k\in\{1,\cdots,q\}}_{n\geq 1}$ satisfaisant (\ref{EquaTioN(b)PoofBBLemma1}), (\ref{EquaTioN(d)PoofBBLemma1}), (\ref{EquaTioN(f)PoofBBLemma1}) et la condition suivante 
\begin{eqnarray}
&&\hbox{pour chaque }(i,j)\in A,\hbox{ chaque }k\in\{1,\cdots,q\}\hbox{ et chaque }n\geq 1,\label{3BAssERtION2}\\
&&O^k_n\cap U^{i,j}_n=U^k_{i,j,n}.\nonumber 
 \end{eqnarray}
 De (\ref{3BAssERtION1}) et (\ref{3BAssERtION2}) on d\'eduit (\ref{EquaTioN(l)PoofBBLemma1}). Utilisant le troisi\`eme point de (\ref{B-B-BLeMMa2}) on obtient (\ref{EquaTioN(k)PoofBBLemma1}) et  (\ref{EquaTioN(j)PoofBBLemma1}) (puisque $g$ est injective). \hfill$\square$

\subsubsection{D\'emonstration du lemme \ref{PoofBBLemma2}} Soit $k\in\{1,\cdots,q\}$. De (\ref{EquaTioN(j)PoofBBLemma1}) on d\'eduit qu'il existe $\rho>0$ tel que $\overline{D}_\rho(s_k)\subset O^k_1$ et $\hat\varphi_1\lfloor_{\overline{D}_\rho(s_k)}$ est injective, o\`u $\overline{D}_\rho(s_k)$ d\'esigne le disque ferm\'e de centre $s_k$ et de rayon $\rho$. Ainsi $\mathcal{L}:=\hat\varphi_1(\partial \overline{D}_\rho(s_k))$ est une courbe ferm\'ee simple de classe $C^1$ (puisque $\hat\varphi_1$ est aussi une $C^1$-immersion sur $O^k_1\setminus\{s_k\}$). Posons $S_k:=\hat\varphi_1(s_k)$ ($S_k\not\in\mathcal{L}$ puisque $\hat\varphi\lfloor_{\overline{D}_\rho(s_k)}$ est injective) et $\mathcal{C}:=\hat\varphi_1(\overline{D}_\rho(s_k))$. Alors $\mathcal{C}$ est le tronc de c\^one de sommet $S_k$ s'appuyant sur la courbe $\mathcal{L}$, i.e., 
\begin{equation}\label{LiSSagESommETProof1}
\mathcal{C}=\{H^k_\lambda(v):\lambda\in[0,1]\hbox{ et }v\in\mathcal{L}\},
\end{equation}
o\`u $H^k_\lambda:\RR^3\to\RR^3$ d\'esigne l'homoth\'etie de centre $S_k$ et de rapport $\lambda$. (En effet, $y\in\mathcal{C}$ si et seulement si $y=\hat\varphi_1(x)$ avec $x\in\overline{D}_\rho(s_k)$. De plus, on peut \'ecrire $x$ sous la forme $x=h^k_\lambda(u)$ avec $\lambda\in[0,1]$ et $u\in \partial \overline{D}_\rho(s_k)$, o\`u $h^k_\lambda:\RR^2\to\RR^2$ d\'esigne l'homoth\'etie de centre $s_k$ et de rapport $\lambda$, donc $\hat\varphi_1(x)=\hat\varphi_1(h^k_\lambda(u))$. D'autre part, par (\ref{EquaTioN(k)PoofBBLemma1}) on a $\hat\varphi_1(h^k_\lambda(u))=H^k_\lambda(v)$ avec $v=\hat\varphi_1(u)\in\mathcal{L}$, d'o\`u (\ref{LiSSagESommETProof1}).)

Consid\'erons un plan $P$ de $\RR^3$ tel que $\hat{\mathcal{L}}:=P\cap\mathcal{C}$ est une courbe ferm\'ee simple de classe $C^1$ et $U\ni I:=D\cap P$, o\`u $U$ est l'ouvert born\'e du plan $P$ tel que $\partial U=\hat{\mathcal{L}}$ et $D$ est la droite orthogonale \`a $P$ passant par $S_k$. Un tel plan existe car, d'apr\`es (\ref{EquaTioN(l)PoofBBLemma1}), $\mathcal{C}=\varphi\big(\overline{D}_\rho(s_k)\setminus\cup_{(i,j)\in A_k}U^{i,j}_1\big)\cup\big(\cup_{(i,j)\in A_k}\varphi_1^{i,j}(U^{i,j}_1)\big)$ avec $A_k:=\{(i,j)\in A:s_k\hbox{ est une extr\'emit\'e de l'ar\^ete }\overline{V}_i\cap\overline{V}_j\}$. G\'eom\'etriquement, $\mathcal{C}$ est le tronc de c\^one $\varphi\big(\overline{D}_\rho(s_k)\big)$ auquel on a ``arrondi" les ar\^etes et pour lequel il est facile de construire un tel plan puisque $\varphi$ est continue, affine par morceaux et localement injective en $s_k$.  Consid\'erons  le rep\`ere orthonorm\'e  $(I,\eps_1,\eps_2,\eps_3)$ avec $\eps_3\in D$ et d\'efinissons $g:U\to[0,1]$ par 
\[
g(y_1\eps_1+y_2\eps_2)=g(y_1,y_2)=\langle M(y_1,y_2),\eps_3\rangle,
\]
o\`u $M(y_1,y_2)$ est l'unique point d'intersection entre le c\^one $\mathcal{C}$ et la droite parall\`ele \`a D passant par $(y_1,y_2,0)$. Alors $g$ est continue et $C^1$-diff\'erentiable sur $U\setminus\{(0,0)\}$. Posons $B^k_1:=\hat\varphi_1^{-1}({\rm Gr}(g))$ avec ${\rm Gr}(g)$ d\'esignant le graphe de $g$.   Alors $B^k_1$ est un voisinage ouvert de $s_k$ tel que $\overline{B}^k_1\subset \overline{D}_\rho(s_k)\subset O^k_1$. Soient $\hat\varphi_{1,1}$, $\hat\varphi_{1,2}$ et $\hat\varphi_{1,3}$ les composantes de $\hat\varphi_1$ dans la base orthonorm\'ee $(\eps_1,\eps_2,\eps_3)$, i.e., $\hat\varphi_1=\hat\varphi_{1,1}\eps_1+\hat\varphi_{1,2}\eps_2+\hat\varphi_{1,3}\eps_3$. Alors, $\hat\varphi_{1,3}(x)=g(\hat\varphi_{1,1}(x),\hat\varphi_{1,2}(x))$ pour tout $x\in B^k_1$ et $(\hat\varphi_{1,1},\hat\varphi_{1,2})$ est un $C^1$-diff\'eomorphisme de $B^k_1$ sur $U$. Consid\'erons la fonction plateau $p\in C^\infty(\overline{U};[0,1])$ donn\'ee par 
\[
p(y_1,y_2):=\left\{
\begin{array}{ll}
1&\hbox{si }g(y_1,y_2)\in[0,{1\over 2}]\\
0&\hbox{si }g(y_1,y_2)\in[{3\over 4},1]
\end{array}
\right.
\]
et d\'efinissons $\hat\varphi^k_1:O^k_1\to\RR^3$ par 
\[
\hat\varphi^k_1(x):=\left\{
\begin{array}{ll}
\hat\varphi_{1,1}(x)\eps_1+\hat\varphi_{1,2}(x)\eps_2+p(\hat\varphi_{1,1}(x),\hat\varphi_{1,2}(x))\hat\varphi_{1,3}(x)\eps_3&\hbox{si }x\in B^k_1\\
\hat\varphi_1(x)&\hbox{si }x\in O^k_1\setminus B^k_1.
\end{array}
\right.
\]
Alors $\hat\varphi^k_1$ est $C^1$-immersion. \hfill$\square$

\section{Permutation de l'infimum et de l'int\'egrale}

\subsection{Th\'eor\`eme g\'en\'eral}

Dans \cite{oah-jpm03} nous avons \'etudi\'e la permutation de l'infimum et de l'int\'egrale sous la forme g\'en\'erale suivante. Soient $\M$ un espace m\'etrique localement compact qui est $\sigma$-compact, $\X$ un espace de Banach r\'eel s\'eparable, $\mu$ une mesure de Radon positive sur $X$ et $f:\M\times\X\to\overline{\RR}$ une int\'egrande de Carath\'eodory, i.e., $f(x,\zeta)$ est mesurable en $x$ et continue en $\zeta$. Pour $\mathcal{H}\subset L^p_\mu(\M;\X)$, on cherche sous quelles conditions il existe une multifonction mesurable $\Gamma:\M\dto\X$ telle que 
\begin{eqnarray}\label{interchangeproblem}
\inf_{\varphi\in \mathcal{H}}\int_{\M}f(x,\varphi(x))d\mu(x)=\int_{\M}\ \inf_{\zeta\in\Gamma(x)}f(x,\zeta)\ d\mu(x).
\end{eqnarray}
 Posons $\domain:=\big\{\varphi:\M\to\X: \varphi\hbox{ est mesurable et }f(\cdot,\varphi(\cdot))\in L^1_\mu(\M)\big\}$. Nous avons montr\'e le th\'eor\`eme suivant (voir \cite[Theorem 1.1]{oah-jpm03}). 
\begin{theorem}\label{InfIntTH1}
Si ${\cal H}$ est normalement d\'ecomposable, i.e., pour tous $\varphi,\hat \varphi\in{\cal H}$, et tous $K,V\subset\M$ avec $K$ compact, $V$ ouvert et $K\subset V$, il existe une fonction continue $\theta:\M\to[0,1]$ telle que $\theta=0$ dans $\M\setminus V$, $\theta=1$ dans $K$ et $\theta\varphi+(1-\theta)\hat\varphi\in{\cal H}$, et si
\begin{eqnarray}
&&\hbox{il existe }\hat \varphi\in{\cal H}\cap\domain\hbox{ tel que }\func_\varphi\in L^1_{{\rm loc},\mu}(\M)\hbox{ pour tout }\varphi\in{\cal H}\hbox{ avec}\label{HypotHeSePerMut}\\
&& \func_\varphi:\M\to\overline{\RR}\hbox{ d\'efinie par }
\func_{\varphi}(x)=\max\limits_{\alpha\in[0,1]}f\left(x,\alpha \varphi(x)+(1-\alpha)\hat \varphi(x)\right)\hbox{,}\nonumber
\end{eqnarray}
alors {(\ref{interchangeproblem})} a lieu avec $\Gamma$ donn\'e par le $\mu$-essentiel supremum de ${\cal H}$, i.e., la plus petite de toutes les multifonctions mesurables \`a valeurs ferm\'ees $\Lambda:\M\dto\X$ telle que pour tout $\varphi\in{\cal H}$, $\varphi(x)\in\Lambda(x)$ $\mu$-p.p. dans $X$. 
\end{theorem}

\begin{remark}
Le th\'eor\`eme \ref{InfIntTH1} contient le th\'eor\`eme de permutation ``convexe"  de Bouchitt\'e-Valadier (voir \cite[Theorem 1]{bouchitte-valadier88}) dans le cas des int\'egrandes de Carath\'eo-dory (voir \cite[\S 5.1]{oah-jpm03}). Bouchitt\'e et Valadier ont prouv\'e leur th\'eor\`eme de fa\c con directe sans faire de lien avec les th\'eor\`emes de permutation ``non convexe et non normalement d\'ecomposable" de Rockafellar (voir \cite[Theorem 3A]{rockafellar75}) et  Hiai-Umegaki (voir \cite[Theorem 2.2]{hiai-umegaki77}, voir aussi le th\'eor\`eme \ref{PermuThEo1}). (En fait, le th\'eor\`eme de Rockafellar implique le th\'eor\`eme de Hiai-Umegaki, mais il n'y avait pas de lien entre le th\'eor\`eme de Hiai-Umegaki et celui de Bouchitt\'e-Valadier.) Dans \cite{oah-jpm03}, nous avons d\'emontr\'e le th\'eor\`eme \ref{InfIntTH1} (et donc celui de Bouchitt\'e-Valadier) \`a partir des th\'eor\`emes de permutation et de d\'ecomposabilit\'e (voir \cite[Theorem 3.1]{hiai-umegaki77}, voir aussi le th\'eor\`eme \ref{PermuThEo2}) de Hiai-Umegaki. Dans notre contexte, on a donc 
$$
\hbox{th\'eor\`eme R.}\then\hbox{th\'eor\`eme H.-U.}\then \hbox{th\'eor\`eme \ref{InfIntTH1}} \then\hbox{th\'eor\`eme B.-V.} 
 $$
(voir \cite[\S 2]{oah-jpm03}). 
\end{remark}

\begin{remark}\label{NormallyDecomposableSets}
Les espaces $L^p_\mu(X;Y)$, $C(X;Y)$ (espace des fonctions continues de $X$ dans $Y$), $C_{\rm c}(X;Y)$ (espace des fonctions continues \`a support compact de $X$ dans $Y$), $C^k(X;Y)$ (espace des fonctions $C^k$-diff\'erentiables de $X$ dans $Y$) et $C^k_{\rm c}(X;Y)$ (espace des fonctions $C^k$-diff\'erentiables \`a support compact de $X$ dans $Y$) sont normalement d\'ecomposables. Plus g\'en\'eralement, si $\Gamma:X\dto Y$ est une multifonction \`a valeurs convexes et si $\mathcal{E}$ est un ensemble normalement d\'ecomposable de fonctions mesurables de $X$ dans $Y$, alors
$
\{\varphi\in\mathcal{E}:\varphi(x)\in\Gamma(x)\ \mu\hbox{-p.p. dans }X\}
$
est normalement d\'ecomposable. D'autre part, \'etant donn\'e un ensemble $\mathcal{U}$  de fonctions mesurables de $X$ dans $[0,1]$, on dit qu'un ensemble $\mathcal{H}$ de fonctions mesurables de $X$ dans $Y$ est $\mathcal{U}$-d\'ecomposable si $\theta\varphi+(1-\theta)\hat\varphi\in{\cal H}$ pour tous $\varphi,\hat\varphi\in \mathcal{H}$ et tout $\theta\in\mathcal{U}$. Les ensembles $\mathcal{U}$-d\'ecomposables avec $\mathcal{U}$ contenant $C_{\rm c}(X;[0,1])$ ou $C^k_{\rm c}(X;[0,1])$ sont normalement d\'ecomposables (voir \cite[Proposition 3.1(1)]{oah-jpm03}).
\end{remark}

\begin{remark}\label{mu-Essential supremum} On peut repr\'esenter le $\mu$-essentiel supremum $\Gamma:X\dto Y$ d'un ensemble $\mathcal{H}$ de fonctions mesurables de $X$ dans $Y$ comme suit (voir \cite[\S 2.2]{bouchitte-valadier88}). 
\begin{itemize}
\item[$\diamond$] Si $\mathcal{H}\subset L^p_\mu(X;Y)$ alors il existe un ensemble d\'enombrable $\mathcal{D}\subset\mathcal{H}$ tel que $\Gamma(x)={\rm adh}\{\varphi(x):\varphi\in\mathcal{D}\}$ $\mu$-p.p. dans $X$, o\`u ${\rm adh}$ d\'esigne l'adh\'erence dans $Y$.
\item[$\diamond$] Si $\mathcal{H}\subset C_{\rm c}(X;Y)$ alors $\Gamma(x)={\rm adh}\{\varphi(x):\varphi\in\mathcal{H}\}$ $\mu$-p.p. dans $X$.
\end{itemize}
\end{remark}

\begin{proof}[Sch\'ema de d\'emonstration du th\'eor\`eme \ref{InfIntTH1}]
Soient $\Gamma:X\dto Y$ le $\mu$-essentiel supremum de $\mathcal{H}$ et ${\rm adh}_p(\mathcal{H})$ l'adh\'erence de $\mathcal{H}$ dans $L^p_\mu(X;Y)$. On montre d'abord que sous (\ref{HypotHeSePerMut}), on a 
\begin{equation}\label{EQuAlITy1PerMUtAtIOnTheoREM}
\inf_{\varphi\in \mathcal{H}}\int_{\M}f(x,\varphi(x))d\mu(x)=\inf_{\varphi\in {\rm adh}_p(\mathcal{H})}\int_{\M}f(x,\varphi(x))d\mu(x)
\end{equation}
(voir \cite[Proposition 4.2]{oah-jpm03}). On prouve ensuite que puisque $\mathcal{H}$ est normalement d\'ecomposable, ${\rm adh}_p(\mathcal{H})$ est $\mathcal{X}(\Omega)$-d\'ecomposable, i.e., $\theta\varphi+(1-\theta)\hat\varphi\in{\rm adh}_p(\mathcal{H})$ pour tous $\varphi,\hat\varphi\in {\rm adh}_p(\mathcal{H})$ et tout $\theta\in\mathcal{X}(\Omega):=\{1_E:E\subset X,\ E\hbox{ mesurable}\}$ avec $1_E$ d\'esignant la fonction caract\'eristique de $E$
(voir \cite[Proposition 4.1(1)]{oah-jpm03}). On consid\`ere alors les deux th\'eor\`emes suivants d\'emontr\'es par Hiai et Umegaki (voir \cite[Theorem 3.1 et Theorem 2.2]{hiai-umegaki77}).
\begin{theorem}\label{PermuThEo2}\begin{sloppypar}
Un ensemble ferm\'e non vide $\mathcal{L}\subset L^p_\mu(X;Y)$ est $\mathcal{X}(\Omega)$-d\'ecomposa-ble si et seulement si il existe une multifonction mesurable \`a valeurs ferm\'ees non vides $\Lambda:X\dto Y$ telle que $\mathcal{L}=L^p_\mu(X;\Lambda):=\big\{\varphi\in L^p_\mu(X;Y):\varphi(x)\in\Lambda(x)\ \mu\hbox{-p.p. dans }X\big\}$.\end{sloppypar}
\end{theorem}
\begin{theorem}\label{PermuThEo1}
Si $\Lambda:X\dto Y$ est une multifonction mesurable \`a valeurs ferm\'ees non vides alors 
\begin{equation}\label{EqUaLItyPermuThEo1}
\inf_{\varphi\in L^p_\mu(X;\Lambda)}\int_{\M}f(x,\varphi(x))d\mu(x)=\int_{\M}\inf_{\zeta\in\Lambda(x)}f(x,\zeta)d\mu(x).
\end{equation}
\end{theorem}
 Du th\'eor\`eme \ref{PermuThEo2} on d\'eduit que ${\rm adh}_p(\mathcal{H})=L^p_\mu(X;\Lambda)$ avec $\Lambda:X\dto Y$ une multifonction mesurable \`a valeurs ferm\'ees non vides. Comme $\Gamma(x)={\rm adh}\{\varphi(x):\varphi\in\mathcal{D}\}$ $\mu$-p.p. dans $X$ avec $\mathcal{D}$ un sous-ensemble d\'enombrable de $\mathcal{H}$ (voir la remarque \ref{mu-Essential supremum}) et $\mathcal{H}\subset{\rm adh}_p(\mathcal{H})$, on a $\Gamma(x)\subset\Lambda(x)$ $\mu$-p.p. dans $X$, donc $L^p_\mu(X;\Gamma)\subset{\rm adh}_p(\mathcal{H})$. D'autre part, si $\varphi\in{\rm adh}_p(\mathcal{H})$ alors il existe $\{\varphi_n\}_{n\ge 1}\subset\mathcal{H}$ tel que $\varphi_n(x)\to\varphi(x)$ $\mu$-p.p. dans $X$. Or $\varphi_n(x)\in\Gamma(x)$ $\mu$-p.p. dans $X$ pour tout $n\geq 1$ (par d\'efinition du $\mu$-essentiel du supremum), donc $\varphi(x)\in\Gamma(x)$ $\mu$-p.p. dans $X$ puisque $\Gamma$ est \`a valeurs ferm\'ees, d'o\`u ${\rm adh}_p(\mathcal{H})\subset L^p_\mu(X;\Gamma)$. Il suit que ${\rm adh}_p(\mathcal{H})=L^p_\mu(X;\Lambda)=L^p_\mu(X;\Gamma)$. Ainsi, utilisant le th\'eor\`eme \ref{PermuThEo1}, on a (\ref{EqUaLItyPermuThEo1}) avec $L^p_\mu(X;\Lambda)={\rm adh}_p(\mathcal{H})$ et $\Lambda=\Gamma$ que l'on combine avec (\ref{EQuAlITy1PerMUtAtIOnTheoREM}) pour obtenir (\ref{interchangeproblem}).
\end{proof}

\subsection{Une version continue du th\'eor\`eme de Hiai-Umegaki}
Rappelons d'abord la d\'efinition suivante (voir \cite{aubin-frankowska90} pour plus de d\'etails).
\begin{definition}\label{DefMultifonctionSCI}
On dit que  $\Lambda:X\dto Y$  est une multifonction sci si pour chaque ferm\'e $F\subset X$, chaque $x\in X$ et chaque $\{x_n\}_{n\geq 1}\subset X$ tels que $x_n\to x$ et $\Lambda(x_n)\subset F$ pour tout $n\geq 1$, on a $\Lambda(x)\subset F$. 
\end{definition}
Le th\'eor\`eme suivant est une cons\'equence du th\'eor\`eme \ref{InfIntTH1} (voir \cite[Corollary 5.4]{oah-jpm03}).
\begin{theorem}\label{InfIntTH2}
Soit $\Sigma\subset\RR^N$ un ouvert born\'e. Supposons que {:}
\begin{eqnarray}\label{ContInterHyp1}
&&\ \ f:\overline{\Sigma}\times\RR^m\to[0,+\infty]\hbox{ est une int\'egrande de Carath\'eodory}\hbox{ ;}\\
&&\ \ \Gamma:\overline{\Sigma}\dto\RR^m\hbox{ est une multifonction sci \`a valeurs convexes ferm\'ees non vides}\hbox{ ;}\label{ContInterHyp2}\\
&&\ \ \int_\Sigma\max_{\alpha\in[0,1]}f(x,\alpha\varphi(x)+(1-\alpha)\hat\varphi(x))dx<+\infty \hbox{ pour tous }\varphi,\hat\varphi\in C(\overline{\Sigma};\Gamma)\label{ContInterHyp3}
\end{eqnarray}
{(}avec $C(\overline{\Sigma};\Gamma):=\big\{\varphi\in C(\overline{\Sigma};\RR^m):\varphi(x)\in\Gamma(x)\hbox{ p.p. dans }\overline{\Sigma}\big\}$ o\`u $C(\overline{\Sigma};\RR^m)$ d\'esigne l'espace des fonctions continues de $\overline{\Sigma}$ dans $\RR^m${)}. Alors 
\[
\inf_{\varphi\in C(\overline{\Sigma};\Gamma)}\int_\Omega f(x,\varphi(x))dx=\int_{\Sigma}\inf_{\zeta\in\Gamma(x)}f(x,\zeta)dx.
\]
\end{theorem}
\begin{proof}
On applique le th\'eor\`eme \ref{InfIntTH1} avec $X=\overline{\Sigma}$, $Y=\RR^m$, $f:\overline{\Sigma}\times\RR^m\to[0,+\infty]$ (qui est une int\'egrande de Carath\'eodory d'apr\`es \eqref{ContInterHyp1}) et $\mathcal{H}=C(\overline{\Sigma};\Gamma)$ (qui est normalement d\'ecomposable puisque d'apr\`es \eqref{ContInterHyp2} $\Gamma$ est \`a valeurs convexes). Comme (par \eqref{ContInterHyp2}) $\Gamma$ est sci \`a valeurs convexes ferm\'ees non vides, on a $C(\overline{\Sigma};\Gamma)\not=\emptyset$ par le th\'eor\`eme  suivant (de s\'election continue de Michael, voir \cite{michael56}).

\begin{theorem}\label{MichaelSelectionTheorem}Si $\Lambda:X\dto Y$ est une multifonction sci \`a valeurs convexes ferm\'ees non vides
alors $\Lambda(x)=\{\varphi(x):\varphi\in C(X;\Lambda)\}$ pour tout $x\in X$. (En particulier, $C(X;\Lambda)\not=\emptyset$.)
\end{theorem}
D'o\`u \eqref{HypotHeSePerMut}  se d\'eduit de \eqref{ContInterHyp3}. Donc, la seule chose \`a prouver est  que 
\begin{equation}\label{ContinuousHUTheorEMEquatION}
\Gamma(x)=\hat\Gamma(x)\hbox{ p.p. dans }\overline{\Sigma},
\end{equation}
o\`u $\hat\Gamma$ d\'esigne l'essentiel supremum (par rapport \`a la mesure de Lebesgue)  de $C(\overline{\Sigma};\Gamma)$. D'apr\`es le premier point de la remarque \ref{mu-Essential supremum}, il existe un ensemble d\'enombrable $\mathcal{D}\subset C(\overline{\Sigma};\Gamma)$ tel que $\hat\Gamma(x)={\rm adh}\{\varphi(x):\varphi\in\mathcal{D}\}$ p.p. dans $\overline{\Sigma}$. Donc $\hat\Gamma(x)\subset\Gamma(x)$ p.p. dans $\overline{\Sigma}$. D'autre part, (par le th\'eor\`eme \ref{MichaelSelectionTheorem}) on a $\Gamma(x)=\{\varphi(x):\varphi\in C(\overline{\Sigma};\RR^m)\}$ pour tout $x\in\overline{\Sigma}$. De plus, comme $\overline{\Sigma}$ est compact on a $\hat\Gamma(x)={\rm adh}\{\varphi(x):\varphi\in C(\overline{\Sigma};\RR^m)\}$ d'apr\`es la deuxi\`eme point de la remarque \ref{mu-Essential supremum}, donc $\Gamma(x)\subset\hat\Gamma(x)$ p.p. dans $\overline{\Sigma}$, et \eqref{ContinuousHUTheorEMEquatION} suit.
\end{proof}
\begin{remark}
On peut facilement \'etendre le th\'eor\`eme \ref{InfIntTH2} comme suit.

\begin{theorem}
Soient $\Sigma\subset\RR^N$ un ouvert born\'e et $\Lambda:\overline{\Sigma}\dto\RR^m$ une multifonction. Sous (\ref{ContInterHyp1}) s'il existe une suite $\{\Gamma_n\}_{n\geq 1}$ de multifonctions sci de $\overline{\Sigma}$ dans $\RR^m$ \`a valeurs convexes ferm\'ees non vides telle que 
\[
 \Gamma_1(x)\subset\Gamma_2(x)\subset\Gamma_3(x)\subset\cdots\subset\cup_{n\geq 1}\Gamma_n(x)=\Lambda(x)\hbox{ pour tout }x\in\overline{\Sigma}
 \] 
et si pour chaque $n\geq 1$ on a (\ref{ContInterHyp3}) avec $\Gamma=\Gamma_n$ alors 
\[
\inf_{\varphi\in C(\overline{\Sigma};\Lambda)}\int_\Omega f(x,\varphi(x))dx=\int_{\Sigma}\inf_{\zeta\in\Lambda(x)}f(x,\zeta)dx.
\]
\end{theorem}
\end{remark}

\bibliographystyle{acm}
%\bibliography{OAHJPM}

\end{document}